\documentclass[12pt,a4paper]{article}
\usepackage{amsmath,amssymb,amsthm,graphicx,color,bbm}
\usepackage{upgreek}
\usepackage{caption}
\usepackage{subcaption}
\usepackage[left=1in,right=1in,top=1in,bottom=1in]{geometry}
\usepackage{cite}
\usepackage[british]{babel}
\usepackage{hyperref} 
\usepackage{enumitem}
\usepackage{tikz}
\usepackage{mathtools}
\mathtoolsset{showonlyrefs}
 
\usetikzlibrary{calc,patterns,angles,quotes}
\usetikzlibrary{automata, positioning,arrows}
\usetikzlibrary{decorations.markings}
 
\makeatletter
\tikzset{nomorepostaction/.code=\let\tikz@postactions\pgfutil@empty}
\makeatother

\makeatother
\tikzset{every loop/.style={min distance=10mm,looseness=10}}

\hypersetup{
    colorlinks=false,
    pdfborder={0 0 0},
}

\newtheorem{theorem}{Theorem}[section]

\newtheorem{proposition}[theorem]{Proposition}
\newtheorem{lemma}[theorem]{Lemma}
 
\newtheorem*{theorem*}{``Theorem''}

\numberwithin{equation}{section}

\theoremstyle{definition}

\newtheorem{examplex}[theorem]{Example}

\newenvironment{example}
  {\pushQED{\qed}\examplex}
  {\popQED\endexamplex}

\theoremstyle{remark}
\newtheorem{remark}[theorem]{Remark}
\newtheorem{remarks}[theorem]{Remarks}
\newtheorem*{remark*}{Remark}

 \allowdisplaybreaks

\newcommand{\1}[1]{{\mathbbm 1}\mkern-0.7mu{\left\{#1\right\}}}
\newcommand{\2}[1]{{\mathbbm 1}_{#1}}

\newcommand{\R}{\ensuremath{\mathbb{R}}}

\newcommand{\N}{\ensuremath{\mathbb{N}}}
\newcommand{\ZP}{\ensuremath{\mathbb{Z}}_+}
\newcommand{\RP}{\ensuremath{\mathbb{R}}_+}

\DeclareMathOperator{\Exp}{\mathbb{E}}
\let\Pr\relax
\DeclareMathOperator{\Pr}{\mathbb{P}}

\DeclareMathOperator{\sign}{sgn} 

\DeclareMathOperator{\Ima}{ran}

\newcommand{\tv}{\mathrm{TV}}

\newcommand{\eps}{\varepsilon}
\newcommand{\pii}{\uppi}

\newcommand{\rc}{{\mathrm{c}}}
\newcommand{\ud}{{\mathrm d}}

\newcommand{\cB}{{\mathcal B}}

\newcommand{\cF}{{\mathcal F}}

\newcommand{\cK}{{\mathcal K}}
\newcommand{\cL}{{\mathcal L}}

\newcommand{\cP}{{\mathcal P}}

\newcommand{\cX}{{\mathcal X}}

\newcommand{\bbH}{{\mathbb H}}
\newcommand{\bbX}{\mathbb{X}}
  
\newcommand{\bigmid}{\; \bigl| \;}

\newcommand{\as}{~\text{a.s.}}
\newcommand{\meas}{\Gamma}
\newcommand{\abs}[1]{\left\vert #1\right\vert}
\newcommand{\norm}[1]{\left\Vert #1\right\Vert}

\newcommand{\Kb}{\cK} 
\newcommand{\Ksp}{\cK^\circ_{\mathrm{b}}}

\newcommand{\Ks}{\cK_{\mathrm{s}}}
\newcommand{\Kss}{\Ks^\circ}
\newcommand{\Ksr}{\Ks^\flat}

\newcommand{\Sigmas}{\Sigma^\star} 

\newcommand{\gR}{\mathcal{D}_\gamma}
\newcommand{\rot}{\textrm{Rot}} 
\newcommand{\oa}{\beta}
\newcommand{\gammacr}{\gamma_{\mathrm c}} 
\newcommand{\gammacs}{\gamma_{{\mathrm c},0}} 
\newcommand{\gammag}{\gamma_{\mathrm c}} 

\newcommand{\Meass}{\mathcal{M}_{\pm}}
\newcommand{\Cb}{{C}_{\mathrm{b}}}
\newcommand{\Mb}{{M}_{\mathrm{b}}}
\newcommand{\Ts}{T^*}

\newcommand{\te}{\tilde{e}}
\newcommand{\tsigma}{\tilde{\sigma}^2}
\newcommand{\txi}{\tilde{\xi}}
\newcommand{\teta}{\tilde{\eta}}
\newcommand{\tW}{\tilde{W}}
\newcommand{\tX}{\tilde{X}}
\newcommand{\cmu}{\mu^\circ}
\newcommand{\tmu}{\tilde{\mu}}
\newcommand{\tdelta}{\tilde{\delta}}
\newcommand{\tcmu}{\tilde{\mu}^\circ}
\newcommand{\pib}{\varpi}

\newlist{remenumi}{enumerate}{10}
\setlist[remenumi]{leftmargin=0pt, labelindent=\parindent, listparindent=\parindent, labelwidth=0pt, itemindent=!, itemsep=4pt, parsep=0pt, label=(\alph*)}

\makeatletter
\def\namedlabel#1#2{\begingroup  
    (#2)%
    \def\@currentlabel{#2}%
    \phantomsection\label{#1}\endgroup
}
\makeatother

\title{Stochastic billiards with Markovian reflections\\ in generalized parabolic domains}
\author{Conrado da Costa\footnote{Department of Mathematical Sciences, Durham University, Upper Mountjoy Campus, Durham DH1 3LE, UK.
} \and Mikhail V.\ Menshikov\footnotemark[1] \and Andrew R.\ Wade\footnotemark[1]}
\date{11 March 2023}

\begin{document}

\maketitle

\begin{abstract}
  We study recurrence and transience for a particle that moves at
  constant velocity in the interior of an unbounded planar
  domain, with random reflections at the boundary governed by a Markov
  kernel producing outgoing angles from incoming angles.  Our domains
  have a single unbounded direction and sub-linear growth.  We
  characterize recurrence in terms of the reflection kernel and growth
  rate of the domain.  The results are obtained by transforming the
  stochastic billiards model to a Markov chain on a half-strip
  $\RP \times S$ where $S$ is a compact set.  We develop the
  recurrence classification for such processes in the near-critical
  regime in which drifts of the $\RP$ component are of generalized
  Lamperti type, and the $S$ component is asymptotically Markov; this
  extends earlier work that dealt with finite~$S$.
\end{abstract}

\medskip


\medskip


\tableofcontents

\section{Introduction} 
\label{sec:intro}
\subsection{Overview}
\label{sec:overview}

Billiards models arise from study of the dynamics of ideal gas
molecules in containers or from optical reflectors (see
Section~\ref{sec:billiards-motivation} below).  For a
parameter $\gamma \in(0,1)$, define an unbounded \emph{generalized
  parabolic} or \emph{horn-shaped} planar domain $\gR$ by
\begin{equation}\label{eq:region-def} 
  \gR := \left\{ (x,y) \in \RP \times \R \colon  \abs{y} \leq x^\gamma \right\};
\end{equation}
here $\RP := [0,\infty)$. 
Suppose that a point
particle moves at unit speed in $\gR$.  In the interior, the
particle's velocity is constant, so it travels in straight lines, and
it reflects instantaneously and randomly when it hits the boundary.
The reflection is governed by a Markovian kernel $\Kb$, that defines
the outgoing angle distribution for each incoming angle, where both
angles are measured relative to the inwards-pointing normal.  We give
a more formal definition in Section~\ref{sec:billiards}.  See
Figure~\ref{F:region} for a picture.

\begin{figure}[!ht]
    \centering
          \begin{tikzpicture}[domain=0:6, scale = 0.7]
\filldraw (0,0) circle (1.5pt);
\node at (-0.25,0) {$0$};
\draw[black, line width = 0.40mm]   plot[smooth,domain=0:8,samples=500] ({\x},  {(\x)^(1/2)});
\draw[black, line width = 0.40mm]   plot[smooth,domain=0:8,samples=500] ({\x},  {-(\x)^(1/2)});
\draw[black,->,>=stealth,dashed] (0,0) -- (8,0);
\node at (8.3, 0)       {$x$};
\draw[black,->,>=stealth,dashed] (0,-3) -- (0,3);
\node at (0,3.3)       {$y$};
\draw[double] (4.5,1.5) arc (-45:-67.5:1);
\draw (4.24,1) arc (-75:-108:1);
\draw (4,2) -- (7,2.75);
\draw (1,1.25) -- (4,2);
\draw (4,2) -- (4.3,0.8);
\draw (3.8,1.95) -- (3.86,1.71);
\draw (4.06,1.76) -- (3.86,1.71);
\node at (7, 2)       {$y = x^{\gamma}$};
\node at (7, -2)      {$y = -x^{\gamma}$};
\draw[black,->,>=stealth, line width=1.3pt, line cap=round, dash pattern=on 0pt off 1.5\pgflinewidth] (4,2) -- (5,1);
\draw[black,-, line width=1.3pt, line cap=round, dash pattern=on 0pt off 1.5\pgflinewidth] (3,-1) -- (4,2);
\end{tikzpicture}
\label{fig:parabola}
\caption{Part of the region~$\gR$ with $\gamma = 1/2$.
  A section of the particle's trajectory is
  indicated by the dotted line.  It hits the boundary at the incoming
  angle indicated by the single-ruled angle, and exits at the angle
  indicated by the double-ruled angle, whose distribution is
  determined by the incoming angle according to a kernel~$\Kb$.  }
\label{F:region}
\end{figure}
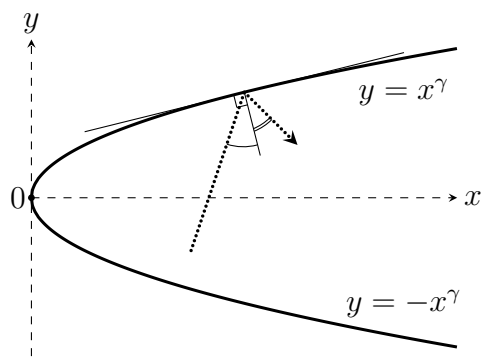

The resulting process is a \emph{stochastic billiards} model with
Markovian reflection; we discuss motivation and related prior work in
Section~\ref{sec:billiards-motivation} below.  At the time of the
$n$th boundary collision, denote by $Z_n \in \RP$ the particle's
horizontal location, and denote  by $\alpha_n \in S:=[-
\frac{\pi}{2}, \frac{\pi}{2} ]$ the incoming angle. 
Then $(Z_n, \alpha_n)$ is a discrete-time Markov chain on $\RP \times
S$.  We aim to establish conditions under which \emph{transience} or
\emph{recurrence} occur, i.e., $\lim_{n \to \infty} Z_n =
\infty$, a.s., or $\liminf_{n \to \infty} Z_n <
\infty$, a.s, respectively.  The classification depends on the
properties of the transition
kernel~$\Kb$ that regulates the reflection at the boundary and on the
growth parameter~$\gamma$.  The case where
$\Kb$ is independent of the incoming angle (i.e., reflections are
i.i.d.)  was considered in~\cite{mvw,MenPopWad16}.

Mild conditions (e.g., appropriate irreducibility) for the reflection kernel $\Kb$ on the compact space $S$ guarantee a unique invariant measure, and we make a density assumption to avoid the possibility of the trajectory of the billiards process hitting the boundary only finitely many times. 
Thus we work in the setting where the reflection kernel~$\Kb$ has a unique
invariant probability density~$\pib$ on~$S$. 
We further assume some mild regularity conditions
that include the reflection angles
being uniformly bounded away from
$\pm\frac{\pi}{2}$.

We will see that the critical regime for this model has
$\int_S \pib (\beta) \tan \beta\, \ud \beta =0$,
corresponding to an asymptotically zero effective drift induced by the
reflections. The main results of this paper on the Markovian billiards model, Theorems~\ref{thm:billiards-strict-lamperti} and~\ref{thm:billiards-general-lamperti}, may be informally
summarized in terms of the following phase transition.

\begin{theorem*}
Suppose that $\int_S \pib (\beta) \tan \beta\, \ud \beta =0$. 
Then, under appropriate conditions, there is a critical value $\gammacr \in [0,1]$, depending on $\Kb$, such that if $\gamma < \gammacr$ the 
stochastic billiards process is recurrent and if $\gamma > \gammacr$ it is transient.
\end{theorem*}

Section~\ref{sec:billiards} below gives 
details of our assumptions, the formal statements of the results, 
and remarks on possible extensions and generalizations,
including the critical case $\gamma =\gammacr$.
The fact that the recurrence phase transition is located in the parameter domain $\gamma \in (0,1)$
is to be expected, under mild conditions,
since in the case $\gamma =0$ (a flat tube) 
the condition $\int_S \pib (\beta) \tan \beta\, \ud \beta =0$ ensures that there is zero averaged drift in the horizontal direction, 
while for $\gamma \geq 1$ (a wedge, or wider), on each boundary reflection the
process will escape to infinity with positive probability.

Under the stronger condition
$\int_S \Kb ( \alpha , \ud \beta) \tan \beta = 0$
for all $\alpha \in S$, which is the case, for example, if every
reflection distribution is symmetric around the normal vector, then~$\gammacr \in (0,1/2)$ and 
the description of $\gammacr$ is rather simple (and constructive),
involving the reflection kernel $\Kb$ only through its stationary density~$\pib$:
see~\eqref{eq:critical-parameter-strict} below.  Otherwise, the
description of~$\gammacr$ exhibits more complex dependence on~$\Kb$:
see~\eqref{eq:critical-parameter-general}.

We analyse the stochastic billiards problem via a transformation 
to a spatially non-homogeneous Markov process
$\xi_n = (X_n, \alpha_n)$ on the \emph{half-strip}
$\Sigma := \RP \times S$, where $X_n := Z_n^{1-\gamma}$. The scaling is
such that the increments of $X_n$ have variance bounded away from $0$
and $\infty$, in which case
the half-strip model falls into a
(generalized) \emph{Lamperti regime} where the effective drift at
$X_n = x$ is of order $1/x$; the terminology is by analogy with
Lamperti's fundamental work on the classification of near-critical
processes on~$\RP$~\cite{lamp}.  Half-strip processes have their own
interest and history: see Section~\ref{sec:strips-motivation} below.

If our billiards model lived in a flat tube ($\gamma =0$)
then~$\alpha_n$ would be itself a Markov chain and the corresponding
strip model would be spatially homogeneous;
the curvature of our $\gamma \in (0,1)$ domain $\gR$ given by~\eqref{eq:region-def}, ensures that~$\alpha_n$ is
only \emph{asymptotically} Markov, in a sense that we make precise
below, since the incoming angle at a reflection is a small perturbation of the preceding outgoing angle.  Processes on the half-strip for which the second co-ordinate
is asymptotically Markov have been investigated
in~\cite{falin,GeoWad14} (the constant drift case) and
\cite{GeoWad14,LoWade17} (the Lamperti case) when $S$ is
\emph{finite}.  Here we extend the classification to the case where
$S$ is a \emph{compact} metric space, such as the interval
$[ - \frac{\pi}{2}, \frac{\pi}{2} ]$.
We use a Lyapunov function approach, similar to~\cite{LoWade17}, but
for existence of suitable Lyapunov functions we must replace
finite-dimensional linear algebra with some theory of linear
operators.

 In this paper, we consider~\eqref{eq:region-def} for $\gamma >0$, so
 $\gR$ is planar and grows asymptotically in the axial direction. Extensions to higher dimensions, and/or domains that
 contract asymptotically, are of interest but need significantly different analysis;
 see Remarks~\ref{rmks:extensions}. Further possible generalizations, of a more technical nature, are discussed in Remarks~\ref{rems:billiards-real} below.
 
 \begin{remarks}
\phantomsection
\label{rmks:extensions} 
\begin{remenumi}
\item
\label{rmks:extensions-a}  
 A natural extension would be to higher dimensions,
 i.e., in~\eqref{eq:region-def} one can take $(x,y) \in \RP \times\R^d$ for general $d \in \N$,
 and read $|y|$ as the Euclidean norm $\| y \|$. Incoming/outgoing `angles'
 are now in the (compact)  hemisphere
  $\bbH^d:= \{ z = (z_1, \ldots , z_{d+1} ) \in \R^{d+1} : \| z \| =1, z_{d+1} \geq 0 \}$.
The structure of the state space $\RP \times \bbH^d$ for the half-strip is unaffected by the increase in dimension, so our method would still be feasible, but to obtain a half-strip model that falls into the class considered here
requires rather strong assumptions. For example, if 
one assumes that the horizontal component of the outgoing angle depends only on the first component of the incoming
 angle, then the model behaves essentially as in the planar case.  Otherwise,
if strong symmetry conditions are not imposed,
 then the recurrence classification would involve a more complex interaction between the geometry and 
 the invariant measure of the angle process on $\bbH^d$, demanding significantly more analysis. Thus we do not pursue higher dimensional extensions in this paper.
\item
\label{rmks:extensions-b} 
Also of interest are `shrinking' domains, as in~\cite{mvw},
in which one expects recurrence, but the questions of interest would be to study stability, i.e.,  positive recurrence, properties of invariant measures, convergence, and ergodicity, for example. Roughly, one would take $\gamma < 0$ in~\eqref{eq:region-def}, but one would need to modify the domain around $x=0$ to ensure that it is smooth and does not produce any pathologies. In the present paper we give results on passage-time moments for the half-strip model (see Section~\ref{sec:strips-results}), but there are two main obstacles to an analogous analysis of the billiards model. These are: (i) from a neighbourhood of the origin, the billiards process can have heavy-tailed increments (cf.~Lemma~\ref{lem:bound-near-apex}) which means that the technical conditions of, e.g., Theorem~\ref{thm:strict-lamperti-moments}, are not satisfied; and~(ii) it seems more natural to ask about ergodic properties of the real-time process, rather than the boundary-collisions process, which would demand a detailed study of the time-change. Such an analysis in the case of i.i.d.~reflections was only partially completed in~\cite{mvw}.
Thus the case $\gamma <0$ also demands a dedicated and thorough analysis that we do not attempt here. 
\end{remenumi}
\end{remarks}
 
The rest of the paper is organized as follows.  In
Sections~\ref{sec:billiards-motivation}
and~\ref{sec:strips-motivation} we discuss motivation and prior
literature for stochastic billiards and half-strip models, to explain
the origin and context of the present paper.  In
Section~\ref{sec:strips} we formulate precisely the half-strip
processes that we study and state our main results on the recurrence
classification.  In Section~\ref{sec:billiards} we do the same for the
stochastic billiards model.  The main structural elements of the
proofs are given in Sections~\ref{sec:strips-proofs}
and~\ref{sec:billiards-proofs}, respectively.  The Appendix collects
some auxiliary results: Section~\ref{sec:fredholm} on results from
functional analysis around the Fredholm alternative theorem for
compact operators, and Section~\ref{sec:semimartingale} on Lyapunov-function criteria for recurrence and transience of processes on half-strips
$\RP \times S$ for compact $S$.

\subsection{Motivation~1: Stochastic billiards}
\label{sec:billiards-motivation}

In the early 1900s, Knudsen undertook a series of experiments studying the flow of
rarefied gases through tubes~\cite{Kn}.
If the mean free path length of the gas is much bigger than the
diameter of the tube, then collisions between gas particles are much rarer than collisions of
particles with the tube boundary, and the bulk behaviour is described via single-particle dynamics.
This \emph{Knudsen regime} of ideal gas dynamics leads to the study of billiards processes,
in which a particle moves with constant velocity until it hits the boundary. Similar processes are also naturally
motivated from optics.

Deterministic reflection leads to classical billiards
models~\cite{tabachnikov}.  The presence of microscopic irregularities
in the domain boundary  (its `microgeometry') motivates considering \emph{random reflections}
and hence \emph{stochastic billiards}: what appears to be a single
reflection at the boundary is comprised of a rapid sequence of
reflections whose cumulative effect is essentially
random~\cite{feres,fy}.  On the basis of his ideal gas experiments,
Knudsen argued for i.i.d.~reflections according to a cosine law; in optics, the
same reflection law  is known as the \emph{Lambertian law}.

Stochastic billiards with i.i.d.~reflections have received much
attention, including~\cite{bg,lr,cpsv,evans} for bounded domains
and~\cite{mvw,MenPopWad16,BurTad} for unbounded domains.  For bounded
domains, stochastic billiards with i.i.d.~reflections are related to
`shake-and-bake' algorithms for sampling uniformly from the
boundary~\cite{dv}. Mathematical results have supported the belief that the Lambertian law is
the most natural law in the case where reflections are independent of
the angle of incidence~\cite{lr,abs}, and 
stochastic billiards with the Lambertian reflection law have received
particular attention. For example, 
in~\cite{BurTad}, the distribution of the exit angle for a
  Lambertian process in a half-infinite tube with an aperture is studied, and in~\cite{bg} the authors prove a scaling limit result for a  
  Lambertian process in a thin annulus.

  As described above, a central motivation for stochastic reflections is the disordered microgeometry of reflectors. However, examining this assumption leads to the conclusion that trajectories at different incoming angles are likely to interact with the same microgeometry in different ways, as described, for example, in~\cite{feres,fy}. Thus there are physical arguments to propose a \emph{Markovian} reflection law, where
the incoming angle is important for determining the reflections; these arguments 
can be made in both the ideal gas and optical settings.

In the probability literature, the study of billiards with Markovian
reflection laws is in the early stages: we are aware only of recent work for
one-dimensional intervals in which the speed (and not just the
direction) may change on each reflection~\cite{bbg};
see Remarks~\ref{rems:billiards-real}\ref{rems:billiards-real-d}
for how our results can be extended to incorporate varying speeds. 
One motivation
for the present paper is to study the probabilistic behaviour of
 Markovian billiards in unbounded, multidimensional
domains.  In this respect, the present paper can be seen as an
extension of the model of~\cite{mvw} from~i.i.d.~to Markovian
reflections.  We focus on the two-dimensional case, the minimal
setting that displays the phenomena we are interested in; see Remarks~\ref{rmks:extensions}\ref{rmks:extensions-a} for some comments on extensions to higher dimensions.

\subsection{Motivation~2: Random walks on half-strips}
\label{sec:strips-motivation}

Let $\xi_n$ be a time-homogeneous, discrete-time Markov chain on state space $\bbX \times S$, and
write $\xi_n = (X_n, \eta_n)$ in coordinates, with $X_n \in \bbX$ and
$\eta_n \in S$.  If the law of $(X_{n+1} -X_n, \eta_{n+1})$ depends
only on~$\eta_n$ (call this assumption \emph{homogeneity}), then
$\xi_n$ is a \emph{Markov random walk}, $\eta_n$ is itself Markov, and
$X_n$ can be represented as an additive functional of the Markov chain
$(X_{n} - X_{n-1}, \eta_n)$.  Under the most common assumptions, $\eta_n$ is
ergodic with a unique stationary distribution~$\pii$.  See
e.g.~\cite{alsmeyer} for a general view of such processes, which arise
in many applications, such as:
\begin{itemize}
\item Queueing, where e.g.~$\bbX = \ZP^d$ is a space of queue-lengths
  and $S$ is a set of service regimes~\cite{neuts}.
\item Random walks with momentum, short memory, or internal degrees of
  freedom, where e.g.~$\bbX = \R^d$ and $S$ is a set of internal
  states for the particle~\cite{ks}.
\item Regime-switching processes in mathematical finance, where
  e.g.~$\bbX = \RP^d$ is a space of prices or interest rates and $S$
  is a set of states of the market~\cite{hamilton}.
\end{itemize}
In practice, these may be~\emph{hidden Markov} models in the sense
that one may not be able to observe~$\eta_n$, only~$X_n$.
 
For concreteness, take $\bbX = \RP$. Then $\RP \times S$ is a
\emph{half-strip} and study of the case of \emph{finite}~$S$ is
classical~\cite{malyshev,FayMalMen95}.  To go deeper, it is natural to
relax the homogeneity assumption, and hence go beyond the Markov
random walk case.  To probe the recurrence/transience phase transition
for the half-strip model, for example, analogy with classical work of
Lamperti~\cite{lamp} suggests that the law of $X_{n+1} -X_n$ should
also depend on $X_n$, and not just $\eta_n$.  Once one admits this
generalization, it is often too restrictive to maintain the Markov
assumption on $\eta_n$: in the presence of non-trivial dependence
between $X_n$ and $\eta_n$, a perturbation of the homogeneous
situation to provide the necessary inhomogeneity for $X_n$ will also
tend to introduce $X_n$-dependence for~$\eta_n$. We refer
to~\cite{GeoWad14,LoWade17} for some examples.  However, progress can
be made if we replace the homogeneity assumption by an
\emph{asymptotic} Markov assumption on~$\eta_n$ and some asymptotic
regularity on the drifts of~$X_n$, both assumptions in the case of large~$X_n$. This
framework is the subject of~\cite{GeoWad14,LoWade17} for the case
where~$S$ is \emph{finite}. The present paper extends this to the case
where~$S$ is a compact metric space.

We emphasize that the application of the half-strip framework to the
Markovian billiards model demands that $\eta_n$ (which will be an
angle in the billiards context) is only asymptotically Markov, so we
are outside the Markov random walk setting. Moreover, the
  reflection rules on a continuous curved surface with inward normal
  vectors in $[-\frac{\pi}{2},\frac{\pi}{2}]$ leads us to consider
   uncountable compact sets $S$. Thus, we need
  to go beyond the finite-$S$ setting of~\cite{GeoWad14,LoWade17}.
In this respect, the present paper can also be seen as an extension of
previous work on half-strips, and is of parallel interest due to
the broad range of applicability of such models: our application
to the stochastic billiards model is one example.

\section{Markov chains on a half-strip}
\label{sec:strips}

\subsection{Asymptotic Markovianity}
\label{sec:absmodel}

We study our stochastic billiards model by a reduction to a Markov
chain on a half-strip~$\RP \times S$. Half-strip models have
their own independent motivation, as described in
Section~\ref{sec:strips-motivation}.  In this section we present our
results on near-critical half-strip models satisfying appropriate
assumptions. The set $S$ will be a compact metric space; in our billiards application, $S$ will be a real interval.  In the somewhat simpler special case where $S$ is
finite, our assumptions align closely with those
of~\cite{GeoWad14,LoWade17}.

For a metric space $(H,d_H)$  with Borel sets~$\cB (H)$, denote by $\cP(H)$ the set of probability measures on~$(H,\cB(H))$.
Recall that a function $\cK : H \times \cB(H) \to [0,1]$ is a \emph{Markov kernel} on~$H$ if
(i) $\cK ( x, \, \cdot \,) \in \cP(H)$ for all $x \in H$, (ii) $x \mapsto \cK (x, A)$
is Borel measurable for each $A \in \cB(H)$, and (iii) $\cK (x, H) =1$ for all $x \in H$.  

To describe our model, fix $(S,d_S)$ a compact metric space with Borel sets $\cB(S)$. 
We denote by $\Sigma := \RP \times S$, our \emph{half-strip}, whose Borel sets $\cB(\Sigma)$ form the product $\sigma$-algebra. 
Suppose that we have a probability space $(\Omega, \cF, \Pr)$ on which there is a filtration $(\cF_n, n \in \ZP)$ 
and an adapted process
$\xi = (\xi_n, n \in \ZP )$ taking values in~$\Sigma$,
with initial state $\xi_0 = (x_0, u_0) \in \Sigma$ deterministic (but arbitrary).
We assume that $\xi$ is a time-homogeneous Markov process  
with Markov kernel~$\Ks$ (`s' for `strip') on $\Sigma$,
so that for all~$A \in \cB(\Sigma)$ and all $n \in \ZP$,
\begin{equation}
\label{eq:strips-kernel}
  \Pr ( \xi_{n+1} \in A \mid \cF_{n} ) =    \Pr ( \xi_{n+1} \in A \mid  \xi_{n} ) = \Ks ( \xi_n ,  A) , \as
\end{equation}

In coordinates, we write $\xi_n = (X_n,\eta_n)$ for $X_n \in \RP$ and $\eta_n \in S$. We will assume
the following basic conditions.

\begin{description}
\item[\namedlabel{ass:non-confinement}{N}] Suppose that $\xi$ is \emph{non-confined}: $\Pr ( \limsup_{n \to \infty} X_n = \infty ) = 1$.
\item[\namedlabel{ass:moments}{B$_{p,q}$}] Suppose that for constants $x_B \in \RP$, $p, q> 0$, and $B_p, B_q <\infty$, 
  \begin{align}
	\label{eq:lamperti-p-moments}
  \Exp \bigl[ \abs{X_{n+1} - X_n}^p \!\bigmid \cF_n \bigr] & \leq B_p, \text{ on } \{ X_n \geq x_B \}; \\
	\label{eq:q-bound}
 \Exp \bigl[ | X_{n+1} |^q \!\bigmid \cF_n \bigr] & \leq B_q, \text{ on } \{ X_n < x_B \} .
  \end{align}
\end{description}

The non-confinement condition~\eqref{ass:non-confinement}
follows from suitable irreducibility or non-degeneracy assumptions (see e.g.~\cite[\S 3.3]{MenPopWad16}).
Condition~\eqref{ass:moments} includes boundedness of $p$th moments in
the $\RP$~coordinate for $X_n\geq x_B$.  While the simplest case is
when $x_B=0$ and~\eqref{eq:lamperti-p-moments} holds everywhere, it is important for our
application to stochastic billiards to permit the case
where~\eqref{eq:lamperti-p-moments} holds on $\{ X_n \geq x_B\}$, and elsewhere 
demand only~\eqref{eq:q-bound} for some $q \in (0,p)$.

We next formulate a condition that says~$\eta_n$ is
\emph{asymptotically Markovian} for large~$X_n$.  This will entail a
limiting kernel on~$S$. Recall that a probability
measure~$\nu \in \cP(S)$ is \emph{invariant} for a Markov kernel~$\cK$
on $(S,d_S)$~if
\begin{equation}
\label{eq:pi-invariant}
\nu (B) = \int_S \nu ( \ud x) \cK ( x, B) , ~ \text{for all}~ B \in \cB(S) .
\end{equation}
Write $\norm{ \,\cdot \,}_\tv$ for the total variation norm, so that
$d_\tv ( \mu, \nu) := \frac{1}{2} \| \mu - \nu \|_\tv$ defines the
total variation metric on~$\cP(S)$.

\begin{description}
 \item[\namedlabel{ass:kernel}{K}]
Suppose that the  Markov kernel $\cK : S \times \cB(S) \to [0,1]$ satisfies the following.
  \begin{enumerate}[label=(\roman*)]
\item
\label{ass:kernel-i} 
~There is a unique solution $\nu = \pii$ to~\eqref{eq:pi-invariant} over~$\nu \in \cP(S)$.
\item
\label{ass:kernel-ii} 
The function~$u \mapsto \cK (u, \, \cdot \,)$
is continuous from $(S,d_S)$ to $(\cP(S),d_\tv)$.
\end{enumerate}
\end{description}
Assumption~\eqref{ass:kernel}\ref{ass:kernel-ii} is a strong version of the Feller property and  guarantees certain
analytic properties of the operator associated with~$\cK$: see Section~\ref{sec:fredholm} below.
To state the asymptotic Markovianity condition,
define for $(x,u) \in \Sigma$ and $B \in \cB (S)$,
\begin{equation}
\label{eq:pseudo-kernel}
\Kss ( x, u , B ) := \Ks ( x, u , \RP\! \times \! B ) ,
\end{equation}
where~$\Ks$ is the kernel from~\eqref{eq:strips-kernel}, and $\Pr ( \eta_{n+1} \in B \mid \cF_n ) = \Kss ( X_n, \eta_n, B)$, a.s.
There are two versions of the asymptotic Markovianity condition, 
the basic~\eqref{ass:asymptotically-markov} and the stronger~\eqref{ass:asymptotically-markov+}; 
which one we will need will depend on the other conditions that we impose.
Let $\Meass (S)$ denote the set of finite signed  measures on~$S$.
In~\eqref{eq:asymptotically-markov} and~\eqref{eq:asymptotically-markov+}, $\cK$
is the kernel from~\eqref{ass:kernel}.

\begin{description}
 \item[\namedlabel{ass:asymptotically-markov}{M}]
Suppose that
  \begin{equation}\label{eq:asymptotically-markov}
  \lim_{x\to \infty}\sup_{u \in S} \norm{\Kss ( x, u, \, \cdot \, )  - \cK ( u, \, \cdot \, )}_\tv = 0.
  \end{equation}
 \item[\namedlabel{ass:asymptotically-markov+}{M$_{+}$}] 
Suppose that there is a continuous $\meas : (S, d_S) \to ( \Meass(S), d_\tv)$ 
such that 
   \begin{equation}\label{eq:asymptotically-markov+}
 \sup_{u \in S}  \norm{\Kss ( x, u, \, \cdot \, ) -  \cK ( u, \, \cdot \, ) -  x^{-1}   \meas_u }_{\tv} = o(x^{-1}), \text{ as } x \to \infty. 
   \end{equation}
\end{description}
Any $\meas$ in~\eqref{ass:asymptotically-markov+} must have $\meas_u (S) = 0$.
 In~\eqref{eq:asymptotically-markov+}
  and subsequently, we use the standard Landau $O, o$ notation: for $f : (0,\infty) \to (0,\infty)$, 
	we write $g(x) = O (f(x))$ to mean that there exist $C, x' \in \RP$ such that
	$| g(x) | \leq C f(x)$ for all $x \geq x'$, and we write
	$g(x) = o (f(x))$ to mean that for every $\eps>0$, there exists $x' \in \RP$ such that
	$| g(x) | \leq \eps f(x)$ for all $x \geq x'$.

\subsection{Lamperti regimes and recurrence classification}
\label{sec:strips-results}

Classical work of Lamperti~\cite{lamp} gives sufficient conditions for recurrence and transience of Markov processes
on $\RP$ in terms of (the first two) increment moment functions: see~\cite[Ch.~3]{MenPopWad16} for a survey of such results.
We develop here the analogous theory for the half-strip model satisfying the assumptions of Section~\ref{sec:absmodel}.

For $(x,u) \in \Sigma$ and $R \in \cB(\RP)$, define $\Ksr (x,u,R) := \Ks ( x, u , R \times S )$.
If~\eqref{ass:moments}
holds for $p \geq k \in \N$ and $x_B \in \RP$,
then for $x \geq x_B$, $u \in S$, define
\begin{equation}
\label{eq:lamperti-mu-def}
\mu_k (x, u) := \int_{\RP} (y-x)^k \Ksr (x, u, \ud y ) ,
\end{equation}
so $\Exp [ (X_{n+1} - X_n)^k \mid \cF_n ] = \mu_k ( \xi_n )$, on $\{ X_n \geq x_B \}$.
For $r \in \RP$, define the \emph{passage time}
\begin{equation}
    \label{eq:def-passage-time}
\tau_r := \min \{ n \in \ZP : X_n \leq r\}, \end{equation}
with the usual convention $\min \emptyset := \infty$. 
In this section we seek to classify the asymptotic behaviour of~$\xi$ using the asymptotic properties of $\mu_1$ and $\mu_2$.

We say $\xi$ is \emph{transient} if
$\lim_{n \to \infty} X_n = \infty$, a.s., \emph{recurrent} if there
exists $r_0 \in \RP$ such that $\liminf_{n \to \infty} X_n \leq r_0$,
a.s., and \emph{positive recurrent} if there exists $r_1 \in \RP$ such
that $\Exp \tau_r < \infty$ for all $r \geq r_1$. If for every $r \in \RP$ there exists
 $r_1 > r$ such
that $\Exp \tau_r = \infty$ whenever $x_0 > r_1$ (recall $X_0=x_0$ is deterministic, but arbitrary), we say
  the process is \emph{null recurrent}.  Under suitable
irreducibility assumptions these are essentially equivalent to other standard
definitions (see e.g.~Chapter~10 of~\cite{dmps}). Let $\Cb(S)$ denote
the continuous (hence bounded) real-valued functions on $S$, and
$\Cb^+(S)$ those that are non-negative.

\begin{proposition}
\label{prop:strips-crude}
Suppose that~\eqref{ass:non-confinement}, \eqref{ass:kernel}, and~\eqref{ass:asymptotically-markov} hold,
and that~\eqref{ass:moments} holds with~$p>1$ and $q>0$. Suppose also that there exists $d \in \Cb(S)$ such that $\mu_1$ defined by~\eqref{eq:lamperti-mu-def} satisfies $\lim_{x \to \infty} \sup_{u \in S} | \mu_1 (x,u) - d_u | = 0$. 
 Set $\delta := \int_S d_u \pii (\ud u)$. Then  $\xi$ is transient if $\delta >0$, and  recurrent if $\delta <0$. If, in addition,
$q \geq 1$ in~\eqref{ass:moments}, then   $\xi$ is positive recurrent if $\delta <0$.
\end{proposition}

In the special case where $S$ is finite,
Proposition~\ref{prop:strips-crude} was established on $\ZP \times S$
as Theorem~2.4 in~\cite{GeoWad14}; see also Theorem~2.1
in~\cite{LoWade17}.  We omit the proof of
Proposition~\ref{prop:strips-crude}, as it is similar to, but simpler
than, those of the subsequent results in this section. A proof may
proceed using appropriate Lyapunov
functions~$f(x,u) = x^\nu+ \nu x^{\nu-1} \varphi(u)$ similarly to
Section~4.2.1 of~\cite{Lo}, but, for the existence of an
appropriate~$\varphi$, replacing the finite-dimensional Fredholm
alternative with the operator version described in
Section~\ref{sec:fredholm}.

The case where $\delta = 0$ in Proposition~\ref{prop:strips-crude} cannot be classified without further assumptions.
We move into the Lamperti  setting, where
the critical case has $\mu_1$ of order $1/x$ (in this context,
after the drifts have been `averaged' against $\pii$)
and $\mu_2$ also comes into play
as long as we have $p>2$ in~\eqref{ass:moments}.
The following are the assumptions we will need.

\begin{description} 
\item[\namedlabel{ass:lamperti}{L}]
Suppose that there exist $d, e, \sigma^2 \in \Cb (S)$ such that, as $x \to \infty$, 
\begin{equation}\label{eq:generalized-lamperti-regime}
  \sup_{u \in S}\abs{ \mu_1 (x, u)  - \left( d_u +\frac{e_u}{x} \right) }  = o(x^{-1}),  ~\text{and}~ 
  \sup_{u \in S}\abs{ \mu_2 (x, u) - \sigma^2_u } = o(1) .
\end{equation}
Moreover, if $\pii$ is as defined in~\eqref{ass:kernel}, suppose that 
\begin{equation}
  \label{eq:lamperti-condition}
  \int_S d_u\, \pii ( \ud u ) = 0.
\end{equation}
\end{description}
	
We describe assumption~\eqref{ass:lamperti} as $\xi$ being in the \emph{Lamperti regime}.
As {mentioned} above, if~\eqref{eq:lamperti-condition} does not hold,
the behaviour is simpler (cf.~Proposition~\ref{prop:strips-crude}),
while if the $1/x$ term in~\eqref{eq:generalized-lamperti-regime}
is replaced by $1/x^\beta$, $\beta \in (0,1)$,
the behaviour is again less critical, in that the phase transition is driven by the sign of the effective drift alone (cf.~the case of $\RP$
as described in Chapter~3 of \cite{MenPopWad16}).
Hence the Lamperti regime is the natural one in which to probe the recurrence phase transition;
it is also the regime that emerges from our stochastic billiards application. 
Sufficient for~\eqref{eq:lamperti-condition} is that $d_u = 0$ for \emph{all}~$u$;
this case has a special place in the theory
and we refer to it as the \emph{strict Lamperti regime}:

\begin{description} 
\item[\namedlabel{ass:strict-lamperti}{L$_0$}]
Suppose that~\eqref{ass:lamperti} holds with $d_u = 0$ for all $u \in S$.
\end{description}

In the strict Lamperti regime, the recurrence classification depends
on the values of 
\begin{equation}\label{eq:deltadef}
    \delta_\theta : = \int_S  (2 e_u + (2\theta - 1) \sigma^2_u ) \,\pii ( \ud u ),
    \end{equation}
where $\theta \in \R$.
Note that if $\theta < \theta'$, then $\delta_{\theta}, \delta_{\theta'} \in \R$ satisfy $\delta_{\theta} \leq \delta_{\theta'}$, with equality if and only if $\sigma^2$ is identically $0$.
The next theorem presents the classification. In the case where $S$ is finite,
Theorem~\ref{thm:strips-strict-lamperti} is essentially Theorem~2.5
of~\cite{GeoWad14} (see also Theorem~2.2 of~\cite{LoWade17}).

\begin{theorem}
\label{thm:strips-strict-lamperti}
Suppose that~\eqref{ass:non-confinement}, \eqref{ass:kernel},
and~\eqref{ass:asymptotically-markov} hold, and
that~\eqref{ass:moments} holds with~$p>2$ and $q>0$. Suppose also
that~\eqref{ass:strict-lamperti} holds.  Then the following
classification applies.
  \begin{enumerate}[label=(\alph*)]
  \item
       \label{thm:strips-strict-lamperti-a}
  The process $\xi$ is transient if $\delta_0 >0$ and recurrent
    if $\delta_0 <0$.
   \item
   \label{thm:strips-strict-lamperti-b}
   If, moreover, $q \geq 2$, then $\xi$
    is positive recurrent if $\delta_1 <0$, while $\xi$ is null recurrent if  $\delta_0 <0 < \delta_1$.
  \end{enumerate}
\end{theorem}

The next theorem presents a refinement of the classification into positive/null recurrence,
via quantitative
  information on the moments of the passage times $\tau_r$ as defined at~\eqref{eq:def-passage-time}. 
  In the case of finite~$S$, analogous results are Theorems~2.3 and~2.4 of~\cite{LoWade17}.

\begin{theorem} 
\label{thm:strict-lamperti-moments}
Suppose
    that~\eqref{ass:non-confinement}, \eqref{ass:kernel},
    and~\eqref{ass:asymptotically-markov} hold, and
    that~\eqref{ass:moments} holds with~$p >2$ and $q \geq 2$. Suppose also
    that~\eqref{ass:strict-lamperti} holds. 
    Define $\delta_\theta$ as at~\eqref{eq:deltadef}. 
  \begin{enumerate}[label=(\alph*)]
    \item 
    \label{thm:strict-lamperti-moments-a}
    If $\delta_\theta<0$ for some $\theta>0$, then for any $s \in [0, \theta \wedge p/2 \wedge q/2 )$
    there exists $r_1 \in \RP$ for which $\Exp[ \tau_r ^s ] <\infty$ for all $r \geq r_1$.
\item
\label{thm:strict-lamperti-moments-b}
 If $\delta_{\theta}>0$ for some $\theta\in (0,p/2 \wedge q/2]$, then for
 every $s > \theta$ and every $r \in \RP$, there exists $r_1 \in (r,\infty)$ for which $\Exp[ \tau_r ^s ] = \infty$ provided $X_0 = x_0$ satisfies $x_0 > r_1$.
 \end{enumerate}
\end{theorem}

\begin{remark}
\label{rem:null-critical-equality}
  Theorem~\ref{thm:strips-strict-lamperti}\ref{thm:strips-strict-lamperti-b} is the special case $\theta = 1$
of Theorem~\ref{thm:strict-lamperti-moments};
the case $q \equiv p$ will suffice for many applications (equivalently, $x_B=0$
in~\eqref{ass:moments}).
With regards to the boundary cases in Theorems~\ref{thm:strips-strict-lamperti} and~\ref{thm:strict-lamperti-moments}, we anticipate, in line with~\cite{GeoWad14,LoWade17}, that under slightly
  stronger convergence rate assumptions in~\eqref{eq:asymptotically-markov}
  and~\eqref{eq:generalized-lamperti-regime}, the cases
  $\delta_0 = 0$ and $\delta_1 =0$ are null recurrent,
  while if $\delta_\theta = 0$ for $\theta >0$, then $\Exp [ \tau_r^{\theta} ] = \infty$.
We believe that the approach of the present paper could be extended to prove this,
but one would need a finer Lyapunov function (e.g., with logarithmic corrections, as in~\cite[\S 3.4]{MenPopWad16}) and additional technical work.
  \end{remark}

We need one further assumption to give a classification
under~\eqref{ass:lamperti}.  By disintegration~\cite[Thm.~6.4,
p.~108]{kall}, one has the representation
\begin{equation}
\label{eq:mu-disintegration}
\mu_1 (x, u) = \int_S   \Kss ( x, u , \ud v ) \cmu_1 (x, u, v) ,
\end{equation}
where $\cmu_1 : \Sigma \times S \to \R$ is measurable, essentially
unique, and can be expressed via regular conditional distributions:
see Section~\ref{sec:general-lamperti-proofs}.  Let $\Mb(S)$ denote
the bounded measurable real-valued functions on~$S$ with the uniform
metric $d_\infty(f,g) := \sup_{u \in S} | f(u) - g(u)|$.

\begin{description} 
\item[\namedlabel{ass:detail-drifts}{D}] Suppose that there exist
  $\lambda_u \in \Mb(S)$ for every $u \in S$ such that
  $u \mapsto \lambda_u$ is continuous from $(S, d_S)$ to
  $( \Mb(S) , d_\infty )$, and
\[
  \lim_{x \to \infty} \sup_{u, v \in S} \left| \cmu_1 (x, u , v) - \lambda_{u} (v) \right| = 0 .
\]
\end{description}

Let $\cK^n$ denote
the $n$-fold convolution of $\cK$, i.e.,
$\cK^n ( u, B) := \int_S \cK (u , \ud v ) \cK^{n-1} ( v , B)$ for
$n \in \N$, with $\cK^0 ( u, B ) := \1 { u \in B }$. The next theorem is our classification in the Lamperti regime. The result is of a similar form to  Theorem~\ref{thm:strips-strict-lamperti}, but the role
of $\delta_\theta$ defined by~\eqref{eq:deltadef} there 
is taken by $\tdelta_\theta$ defined in~\eqref{eq:deltadef-general}; now $\tdelta_\theta$
is less explicit due to the presence of the function~$\psi$ (see Remarks~\ref{rmks:psi}).

\begin{theorem}
\label{thm:strips-general-lamperti}
Suppose that~\eqref{ass:non-confinement}, \eqref{ass:kernel},
and~\eqref{ass:asymptotically-markov+} hold, and
that~\eqref{ass:moments} holds with~$p>2$ and $q>0$.  Suppose also
that~\eqref{ass:lamperti} and~\eqref{ass:detail-drifts} hold.  Then
there exists $\psi \in \Cb(S)$ (unique up to translation) with the
property $\int_S ( \psi (u) - \psi (v) ) \cK ( u, \ud v ) = d_u$ for
all $u \in S$. For $\theta \in \R$, define
\begin{align}
\label{eq:deltadef-general}
\tdelta_\theta  & := 2 \int_S \left[  e_u +  \int_S \psi (v) \meas_u
      (\ud v)\right]\pii ( \ud u ) 
      \nonumber\\
      &\qquad
      +(2\theta - 1) \int_S\bigg[\sigma^2_u + 2 \int_S \lambda_{u} (v) \psi (v) \cK ( u, \ud v)\bigg] \pii ( \ud u ). \end{align}
Then $\tdelta_\theta$ is invariant under translation of $\psi$, and $\tdelta_{\theta} \leq \tdelta_{\theta'}$
whenever $\theta \leq \theta'$. The following
classification applies.
  \begin{enumerate}[label=(\alph*)]
  \item
       \label{thm:strips-general-lamperti-a}
  The process $\xi$ is transient if $\tdelta_0 >0$ and recurrent
    if $\tdelta_0 <0$.
   \item
   \label{thm:strips-general-lamperti-b}
   If, moreover, $q \geq 2$, then $\xi$
    is positive recurrent if $\tdelta_1 <0$, while $\xi$ is null recurrent if  $\tdelta_0 <0 < \tdelta_1$.
  \end{enumerate}
	Moreover, if it also holds that
\begin{equation}
\label{eq:strips-convergence}
\lim_{n \to \infty} \sup_{u \in S} \left\| \cK^{n} (u, \, \cdot \, ) - \pii  (\, \cdot \, )  \right\|_\tv = 0 ,
\end{equation}	
then one may take $\psi \in \Cb(S)$ given by the convergent series
\begin{equation}
  \label{eq:psi-series}
  \psi ( u ) = \sum_{n=0}^\infty \int_S \cK^n ( u , \ud v ) d_v   .
\end{equation}
\end{theorem}
\begin{remarks}
\phantomsection
\label{rmks:psi} 
\begin{remenumi}
\item
\label{rmks:psi-a} 
An alternative expression for~\eqref{eq:psi-series} is obtained
in terms of the linear operator $T_\cK$
  associated with kernel~$\cK$, which acts on  bounded
  continuous $f : S \to \R$ via
  $T_\cK f(u) := \int_S \cK (u, \ud v ) f(v)$, and is discussed
  in detail in Appendix~\ref{sec:fredholm}. If we set
   $T_\cK^{n+1} := T_\cK \circ T_\cK^n$, $n \in \ZP$ (with
  $T_\cK^0$ the identity operator), then~\eqref{eq:psi-series} becomes
$\psi = \sum_{n = 0}^\infty T_\cK^n d$.
\item
\label{rmks:psi-b} 
Only for~\eqref{eq:psi-series} do we explicitly
  assume convergence of $\cK^n$ to the unique invariant
  probability~$\pii$; under~\eqref{ass:kernel},
  condition~\eqref{eq:strips-convergence} holds for any irreducible,
  aperiodic, Harris recurrent $\cK$: see e.g.~\cite[pp.~251,
  262]{dmps}.  For finite $S$, a version of
  Theorem~\ref{thm:strips-general-lamperti} was given in Theorem~2.6
  of~\cite{LoWade17}, without the identification of~$\psi$
  at~\eqref{eq:psi-series}.  Even with~\eqref{eq:psi-series}, the
  classification in Theorem~\ref{thm:strips-general-lamperti} is less
  explicit than that in Theorem~\ref{thm:strips-strict-lamperti} due
  to involvement of function~$\psi$, whose probabilistic significance
  is explained in the next remark.  In some cases, it is possible to
  compute~$\psi$ explicitly: see e.g.~\cite[\S 5.1]{Lo} and
  Example~\ref{ex:example-family} below.
\item
\label{rmks:psi-c} 
Although~\eqref{ass:lamperti} is weaker
  than~\eqref{ass:strict-lamperti},
  Theorem~\ref{thm:strips-general-lamperti} does not imply
  Theorem~\ref{thm:strips-strict-lamperti} because of the presence of
  the stronger conditions~\eqref{ass:asymptotically-markov+}
  and~\eqref{ass:detail-drifts}.  On the contrary, we deduce
  Theorem~\ref{thm:strips-general-lamperti} from
  Theorem~\ref{thm:strips-strict-lamperti} by showing that, under the
  hypotheses of Theorem~\ref{thm:strips-general-lamperti}, the process
  $(X_n + \psi (\eta_n), \eta_n )$ satisfies the assumptions of
  Theorem~\ref{thm:strips-strict-lamperti} with appropriately
  transformed parameters: see Theorem~\ref{thm:general-to-strict}
  below.
\end{remenumi}
\end{remarks}

Recall that $\tau_r$ is the passage time defined at~\eqref{eq:def-passage-time},
and that $\tdelta_\theta$ is defined by~\eqref{eq:deltadef-general}
in terms of the function~$\psi$ described in Theorem~\ref{thm:strips-general-lamperti}.
The following result on passage-time moments 
provides a quantification of recurrence, and is the analogue of Theorem~\ref{thm:strict-lamperti-moments}.
  In the case of finite~$S$, analogous results are Theorems~2.7 and~2.8 of~\cite{LoWade17}.

\begin{theorem}
    \label{thm:moments-general}
Suppose that~\eqref{ass:non-confinement}, \eqref{ass:kernel},
and~\eqref{ass:asymptotically-markov+} hold, and
that~\eqref{ass:moments} holds with~$p>2$ and $q \geq 2$.  Suppose also
that~\eqref{ass:lamperti} and~\eqref{ass:detail-drifts} hold.
    Define $\tdelta_\theta$ as at~\eqref{eq:deltadef-general}. 
  \begin{enumerate}[label=(\alph*)]
    \item 
     If $\tdelta_\theta<0$ for some $\theta>0$, 
  then for any $s \in [0, \theta \wedge p/2 \wedge q/2 )$
    there exists $r_1 \in \RP$ for which $\Exp[ \tau_r ^s ] <\infty$ for all $r \geq r_1$.
\item
 If $\tdelta_{\theta}>0$ for some $\theta\in (0,p/2 \wedge q/2]$,
 then for
 every $s > \theta$ and every $r \in \RP$, there exists $r_1 \in (r,\infty)$ for which $\Exp[ \tau_r ^s ] = \infty$ provided $X_0 = x_0$ satisfies $x_0 > r_1$.
 \end{enumerate}
  \end{theorem}

\section{Stochastic billiards} 
\label{sec:billiards}

\subsection{Model formulation and construction}
\label{sec:billiards-setup}

Fix a domain $\gR$ as defined at~\eqref{eq:region-def}, with $\gamma \in(0,1)$. 
We consider a stochastic billiards model that can be described informally as follows.
A particle moves at unit speed, in a fixed direction in the interior of $\gR$, until it hits the boundary,
at which point it reflects, randomly, according to a \emph{reflection kernel} $\Kb$ that operates on the \emph{incoming angle}
to give an \emph{outgoing angle}. Angles are measured relative to the inwards pointing unit normal vector
at the collision point.
Instead of working with the continuous-time process, we construct
a
discrete-time Markov process that records the collision locations and the incoming angles at the collisions;
the continuous-time process can be easily constructed from the collisions process, but as we do not need it in this paper, we omit the details.
 
We outline the construction of the discrete-time collisions process $\zeta := (\zeta_n, n \in \ZP)$ 
with $\zeta_n = (Z_n, \chi_n, \alpha_n ) \in \Sigmas := \RP \times \{ -1, +1\} \times S$,
where
$S := [ - \frac{\pi}{2}, \frac{\pi}{2} ]$, endowed with the usual Euclidean metric.
Here $Z_n \in \RP$ represents the horizontal coordinate of the collision location, $\chi_n \in \{ -1, +1\}$
is the sign of the vertical coordinate (with the convention that $\chi_n =1$ if $Z_n =0$), 
and $\alpha_n \in S$ is the incoming angle.
The Markov kernel $\Kb ( \alpha_n, \, \cdot \,)$ is then used to generate
the outgoing angle $\beta_n$. 
Our sign conventions are such that
if one extends the normal vector at a collision point (other than the origin) so as
to divide the domain~$\gR$ into one bounded and one unbounded component, 
 positive $\beta_n$ means that the outgoing trajectory enters the unbounded component, 
while positive $\alpha_n$ means that the incoming trajectory originates in the bounded component.  
There is then a deterministic function,
derived from the geometry of the problem, that gives $\zeta_{n+1}$ as a function of $(Z_n, \chi_n , \beta_n)$.
This gives us a Markov evolution for~$\zeta$. 
We now give the details.

Let $\Kb$ denote a Markov kernel on the compact metric space $(S,d_S)$.
  We also set $S_0 := [ - \theta_0, \theta_0 ]$ for some fixed $\theta_0 \in (0, \pi/2)$,
	and assume an \emph{ellipticity condition}:

\begin{description} 
\item[\namedlabel{ass:ellipticity}{B1}]
Suppose that $ \Kb ( \alpha , S_0 ) = 1$ for all $\alpha \in S$.
\end{description}
 
On a probability space $(\Omega, \cF, \Pr)$, let $U, U_1, U_2, \ldots$
be a sequence of independent $U[0,1]$ random variables, that will serve as our random inputs.
There is a measurable function $\Phi : S \times [0,1] \to S$ such that
$\Pr ( \Phi ( \alpha, U) \in B ) = \Kb ( \alpha, B)$ for $B \in \cB(S)$~(see e.g.~\cite[Lem.~3.22, p.~56]{kall}).
For $(z, j ) \in \RP \times \{-1, +1\}$, we let $h (z, j) := j z^\gamma$, so that if $z >0$, then
 $(z, h(z,\pm 1)) \in \partial \gR$ are the points on the upper and lower boundary at horizontal distance~$z$.
Note that $(0, h(0,j)) = (0,0)$ for either value of~$j$. 
For $(z, j) \in (0,\infty) \times \{-1,+1\}$, denote the inwards pointing
normal vector at $(z, h(z,j)) \in \partial \gR$ by 
\begin{equation}\label{eq:normalatz}
  n (z ,j)  : = \left( 1 + \gamma^2 z^{2\gamma - 2} \right)^{-1/2} \begin{bmatrix} \gamma z^{\gamma - 1} \\ -j \end{bmatrix} ;
\end{equation}
also set $n (0, j) := (1, 0)$ and let $\theta (z)$ represent the
magnitude of the angle between $n(z, j)$ and the vertical (see
Figure~\ref{fig:normal}). Put differently, $\theta(z)$ is given by
\begin{equation}\label{eq:normalangle}
 \theta(0) := \pi/2, \text{ and } \theta(z) :=  \arctan ( \gamma z^{\gamma-1} ) \text{ for } z >0 .
\end{equation}
Note that $\theta (z) \sim \gamma z^{\gamma-1}$ as $z \to \infty$.

\begin{figure}[!ht]
  \begin{center} 
\begin{tikzpicture}[domain=0:7, scale = 1.4]
\draw[black, line width = 0.40mm]   plot[smooth,domain=1:7,samples=500] ({\x},  {(\x)^(1/2)});
\draw[double,<-,>=stealth] (4.5,1.5) arc (-45:-67.5:1);
\draw[->,>=stealth] (4.24,1) arc (-75:-108:1);
\draw (4,2) -- (7,2.75);
\draw[black,dashed] (4,2) -- (4,-1);
\filldraw (4,2) circle (1.0pt);
\node at (3.9,2.4) {{\small $(z,h(z,1))$}};
\draw (1,1.25) -- (4,2);
\draw[->] (4,2) -- (4.6,-0.4);
\draw (3.8,1.95) -- (3.86,1.71);
\draw (4.06,1.76) -- (3.86,1.71);
\draw[->,>=stealth] (4,0.2) arc (270:284:1.8);
\node at (5.2,2) {{\small $\partial \gR$}};
\node at (4.25,0.03) {{\footnotesize $\theta(z)$}};
\node at (4.65,-0.6) {{\small $n (z,1)$}};
\node at (4.4,1.2) {{\footnotesize $\beta$}};
\node at (3.8,0.85) {{\footnotesize $\alpha$}};
\node at (5.5,0.8) {{\small $\ell_t (z,1,\beta)$}};
\draw[black,->,>=stealth, line width=1.3pt, line cap=round, dash pattern=on 0pt off 1.5\pgflinewidth] (4,2) -- (5,1);
\draw[black,-, line width=1.3pt, line cap=round, dash pattern=on 0pt off 1.5\pgflinewidth] (3.2,-0.4) -- (4,2);
\end{tikzpicture}
\caption{Point $(z,h(1,z)) \in \partial \gR$
   has inwards-pointing normal $n(z,1)$, making angle $\theta(z)$
with the vertical. The ray from $(z,h(1,z))$ at angle $\beta$ relative to the normal is parametrized by $\ell_t (z,1,\beta)$, $t >0$.
If the particle hits $\gR$ at point $(z,h(1,z))$ at incoming angle $\alpha$, then it reflects at outgoing angle $\beta$ drawn from $\cK(\alpha , \, \cdot \,)$.
In the picture, both~$\alpha$ and~$\beta$ are positive.
}
\label{fig:normal}
\end{center}
\end{figure}
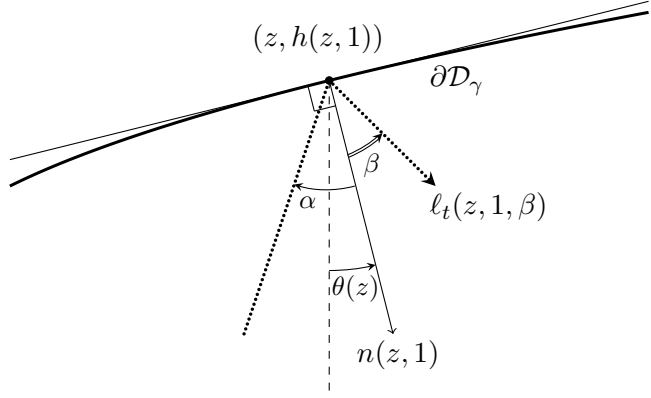

For $(j,\theta) \in \{-1,+1\} \times S$, define $\rot(j,\theta) :  \R^2 \to \R^2$ 
by
\begin{equation}
\label{eq:rotation}
  \rot(j,\theta)   \begin{bmatrix} x\\ y \end{bmatrix}:= 
  \begin{bmatrix}\cos \theta & -j\sin \theta\\
   j \sin \theta & \phantom{-} \cos \theta
  \end{bmatrix}  \begin{bmatrix} x\\ y \end{bmatrix} = 
	\begin{bmatrix}
 x\cos \theta - j y \sin \theta\\
j x \sin \theta +  y \cos \theta \end{bmatrix}.
\end{equation}
In words, $\rot(j,\theta)$ acts as a rotation by $\theta$, anticlockwise for $j=1$ and clockwise for $j=-1$.
Combining the notation at~\eqref{eq:rotation} with~\eqref{eq:normalatz} and~\eqref{eq:normalangle}, we obtain
\[
n (z ,j) = \rot (j,   \theta (z ) )
\begin{bmatrix}
  0 \\ -j
\end{bmatrix}
= \begin{bmatrix}
  \sin \theta (z) \\ -j \cos \theta (z)
\end{bmatrix} .
\]

Now we can describe the construction of the Markov chain.
We take arbitrary initial values for $(Z_0, \chi_0, \alpha_0) \in \Sigmas$ (subject to the convention $\chi_0 = 1$ if $Z_0=0$).
Given $(Z_n, \chi_n, \alpha_n) = (z, j, \alpha)$, if $z>0$
we generate an outgoing angle $\beta_n := \Phi (\alpha_n, U_n)$ according to the  kernel $\Kb$.
If $z=0$, then instead we take $\beta_n := \theta_0 (1 - 2U_n)$, a uniform angle on~$S_0$.

Given $(z,j) \in \RP \times \{-1,+1\}$, and an outgoing angle $\oa \in S$,
 we define the ray from $(z,h(z,j))$ with angle $\oa$ to be the open semi-line $L(z,j,\oa): = \{ \ell_t (z,j,\oa), t > 0\}$, where 
\begin{align}
\label{eq:ell-non-zero}
  \ell_t(z,j,\oa) & : =  \begin{bmatrix} z \\ h(z,j) \end{bmatrix} + t \, \rot (j, \theta(z) +\oa) \begin{bmatrix} 0 \\ -j \end{bmatrix} .
	\end{align}
Let $\lambda := \lambda (z, j, \oa): = \inf\{ t>0 \colon \ell_t(z,j,\oa)\in \partial \gR\}$ be
the travel time  of the particle until the next collision (equivalently, the distance between collision points).
To construct the subsequent boundary value, set 
\begin{equation}\label{eq:Lambda-def}
  \Lambda (z,j,\oa): = 
	\begin{cases}
    \ell_{\lambda}(z,j ,\oa) & \text{if }  \lambda = \lambda (z,j,\oa)< \infty, \\
    0 & \text{otherwise},
        \end{cases}
\end{equation} 
and write coordinates of $\Lambda$ as $\Lambda_1, \Lambda_2$.
Then, with $\sign(x) : = 2\cdot\1{x \geq 0} -1 $, define
\begin{equation}
  \label{eq:billiards-construction-location}
  Z_{n+1} = \Lambda_1 ( Z_n, \chi_n , \beta_n ) , ~\text{ and }~ \chi_{n+1} = \sign ( \Lambda_2 ( Z_n, \chi_n , \beta_n ) ).
\end{equation}
In words, given $(z,j)$ locating the particle on the boundary, and an
outgoing angle $\oa$, the subsequent boundary value is at the
intersection of the ray $L(z,j,\oa)$ and $\partial \gR$, assuming that
there is such an intersection.  One has $\lambda =\infty$ only if
$\theta(z) + \oa = \pi/2$, but this will be a probability zero event
for us, as we will assume $\Kb (\alpha , \, \cdot \, )$ has a density (see~\eqref{ass:billiards-density} below).

Finally, to determine the next incoming angle, if $z > 0$ and  if the outgoing angle at $(z,h(z,j))$ is $\oa$,
then
the incoming angle at $\Lambda ( z,j,\oa)$, as illustrated in
Figure~\ref{F:angle}, is 
\begin{equation}\label{eq:inangle}
   \Theta (  z, j , \oa ) := 
    \begin{cases}
      \oa + \theta(z) + \theta( \Lambda_1 ( z,j,\oa) )&\text{if }  j \Lambda_2 ( z , j , \oa ) < 0,\\
      \sign(\oa) \pi - \oa - \theta(z) + \theta( \Lambda_1 ( z,j,\oa) ) &\text{otherwise.}
      \end{cases}
 \end{equation}
 In the exceptional case that $z = 0$ we set
 $\Theta ( 0 , j , \oa ) := \frac{\pi}{2} - \abs{\oa} + \theta( \Lambda_1 ( 0,j,\oa) )$.  Then define
\begin{equation}
  \label{eq:billiards-construction-angle}
  \alpha_{n+1} := \Theta ( Z_n, \chi_n , \beta_n ) .
\end{equation}
The combination of~\eqref{eq:billiards-construction-location}, \eqref{eq:inangle}, \eqref{eq:billiards-construction-angle}
and the function~$\Phi$ that applies $\Kb$ to $\alpha_n$
to generate $\beta_n$
completes the construction of the time-homogeneous Markov chain $\zeta$;
to see this note that the functions $\Lambda$ and $\Theta$ are measurable, and use e.g.~\cite[Prop.~8.6, p.~145]{kall}. 
The next section describes our assumptions on the kernel $\Kb$ and our recurrence
classification.

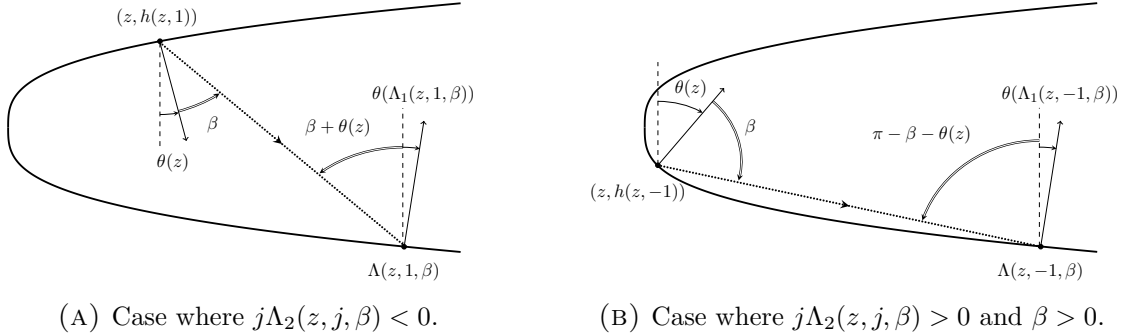
\begin{figure}[!ht]
\begin{subfigure}{.49\textwidth}
  \centering
\scalebox{0.62}{\begin{tikzpicture}[scale = 1.4]
	\draw[black, line width = 0.40mm]   plot[smooth,domain=0:1.9,samples=500] (\x^3,    \x);
	\draw[black, line width = 0.40mm]   plot[smooth,domain=0:1.9,samples=500] (\x^3,  -  \x);
    
\draw[->,>=stealth] (2.3,.2) arc (-90:-79:1.5);
\draw[double,->,>=stealth](2.58,.25) arc (-75:-58:2.3);
\draw[->,>=stealth] (5.98,-.3) arc (90:80:1.5);
\draw[double,->,>=stealth] (5.98,-.3) arc (90: 125:2.2);
		
		  \draw[decoration={
    markings,
    mark=at position 0.5 with {\arrow{stealth}}}, postaction={nomorepostaction,decorate}, 	
	black, >=stealth, line width=1.3pt, line cap=round, dash pattern=on 0pt off 1.5\pgflinewidth] (2.3, 1.3)-- (6, -1.8);
		
\draw[black,dashed]  (2.3,1.3)--(2.3,- .3);
    \draw[black, ->] (2.3,1.3)--(2.7,- .2);

    \draw[black,dashed] (5.98,-1.82)--(5.98, .3);
    \draw[black, ->] (6,-1.82)--(6.3, .1);

\filldraw (2.3,1.32) circle (1.0pt);
\filldraw (6,-1.82) circle (1.0pt);

    \node at (2.26, 1.68) {{\small $(z,h(z,1))$}};
    \node at (2.5, -0.54) {\small $\theta(z)$};
    \node at (3.1, 0.05) {\small $\oa$};

    \node at (6, -2.2) {\small $\Lambda ( z,1,\oa)$}; 
    \node at (6.3, .5) {\small $\theta(\Lambda_1 ( z,1,\oa))$};
    \node at (5, 0) {\small $\oa +\theta(z)$};
    
  \end{tikzpicture}
  }
  \caption{Case where $j \Lambda_2 ( z , j , \oa ) < 0$.}
  \label{F:case1}   
\end{subfigure}
\begin{subfigure}{.49\textwidth}
  \centering
  \scalebox{0.62}{
    \begin{tikzpicture}[scale = 1.4]
      \draw[black, line width = 0.40mm]   plot[smooth,domain=0:1.9,samples=500] (\x^3,    \x);
      \draw[black, line width = 0.40mm]   plot[smooth,domain=0:1.9,samples=500] (\x^3,  -  \x);
	
      \draw[->,>=stealth] (.2,0.4) arc (90:61:1.4);
      \draw[->,>=stealth] (5.98,-.3) arc (90:80:1.5);
      \draw[double,->,>=stealth] (1.03,.4) arc (50:-14:1.2);
      \draw[double,->,>=stealth] (5.98,-.2) arc (90: 159:1.9);
      
      \draw[decoration={
          markings,
          mark=at position 0.5 with {\arrow{stealth}}}, postaction={nomorepostaction,decorate}, 	
	black, >=stealth, line width=1.3pt, line cap=round, dash pattern=on 0pt off 1.5\pgflinewidth] (.2, -.577)-- (6, -1.82);
      
      \draw[black,dashed] (.2,-.57)--(.2, 1);
      \draw[black, ->] (.2,-.577)--(1.2, 0.6);
	
      \draw[black,dashed] (5.98,-1.82)--(5.98, .3);
    \draw[black, ->] (6,-1.82)--(6.3, .1);
    
    \filldraw (0.2,-0.577) circle (1.0pt);
    \filldraw (6,-1.82) circle (1.0pt);

    \node at (6, -2.2) {\small $\Lambda ( z,-1,\oa)$}; 
    \node at (6.3, .5) {\small $\theta(\Lambda_1 ( z,-1,\oa))$};
    \node at (-0.1, -1.0) {{\small $(z,h(z,-1))$}};
    \node at (0.7, 0.6) {\small $\theta(z)$};
    \node at (1.6, 0)  {\small $\oa$};
    \node at (4.2, -0.1) {\small$\pi - \oa - \theta(z)$};	
\end{tikzpicture}}
  \caption{Case where $j \Lambda_2 ( z , j , \oa ) > 0$ and $\oa >0$.}
\label{F:case2}
\end{subfigure}
\caption{Two examples of the computation of the new incoming angle
  $\Theta ( z, j , \oa )$ as given at~\eqref{eq:inangle}.
   In case~({\sc a}), the next collision point is on the opposite side
  of the domain, and
  $\Theta ( z, j , \oa ) = \oa + \theta(z) + \theta( \Lambda_1 (
  z,j,\oa) )$.  In case~({\sc b}), the next collision point is on the
  same side of the domain and $\oa>0$, so
  $\Theta ( z, j , \oa ) = \pi -\oa - \theta(z) + \theta( \Lambda_1 (
  z,j,\oa) )$.}
\label{F:angle}
\end{figure}

\subsection{Assumptions and results}
\label{sec:billiards-results}
 
Recall that the billiards reflection kernel $\Kb$ is a Markov kernel on the compact set $S = [ -\frac{\pi}{2}, \frac{\pi}{2} ]$.
In what follows, in addition to~\eqref{ass:ellipticity} above,
we assume the following density and spread conditions.
\begin{description}
\item[\namedlabel{ass:billiards-density}{B2}] Suppose that there is a bounded measurable $\kappa : S^2 \to \RP$ such that
   \begin{equation} \label{eq:billiards-density}
  \Kb (\alpha , B )= \int_B \kappa ( \alpha , \beta) \ud \beta , \text{ for all } B \in \cB(S).
\end{equation}
Moreover, suppose that $\kappa$ is uniformly equicontinuous in each
argument, i.e., for any $\eps >0$ there exists $\delta >0$ such that
(i)
$\sup_{\alpha \in S} | \kappa (\alpha, \beta) - \kappa (\alpha,
\beta') | \leq \eps$ for all $\beta, \beta' \in S$ with
$| \beta - \beta' | \leq \delta$, and (ii)
$\sup_{\beta \in S} | \kappa (\alpha, \beta) - \kappa (\alpha', \beta)
| \leq \eps$ for all $\alpha, \alpha' \in S$ with
$| \alpha - \alpha' | \leq \delta$.
\end{description}
\begin{description}
\item[\namedlabel{ass:loopable}{B3}] Suppose that~$\Kb$ is \emph{right progressive} in the sense that
there exists $\eps >0$ for which
\begin{equation}
\nonumber
 \Kb (\alpha, [\eps, \pi/2]) \geq \eps \text{ for all } \alpha \in S.
\end{equation}
\end{description}
Under~\eqref{ass:billiards-density}, if $\mu$ is an invariant measure for $\Kb$, then, by Fubini's theorem, 
\[
\mu (B) = \int_S \mu ( \ud \alpha ) \Kb ( \alpha, B ) = \int_B \left[ \int_S  \mu (\ud \alpha) \kappa (\alpha,\beta) \right] \ud \beta , ~\text{for any}~ B \in \cB(S).
\]
Hence every invariant measure $\mu$ has a density. The next assumption is uniqueness. 
\begin{description}
\item[\namedlabel{ass:billiards-invariant}{B4}]
 Suppose that $\Kb$ has a unique invariant probability measure~$\mu$, whose density we denote by~$\pib$.
\end{description} 

Define for $k \in \N$ and $\alpha \in S$, 
\begin{align}
\label{eq:rho-def}
\rho_k ( \alpha ) & := \int_S \Kb ( \alpha , \ud \beta) \tan^k \beta = \int_S \kappa (\alpha , \beta) \tan^k \beta \,  \ud \beta; \\
\label{eq:billiards-average-drift}
 \bar \rho_k & := \int_S \pib (  \alpha) \rho_k (\alpha)  \ud \alpha  = \int_S \pib (  \beta ) \tan^k \beta \,  \ud \beta;
\end{align}
the second equality in~\eqref{eq:rho-def}
uses the reflection density $\kappa$~from~\eqref{eq:billiards-density}, while the second
equality 
in~\eqref{eq:billiards-average-drift} uses~\eqref{eq:rho-def}, Fubini's theorem, and the invariance of~$\pib$, as assumed at~\eqref{ass:billiards-density}. 
Our first result deals with the case where $\bar \rho_1 \neq 0$.

\begin{proposition}
\label{prop:billiards-non-critical}
Suppose that~\eqref{ass:ellipticity}--\eqref{ass:billiards-invariant} hold. Then $\zeta$ is transient if $\bar \rho_1 >0$ and recurrent if $\bar \rho_1 < 0$. 
\end{proposition}

As we shall explain when we make the connection to the half-strip model,
the critical (Lamperti) regime corresponds to~$\bar \rho_1 = 0$;
this case occurs if, for example, $\pib$ is symmetric about~$0$. 
As in the half-strip model, the case where $\rho_1 (\alpha) =0$ for all $\alpha$ is simpler,
and corresponds to the \emph{strict} Lamperti regime in the terminology of Section~\ref{sec:strips-results}. Here the
key quantity determining the classification is 
\begin{equation}
\label{eq:critical-parameter-strict}
  \gammacs : = \frac{\bar \rho_2}{1+ 2\bar \rho_2} = \frac{\int_{S} \pib ( \alpha) \tan^2 \alpha \, \ud \alpha}{1 + 2\int_{S}   \pib ( \alpha) \tan^2 \alpha \, \ud \alpha}.
\end{equation}
We use the extra subscript `$0$' to indicate the strict Lamperti setting, to parallel~\eqref{ass:strict-lamperti}.
Note that, by~\eqref{ass:ellipticity},
$\int_{S} \pib ( \alpha) \tan^2 \alpha \, \ud \alpha <\infty$ and therefore
$\gammacs \in (0,1/2)$.

\begin{theorem}
\label{thm:billiards-strict-lamperti}
Suppose that~\eqref{ass:ellipticity}--\eqref{ass:billiards-invariant} hold, and $\rho_1 ( \alpha ) = 0$ for all $\alpha \in S$.
Then $\zeta$ is transient if $\gammacs < \gamma < 1$ and recurrent if $0 < \gamma < \gammacs$,
where $\gammacs$ is given by~\eqref{eq:critical-parameter-strict}.
\end{theorem}  

\begin{remarks}
\phantomsection
\label{rems:billiards-real}
\begin{remenumi}
\item
\label{rems:billiards-real-a}
In the case where
  $\kappa ( \alpha , \beta)$ does not depend
  on~$\alpha$, reflection angles are i.i.d.~with density~$\pib$, and the result of
  Theorem~\ref{thm:billiards-strict-lamperti} is due to~\cite{mvw}.
\item
\label{rems:billiards-real-b}
Recurrence/transience of $\zeta$ will transfer to the continuous-time process
that follows the trajectories of the particle. 
This is because Lemma~\ref{lem:lower-bound} below gives a lower bound on the real time
elapsed between successive collisions under assumption~\eqref{ass:ellipticity}.
\item
\label{rems:billiards-real-c}
Theorem~\ref{thm:billiards-strict-lamperti} does not admit $\gamma = \gammacs$. 
We omit this case from our analysis, since to treat the corresponding critical case
of the half-strip model requires slightly stronger assumptions, 
a more refined Lyapunov function, and associated additional technicalities: cf.~Remark~\ref{rem:null-critical-equality}.
However, we anticipate that the billiards model would satisfy the necessary stronger
assumptions, and hence we would expect that the case $\gamma = \gammacs$ is recurrent. 
\item
\label{rems:billiards-real-d}
One extension of our model is to a Markov process $(\zeta_n, \alpha_n, u_n)$, in which $u_n \in V$ 
is the speed of the particle between the $n$th and $(n+1)$st collisions,
 $V \subset (0,\infty)$ is compact, 
and the refection kernel $\cK$ is extended to operate on $S^\star: = S \times V$.
Under the natural extension of our assumptions, our analysis extends to this case, at the expense of some heavier notation, 
working now on the half-strip $\RP \times S^\star$.
If outgoing angle $\beta_n$ depends only on incoming angle $\alpha_n$, and not on incoming speed $u_n$, then
 our recurrence/transience results would be unchanged, the speeds playing no role. 
 In general, the density~$\varpi$ in the critical parameter $\gammacs$ at~\eqref{eq:critical-parameter-strict} would correspond to the $S$-projection of the stationary distribution
 of $\cK$ on $S^\star$.
\item
\label{rems:billiards-real-e}
A possible generalization would be to relax the assumption~\eqref{ass:ellipticity},
allowing the reflection density $\kappa (\alpha, \beta)$ to be supported on the whole of
$\beta \in S = [-\frac{\pi}{2}, +\frac{\pi}{2} ]$, but with suitable bounds on the tails near $\beta = \pm \pi/2$.
The half-strip setting of Section~\ref{sec:strips-results} can accommodate unbounded increments, 
and the $p$th-moments condition in~\eqref{eq:lamperti-p-moments}
translates to a condition of the form $\sup_{\alpha \in S} \int_S \kappa (\alpha, \beta) | \tan \beta |^p \ud \beta < \infty$ (the relevant technical results on the increments of the billiards process are Lemmas~\ref{lem:displacement} and~\ref{lem:billiards-to-strip-markov}).
However, the possibility of multiple collisions in rapid succession introduces some technical obstacles in the billiards setting. Furthermore, the example of the Lambertian density $\kappa (\alpha , \beta ) = (1/2) \cos \beta$ has $\int_S \kappa (\alpha, \beta)   | \tan \beta |^p  \ud \beta  < \infty$ if and only if $p <2$,
which suggests that further exploration of this interesting generalization might fruitfully take place in a heavy-tailed setting; in the Lambertian case, taking $\bar\rho_2 = \infty$ in~\eqref{eq:critical-parameter-strict} suggests the conjecture~$\gammacs = 1/2$ (cf.~Open Problem~6.4.13 in~\cite[p.~307]{MenPopWad16}).
\end{remenumi}
\end{remarks}

The more general case requires further assumptions, and produces a less explicit result.
In particular, we assume that the reflection densities in~\eqref{eq:billiards-density} are sufficiently smooth.

\begin{description}
\item[\namedlabel{ass:density-smoothness}{B5}] 
Suppose that 
$\kappa' (\alpha,\beta) := (\partial / \partial \beta ) \kappa (\alpha,\beta)$  
and $\kappa'' (\alpha,\beta) := (\partial^2 / \partial^2 \beta  ) \kappa (\alpha,\beta)$ exist,
are continuous in each argument, and are bounded uniformly for all $\alpha, \beta \in S$.
\end{description}

Under~\eqref{ass:density-smoothness}, 
	$\pib (\beta) = \int_S \pib ( \alpha)  \kappa (\alpha , \beta ) \ud \alpha$
	is differentiable, since $\frac{\partial}{\partial \beta} \kappa (\alpha, \beta)$ exists and is uniformly bounded over $\alpha, \beta \in S$, with 
	$\pib' (\beta) := \frac{\partial}{\partial \beta} \pib (\beta)$ given by
	\begin{equation}
	\label{eq:pi-derivative}
	\pib' (\beta ) = \int_S  \kappa' (\alpha , \beta ) \pib ( \alpha) \ud \alpha  .
        \end{equation}
				In particular, $\pib$ is continuous on $S$ and vanishes outside $S_0$, so $\pib (\pm \theta_0) = 0$.

 \begin{theorem}
\label{thm:billiards-general-lamperti}
Suppose that~\eqref{ass:ellipticity}--\eqref{ass:density-smoothness} hold, and~$\bar \rho_1 =0$. 
Then there exists $\psi_0 \in \Cb (S)$  (unique up to translation) with
$\int_S ( \psi_0 (\alpha) - \psi_0(\beta) ) \kappa (\alpha, \beta) \ud \beta = 2 \rho_1 (\alpha)$
for all $\alpha \in S$.  
Let
\begin{align}
\label{eq:A-def}
 A_1 & :=\int_S \psi_0 (\beta ) \pib (\beta ) \tan \beta \, \ud \beta, ~~ A_2:= \int_S \psi_0 (\beta)   \pib' ( \beta) \, \ud \beta .\end{align}
 Then $A_1 + \bar\rho_2 \geq 0$. Suppose that the quantities in~\eqref{eq:A-def} satisfy
 \begin{equation}
\label{eq:A-conditions} 
 1 +A_1 - A_2 + 2\bar\rho_2 \geq 0, \text{ and } \max \{ A_1 + \bar\rho_2 , 1 +A_1 - A_2 + 2\bar\rho_2 \} > 0,
 \end{equation}
 and define 
 \begin{equation}
\label{eq:critical-parameter-general}
\gammag := \left( \frac{A_1 + \bar \rho_2 }{1 + A_1 - A_2 + 2 \bar \rho_2 } \right) \wedge 1 . 
\end{equation}
Then $\gammag \in [0,1]$, and $\zeta$ is transient if $\gammag < \gamma < 1$ and recurrent if $0 < \gamma < \gammag$.

Moreover, if, in addition, 
$\lim_{n \to \infty} \sup_{\alpha \in S} \| \cK^n ( \alpha, \, \cdot \,) - \mu (\, \cdot \,) \|_\tv = 0$, 
where~$\mu$ is the unique invariant measure from~\eqref{ass:billiards-invariant}, 
then one may take~$\psi_0$ defined by
\begin{equation}
\label{eq:psi-0-billiards}
\psi_0 (\alpha ) = 2 \sum_{n = 0}^\infty \int_S \cK^n ( \alpha, \ud \beta ) \rho_1 (\beta ) = 2 \sum_{n = 1}^\infty \int_S \cK^n ( \alpha, \ud \beta ) \tan \beta .
\end{equation}
\end{theorem}

\begin{remark}
\label{rem:billiards-general}
Under the hypothesis~\eqref{eq:A-conditions}, 
the fraction in~\eqref{eq:critical-parameter-general} is not $0/0$, and hence $\gammag \in [0,1]$ is well-defined (with the usual interpretations that $1/0 := \infty$ and $\infty \wedge 1 := 1$).
We do not rule out, however, the possibilities $\gammag = 0$ (if and only if $A_1 + \bar\rho_2 = 0$)
or $\gammag = 1$. In these extreme cases there would be no phase transition for $\gamma \in (0,1)$. 
\end{remark}

The quantities $A_0, A_1$ depend on the function $\psi_0$ and on $\varpi$, and so are hard to compute in general.
However, for a restricted class of reflection kernels,
we apply Theorem~\ref{thm:billiards-general-lamperti} to obtain Proposition~\ref{prop:example-family},
which shows that 
the classification of Theorem~\ref{thm:billiards-strict-lamperti}
extends beyond the case $\rho_1 (\alpha ) \equiv 0$: this can be seen as a further
generalization of the results of~\cite{mvw} from the case of i.i.d.~reflections.

\begin{proposition}
\label{prop:example-family}
Suppose that~\eqref{ass:ellipticity}--\eqref{ass:density-smoothness} hold, $\bar \rho_1 =0$,
and that for some $\lambda \in (-1,1)$,
\begin{equation}
\label{eq:rho1-condition}
\rho_1 (\alpha) = \lambda \tan \alpha, \text{ for all } \alpha \in S .
\end{equation}
Suppose also that $\varpi (\beta) = \varpi(-\beta)$ for all $\beta \in S$. 
Then $\zeta$ is transient if $\gamma > \gammacs$ and recurrent if $\gamma < \gammacs$,
where $\gammacs \in (0,1/2)$ is given by~\eqref{eq:critical-parameter-strict}.
\end{proposition} 

The next example gives a family of reflection kernels
to which Proposition~\ref{prop:example-family} applies.

\begin{example}
\label{ex:example-family}
Fix $\lambda \in (-1,1)$. 
Take for $\alpha, \beta \in S$,
\begin{equation}
\label{eq:example-density}
  \kappa (\alpha , \beta) = f (\alpha ) g (\beta ) + ( 1 - f(\alpha) ) g (-\beta ),
\end{equation}
where~$f: S \to [0,1]$ is uniformly continuous on~$S$ and satisfies
\[
f (\alpha ) = \frac{1}{2} + \frac{\tan \alpha}{2\tan \theta_0}, ~~~ \alpha \in S_0,
\]
and $g: S \to \RP$ will be constructed later to satisfy (i) $g$ is twice continuously differentiable on~$S$, (ii) $g (\beta ) = 0$ for $\abs{\beta} \geq \theta_0$;
(iii) $\int_{S_0} g(\beta) \ud \beta =1$; and (iv)
\begin{equation}
\label{eq:pulse-g}
 \int_{S_0} g (\beta ) \tan \beta \ud \beta = \lambda \tan \theta_0 .
\end{equation}
These properties for~$g$ ensure that~\eqref{ass:density-smoothness} holds,
and, by~(iii),
$\int_{S_0} g(-\beta) \ud \beta =1$, so that
$\kappa$ defined at~\eqref{eq:example-density} satisfies $\int_{S_0} \kappa (\alpha, \beta ) \ud \beta =1$ for all $\alpha \in S$.
Note also that~\eqref{eq:pulse-g} implies 
\[
  \int_{S_0} g (-\beta ) \tan \beta \ud \beta =  \int_{S_0} g (\beta ) \tan (-\beta) \ud \beta = - \lambda \tan \theta_0 , \]
so that, for any $\alpha \in S$, $\rho_1 (\alpha)$ defined at~\eqref{eq:rho-def} satisfies
\begin{align*} \rho_1 (\alpha) & = \int_{S_0} \kappa (\alpha, \beta ) \tan \beta \ud \beta   =  \left[ 2 f(\alpha) -1 \right]  \lambda \tan \theta_0 = \lambda \tan \alpha ,\end{align*}
by choice of~$f$, verifying~\eqref{eq:rho1-condition}.
Any stationary density $\pib$ for $\kappa$ given by~\eqref{eq:example-density} must satisfy
\[
  \pib (\beta ) = a g (\beta) + (1 - a) g (-\beta), ~~~ a: = \int_{S_0} f(\alpha) \varpi (\alpha ) \ud \alpha .
\]
Substituting the former equality for $\pib$ in the definition of $a$ and using the fact that $f(-\alpha) = 1 - f(\alpha)$
leads to
$ a ( 1- b) = (1-a) (1-b)$, where $b:= \int_{S_0} f(\alpha) g(\alpha) \ud \alpha \in (0,1)$.
It follows that $a=1/2$, and hence the unique stationary density is
\[
  \pib (\beta ) = \frac{1}{2} \left( g(\beta) + g (-\beta) \right) = \pib (-\beta) , ~~~ \beta \in S.
\]
Hence, by~\eqref{eq:billiards-average-drift} and~\eqref{eq:rho1-condition}, 
$\bar \rho_1 = \lambda \int_{S_0} \varpi ( \alpha )  \tan \alpha\,  \ud \alpha = 0$.

Thus all the conditions of Proposition~\ref{prop:example-family} are satisfied. It remains to check that a suitable~$g$ satisfying (i)---(iv)
above can be chosen.
We present one reasonably concrete construction.
For real numbers $m,n \geq 3$ to fixed later, take
 $P(\alpha): = (\alpha + \theta_0)^m (\theta_0 - \alpha)^n$ for $|\alpha| \leq \theta_0$.
 Then $P$ has zeros at~$\pm \theta_0$, and is strictly positive on $(-\theta_0, \theta_0)$. Define
\begin{equation}
\nonumber
 G(\alpha) :=\begin{cases}
P (\alpha ) & \text{if } |\alpha|< \theta_0,\\
0 & \text{otherwise}.
 \end{cases}
\end{equation}
For $\alpha \in (-\theta_0, \theta_0)$, $P(\alpha)$ is infinitely differentiable, and the $k$th derivative $P^{(k)} (\alpha)$
is a sum of products involving $ (\alpha + \theta_0)^r (\theta_0 - \alpha)^s$
for exponents $r,s$ satisfying $m-k \leq r \leq m$ and $n-k \leq s \leq n$.
Since $m, n \geq 3$, 
it follows that 
 the first two derivatives of~$P$ approach~$0$ continuously at $\pm \theta_0$. 
We normalize $G$ to obtain our~$g$, via
\begin{equation}
\nonumber
g(\alpha) := G(\alpha)/Z, ~~~ Z := \int_{S_0} P(\alpha) \ud \alpha = ( 2 \theta_0)^{m+n+1} \frac{\Gamma (n+1) \Gamma (m+1)}{ \Gamma (n+m+2)} ,
\end{equation}
evaluating the integral using the change of variable~$\alpha = \theta_0 (2u-1)$, $u \in [0,1]$. 
By the properties of $P$ and $G$ described, properties (i)--(iii) hold for this~$g$.

To achieve~(iv), i.e.~\eqref{eq:pulse-g}, we describe how to tune $n, m$ in the choice of $P$.
Note that~$P$ admits a unique maximum in $[-\theta_0, \theta_0]$ at the point $\alpha = \alpha^* := \frac{m-n}{m+n} \theta_0$.
Suppose that $\lambda \in (0,1)$. Fix $n$ ($n=3$ will do).
Now, as $m \to \infty$, $\alpha^* \to \theta_0$
and $g$ converges to the Dirac mass at $\theta_0$, and
by the dominated convergence theorem, 
\begin{equation}
\nonumber
\lim_{m \to \infty} \int_{S_0}  g(\alpha) \tan \alpha \, \ud \alpha = \tan \theta_0.
\end{equation}
The function
$m \mapsto \int_{S_0}  g(\alpha) \tan \alpha \, \ud \alpha$ is continuous,
takes value $0$ when $m=n$ and, as $m \to \infty$, eventually exceeds
$\lambda  \tan \theta_0$ (since $\lambda < 1$). Hence, by the intermediate value theorem, there exists $m > n$ for which~\eqref{eq:pulse-g} holds.
On the other hand, if $\lambda \in (-1,0)$, a similar argument applies with $n \to \infty$. If $\lambda =0$, we can take $m=n$.
\end{example}

\begin{remark}
We prove our results for the stochastic billiards model
by considering the process $(X_n ,\alpha_n)$ where $X_n := Z_n^{1-\gamma}$,
and $\alpha_n$ is the sequence of \emph{incoming} angles, 
in the framework of the half-strip model of Section~\ref{sec:strips}.
One could instead work with the process $(X_n ,\beta_n)$, with $\beta_n$ the sequence
of \emph{outgoing} angles. Again the results of  Section~\ref{sec:strips}
can be applied,  although the technical details differ. It is worth noting that although the~reflection kernel $\Kb$
is the same in both approaches, $\rho_1 (\alpha)$ and hence $\psi_0$ differ, but the
ultimate quantities $A_1$ and $A_2$ in~\eqref{eq:A-def} are the same.
\end{remark}

\section{Proofs for the half-strip model}
\label{sec:strips-proofs}

In this section we prove the results presented in Section~\ref{sec:strips-results},
and we adopt the notation of that section. In particular, note that $S$ is a general compact metric space. We first in Section~\ref{sec:strict-lamperti-proofs} work in the strict Lamperti regime, and then (in Section~\ref{sec:general-lamperti-proofs}) use a transformation to reduce the more general Lamperti setting to the strict case, with appropriate transformation of parameters. 

\subsection{The strict Lamperti regime}
\label{sec:strict-lamperti-proofs}

To prove 
Theorems~\ref{thm:strips-strict-lamperti} and~\ref{thm:strict-lamperti-moments}, we use
a Lyapunov function among the class of functions
$H_{\nu,\varphi}: \Sigma \to \RP$ defined in terms of a given
$\varphi \in \Cb^+(S)$ and a parameter $\nu \in \R$ by
\begin{equation}
\label{eq:strips-lyapunov}
 H_{\nu,\varphi}(x,u): = \begin{cases} 1 & \text{if } x \leq 1, \\
	x^{\nu} + \frac{\nu}{2} x^{\nu-2} \varphi(u) & \text{if } x > 1. \end{cases}
\end{equation}
For appropriate choices of $\nu$ and $\varphi$, depending on the parameters of the process $\xi_n$, the process $H_{\nu,\varphi} (\xi_n)$ will satisfy an appropriate super/submartingale condition outside a bounded set, which will enable us to apply martingale methods for adapted processes on $\RP$. 
In this direction, the following result estimates the expected increment of the process $H_{\nu,\varphi} (\xi_n)$.

\begin{proposition}
\label{prop:strips-lyapunov}
Suppose that~\eqref{ass:kernel}
and~\eqref{ass:asymptotically-markov} hold, and
that~\eqref{ass:moments} holds with~$p>2$ and~$q>0$. Suppose also
that~\eqref{ass:strict-lamperti} holds.  Let $\varphi\in \Cb^+(S)$ and
$\nu \in (2-p, p \wedge q]$. Then $\Exp H_{\nu, \varphi} ( \xi_n ) < \infty$
for all $n \in \ZP$, and
\[
  \Exp \bigl[ H_{\nu, \varphi} (\xi_{n+1}) - H_{\nu, \varphi}(\xi_n) \bigmid \cF_n \bigr] = W_{\nu,\varphi} ( \xi_n ) , \as,
\]
where
\begin{align}
\label{eq:strips-lyapunov-increment}
W_{\nu,\varphi} (x, u) 
& =  \frac{\nu}{2} x^{\nu-2} \left[ 2e_{u}  - (1-\nu) \sigma^2_{u} + \int_{ S} ( \varphi (v) - \varphi (u) ) \cK (u , \ud v) + \eps_{x,u} \right] , 
    \end{align}
with $\lim_{x \to \infty} \sup_{u \in S} | \eps_{x,u} | = 0$.
\end{proposition} 

We defer the proof of Proposition~\ref{prop:strips-lyapunov}
until the end of this subsection.  Let
\begin{equation}
\label{eq:pi-kernel}
 \Cb^0 (S) := \Big\{ g \in \Cb (S) : \int_S g (u)  \pii (\ud u) = 0\Big\},
\end{equation}
where~$\pii$ is the stationary measure from~\eqref{ass:kernel}\ref{ass:kernel-i}.
Define $g_\theta \in \Cb (S)$ by $g_\theta (u) : = 2e_{u}  + (2\theta -1)\sigma^2_{u}
  -\delta_{\theta}$,
  where $\delta_\theta$ is defined at~\eqref{eq:deltadef}.
 Then, by~\eqref{eq:deltadef},  $\int_S g_{\theta} (u) \uppi(\ud u) = 0 $, i.e., $g_{\theta} \in \Cb^0(S)$ as at~\eqref{eq:pi-kernel}. Hence, by~\eqref{ass:kernel} and Proposition~\ref{prop:fredholm}, 
there is a $\varphi_{\theta} \in \Cb^+ (S)$ for which
\begin{equation}
\label{eq:choice-of-psi-moments}
  \int_S  (\varphi_{\theta} (u) - \varphi_{\theta} (v) ) \cK ( u, \ud v )  = g_{\theta} ( u ), \text{ for all } u \in S;
\end{equation}
note we have specified that $\varphi_\theta \geq 0$.

We prove Theorem~\ref{thm:strips-strict-lamperti}\ref{thm:strips-strict-lamperti-a};
 Theorem~\ref{thm:strips-strict-lamperti}\ref{thm:strips-strict-lamperti-b} is a special case of Theorem~\ref{thm:strict-lamperti-moments} (cf.~Remark~\ref{rem:null-critical-equality}) which we prove later in this section.

\begin{proof}[Proof of Theorem~\ref{thm:strips-strict-lamperti}\textnormal{\ref{thm:strips-strict-lamperti-a}}.]
Recall the definitions of $e, \sigma^2$ from~\eqref{eq:generalized-lamperti-regime} and $\delta_0$ in the $\theta=0$ case of~\eqref{eq:deltadef}. Recall also that the functions $g_0 \in \Cb^0(S)$ and $\psi_0 \in \Cb^+ (S)$
satisfy $g_0 (u) := 2 e_u - \sigma^2_u - \delta_0$ and the $\theta=0$ case of~\eqref{eq:choice-of-psi-moments}.
Then 
$W_{\nu,\varphi_0}$ given at~\eqref{eq:strips-lyapunov-increment} satisfies
\begin{equation}
\label{eq:delta-recurrence}
W_{\nu,\varphi_0} (x, u) =  \frac{\nu}{2} x^{\nu-2} \left( \delta_0 + \nu \sigma^2_u + \eps_{x,u} \right),
\end{equation}
with $\sup_{u\in S} \abs{\eps_{x,u}} \to 0$ as $x \to \infty$.
Suppose that $\delta_0 < 0$.
Proposition~\ref{prop:strips-lyapunov} with~\eqref{eq:delta-recurrence} then shows that
there exist $\nu > 0$ and $r_0 \in \RP$ for which 
 \begin{equation}\label{eq:strips-lyapunov-drift}
 \Exp \bigl[ H_{\nu, \varphi_0} (\xi_{n+1}) - H_{\nu, \varphi_0}(\xi_n) \bigmid \cF_n \bigr] \leq 0, \text{ on } \{ X_n \geq r_0 \}. 
\end{equation}
Since $\nu>0$ it follows that
$\inf_{u \in S} H_{\nu, \varphi_0} (x,u) \to \infty$ as
$x \to \infty$. Then by~\eqref{ass:non-confinement}
 and Lemma~\ref{lem:recurrence}, we conclude that~$\xi$
is recurrent.  On the other hand, suppose that $\delta_0 > 0$.  Now,
by~\eqref{eq:delta-recurrence}, there exist $\nu < 0$ and
$r_0 \in \RP$ for which~\eqref{eq:strips-lyapunov-drift} again holds,
but now $\sup_{u \in S} H_{\nu, \varphi_0} (x,u) \to 0$ as
$x \to \infty$, and Lemma~\ref{lem:transience} 
with~\eqref{ass:non-confinement} shows that~$\xi$ is transient.
\end{proof}

Now we turn to the proof of Proposition~\ref{prop:strips-lyapunov}. The necessary computations
run along similar lines to those in the proof of Lemma~3.2 in~\cite{LoWade17},
which is a similar result in the case of \emph{finite}~$S$, under similar hypotheses.
We give the outline of the arguments, emphasizing the differences from~\cite{LoWade17}.
For ease of notation, define for $n \in \ZP$,
\begin{equation}
\nonumber
  \Delta_n := X_{n+1} - X_n, \text{ and } D_{\nu,\varphi, n} := H_{\nu,\varphi} ( \xi_{n+1} )- H_{\nu,\varphi} ( \xi_n ).
\end{equation}
In our proofs we often separate computations of expected functional increments over whether or not the increment is relatively big; for this purpose, we define the event
\begin{equation}
\label{eq:truncation-event}
E_{n,r} := \{ | \Delta_n | \leq X_n^r \}, \text{ for } r \in (0,1). 
\end{equation}
The following technical
lemma is the analogue of Lemma~3.3 of~\cite{LoWade17} and is proved similarly.

\begin{lemma}
\label{lem:truncation}
Suppose that~\eqref{ass:moments} holds for some $p>2$, $B_p \in \RP$, $q>0$,
and $x_B \in \RP$.
Then for all
$r \in (0,1)$, all
$s \in [0,p]$, and all $n \in\ZP$,
\begin{equation}
\label{eq:truncation}
\Exp \bigl[ | \Delta_n |^s \2 { E_{n,r}^\rc } \bigmid \cF_n \bigr] \leq B_p  X_n^{r(s-p)} , \text{ on } \{ X_n \geq x_B \} .
\end{equation}
Moreover, if $r \in ( \frac{1}{p-1}, 1 )$, then for $k \in \{1,2\}$,
\begin{equation}
\label{eq:moments-truncated} 
\Exp \bigl[ \Delta_n^k \2 { E_{n,r} } \bigmid \cF_n \bigr] = \Exp \bigl[ \Delta_n^k \bigmid \cF_n \bigr] + X_n^{k-2} \eps_k ( \xi_n ) ,
\text{ on }\{ X_n \geq x_B \} ,
\end{equation}
where $\lim_{x \to \infty} \sup_{u \in S} | \eps_k ( x, u) | =0$.
\end{lemma} 

On the event $E_{n,r}$ given by~\eqref{eq:truncation-event}, we have $X_{n+1} \geq X_n - X_n^r$, and so for
fixed $r \in (0,1)$ we may choose $x_1 \in \RP$ such that
$X_{n+1} > 1$ on $\{ X_n \geq x_1 \} \cap E_{n,r}$.  Hence,
by~\eqref{eq:strips-lyapunov},
\begin{align}
\label{eq:Dn-small-jump}
\Exp [ D_{\nu,\varphi, n} \2 { E_{n,r} } \mid \cF_n ] =
U_{\nu, \varphi,r} ( \xi_n) + \frac{\nu}{2} V_{\nu, \varphi,r} ( \xi_n) , \text{ on } \{ X_n \geq x_1 \},
\end{align}
where,
\begin{align*}
  U_{\nu, \varphi,r} ( \xi_n) &:=
   \Exp \left[ \left( X_{n+1}^\nu - X_n^\nu \right) \2 {E_{n,r}} \mid \cF_n \right] ; \\
  V_{\nu, \varphi,r} ( \xi_n) &:=
  \Exp \left[ \left( X_{n+1}^{\nu-2} \varphi (\eta_{n+1})  - X_n^{\nu-2} \varphi (\eta_{n}) \right) \2 {E_{n,r}} \mid \cF_n \right].  
\end{align*}
The next two results give asymptotics for $U_{\nu, \varphi,r}$ and $V_{\nu, \varphi,r}$.

\begin{lemma}
\label{lem:U-estimate}
Suppose that~\eqref{ass:moments} holds for some $p>2$ and $q>0$. Suppose also
that~\eqref{ass:strict-lamperti} holds, and that $r \in ( \frac{1}{p-1}, 1 )$.
Then for any $\nu \in \R$,
\begin{equation}
\nonumber
  U_{\nu, \varphi,r} ( x, u) = \frac{\nu}{2} x^{\nu -2} \left[ 2 e_u - (1-\nu) \sigma_u^2 + \eps_{x,u} \right] ,
\end{equation}
where $\lim_{x \to \infty} \sup_{u \in S} | \eps_{x,u} | =0$.
\end{lemma}
\begin{proof}
The proof is 
similar to that of Lemma~3.4 in~\cite{LoWade17}. On the event $E_{n,r}$, we can apply Taylor's
theorem to get
\begin{equation}
\nonumber
  \left( X_{n+1}^\nu - X_n^\nu \right) \2 {E_{n,r}} = \nu X_n^{\nu-1} \Delta_n \2 {E_{n,r}} + \frac{\nu (\nu-1)}{2}  X_n^{\nu-2} \Delta_n^2 \2 {E_{n,r}} + \omega_n ,
\end{equation}
where $|\omega_n| \leq C X_n^{\nu-3} | \Delta_n |^3 \2 {E_{n,r}} \leq C | \Delta_n |^2 X_n^{\nu+r-3}$. Using~\eqref{ass:moments}
and~\eqref{eq:moments-truncated}, we get
\begin{equation}
\nonumber
  U_{\nu, \varphi,r} (x, u )
  = \nu x^{\nu-1} \mu_1 (x,u) + \frac{\nu (\nu-1)}{2} x^{\nu-2} \mu_2 (x,u) + x^{\nu-2} \eps_{x,u} ,
\end{equation}
where $\lim_{x \to \infty} \sup_{u \in S} | \eps_{x,u} | =0$. The result now follows from~\eqref{ass:strict-lamperti}.
\end{proof}

\begin{lemma}
\label{lem:V-estimate}
Suppose that~\eqref{ass:asymptotically-markov} holds and that~\eqref{ass:moments} holds for some $p>2$ and $q>0$. 
Then for any $r \in ( 0, 1 )$ and any $\nu \in \R$,
\[
 V_{\nu, \varphi,r} ( x, u) =   x^{\nu -2} \left[ \int_S \left( \varphi(v) - \varphi(u) \right) \cK (u, \ud v) + \eps_{x,u} \right] ,\]
where $\lim_{x \to \infty} \sup_{u \in S} | \eps_{x,u} | =0$.
\end{lemma}
\begin{proof}
Similarly to the proof of Lemma~3.5 of~\cite{LoWade17}, first note that
\[
  \Exp \left[ \left| \left( X_{n+1}^{\nu-2} - X_n^{\nu-2} \right) \varphi (\eta_{n+1})  \right| \2 {E_{n,r}} \mid \cF_n \right]
  \leq \sup_{u\in S} | \varphi (u) | \Exp \left[ \left| X_{n+1}^{\nu-2} - X_n^{\nu-2} \right|
    \2 {E_{n,r}} \mid \cF_n \right] ,
\]
where $|  X_{n+1}^{\nu-2} - X_n^{\nu-2}  |  \2 {E_{n,r}}$ is bounded by a constant times $X_n^{\nu+r-3}$.
On the other hand,
\begin{align*} 
\Exp \left[  \left( \varphi (\eta_{n+1})  - \varphi (\eta_n) \right)  \2 {E_{n,r}} \mid \cF_n \right]
& = \Exp \left[   \varphi (\eta_{n+1})  - \varphi (\eta_n)  \mid \cF_n \right] + \eps (\xi_n) , \as,
\end{align*}
where $\lim_{x \to \infty} \sup_{u \in S} | \eps (x,u) | =0$, using
the fact that $\varphi$ is uniformly bounded and
$\Pr ( E^\rc_{n,r} \mid \cF_n ) \leq B_p X^{-rp}_n$, by the $s=0$ case
of~\eqref{eq:truncation}.  Here, by~\eqref{eq:pseudo-kernel},
\[
  \Exp \left[   \varphi (\eta_{n+1})  - \varphi (\eta_n)  \mid \cF_n \right] = \int_S \left( \varphi (v) - \varphi (\eta_n) \right) \Kss (X_n, \eta_n, \ud v), \as
\]
The result now follows if we note that
\begin{align*}
   V_{\nu, \varphi,r} ( \xi_n) &=
 \Exp \left[ \left( X_{n+1}^{\nu-2}  - X_n^{\nu-2}\right) \varphi (\eta_{n+1})  \2 {E_{n,r}} \mid \cF_n \right] \\
& \qquad + \Exp \left[   X_{n}^{\nu-2} \left(\varphi (\eta_{n+1})
      - \varphi (\eta_{n}) \right)  \2 {E_{n,r}} \mid \cF_n \right],
\end{align*}
and combine~\eqref{eq:asymptotically-markov} with the preceding estimates.
\end{proof}

The following result provides a bound when the increment is large.

\begin{lemma}
\label{lem:big-jump}
Suppose that~\eqref{ass:moments} holds for some $p>2$ and $q>0$. Then for any $\nu  \in (2-p, p \wedge q]$, there exists $r \in (\frac{1}{p-1},1)$ for which
\[
 \Exp [ | D_{\nu,\varphi, n} | \2 { E^\rc_{n,r} } \mid \cF_n ] = X_n^{\nu-2} \eps (\xi_n) , \as, \]
where $\lim_{x \to \infty} \sup_{u \in S} | \eps (x,u) | = 0$.
\end{lemma}
\begin{proof}
  The proof follows exactly that of Lemma~3.7 of~\cite{LoWade17},
  using the truncation estimates from Lemma~\ref{lem:truncation} and
  bounds for $H_{\nu,\varphi}(x,u)$ in terms of $x$; this relies on
  the fact that~$\varphi$ is uniformly bounded.
\end{proof}

\begin{proof}[Proof of Proposition~\ref{prop:strips-lyapunov}.]
  First note that~\eqref{ass:moments} shows that
  $\Exp H_{\nu,\varphi} (\xi_n) < \infty$ provided $\nu \leq p \wedge q$.
  The statement follows from combining~\eqref{eq:Dn-small-jump}
  with the estimates from Lemmas~\ref{lem:U-estimate},
  \ref{lem:V-estimate} and~\ref{lem:big-jump}, on choosing a suitable
  $r \in ( \frac{1}{p-1} , 1)$.
\end{proof}

We turn to the proof of Theorem~\ref{thm:strict-lamperti-moments}.
We will need
  the following two results that give conditions for existence and non-existence of moments of passage times. The formulations, taken from~\cite[\S 2.7]{MenPopWad16},
  are based closely on results of~\cite{aim}, and apply to an $\RP$-valued, adapted process $Y_n$ and its passage times $\lambda_y : = \min\{n \in \ZP \colon Y_n \leq y\}$, $y \in \RP$.
  
\begin{lemma}[Corollary 2.7.3 in~\cite{MenPopWad16}]
\label{lem:borrow-aim-finite}
  Let $Y_n$ be an integrable $\cF_n$-adapted stochastic process, taking values
  in an unbounded subset of $\RP$, with $Y_0 = y_0$ fixed. Suppose that there
  exist constants $\delta \in (0,\infty)$, $y \in (0,\infty)$, and $a < 1$ such that for any $n\in\ZP$, 
  \begin{equation}
    \label{eq:existence-of-moments-drift}
    \Exp [Y_{n+1}-Y_n \mid \cF_n ]\leq -\delta
    Y_n^a, \text{ on } \{n<\lambda_y\}.
  \end{equation}
  Then $\Exp[\lambda^s_y]<\infty$ for
  any $s \in [0, (1-a)^{-1})$.
\end{lemma}   

\begin{lemma}[Theorem 2.7.4 in~\cite{MenPopWad16}]
\label{lem:borrow-aim-infinite}
Let $Y_n$ be an integrable $\cF_n$-adapted stochastic process, taking values
  in an unbounded subset of $\RP$, with $Y_0 = y_0$ fixed. 
  Suppose that there
  exist constants $y \in (0,\infty)$, $B \in \RP$ and $c \in \R$, such that for any $n\in\ZP$, 
  \begin{align}
    \label{eq:non-existence-drift}
      \Exp [Y_{n+1}-Y_n \mid \cF_n ] & \geq -\frac{c}{Y_n},  \text{ on } \{ Y_n > y \},\\
      \Exp [ (Y_{n+1}-Y_n)^2 \mid \cF_n ] & \leq B, \text{ on } \{Y_n> y\}.
    \label{eq:non-existence-variance}
  \end{align}
  Suppose, in addition, that for some $s_0>0$, the process $Y^{2
    s_0}_{n\wedge \lambda_y}$ is a submartingale. Then,
    for any $s>s_0$, 
  $\Exp[\lambda^s_y ] = \infty$ provided $y_0 > y$.
\end{lemma}

\begin{proof}[Proof of Theorem~\ref{thm:strict-lamperti-moments}.]
The proof is
  divided into two parts; we first establish
  existence of moments (Theorem~\ref{thm:strict-lamperti-moments}\ref{thm:strict-lamperti-moments-a})
and then non-existence of moments  (Theorem~\ref{thm:strict-lamperti-moments}\ref{thm:strict-lamperti-moments-b}). Recall the definition of $\delta_\theta$ from~\eqref{eq:deltadef},
that  $g_\theta (u)  = 2e_{u}  + (2\theta -1)\sigma^2_{u}
  -\delta_{\theta}$, and that  $\varphi_{\theta} \in \Cb^+ (S)$ satisfies~\eqref{eq:choice-of-psi-moments}.
 We will use the process $H_{\nu, \varphi_{\theta}}(\xi_n)$, $\nu >0$,
defined via~\eqref{eq:strips-lyapunov}, 
in slightly different ways in the proofs of each of the two parts of the theorem.

\paragraph*{Proof of Theorem~\ref{thm:strict-lamperti-moments}\ref{thm:strict-lamperti-moments-a}}
Take  $\nu:=  2\theta \wedge p \wedge q$, where $\theta >0$, $p >2$, and $q \geq 2$ are as in the hypotheses
of Theorem~\ref{thm:strict-lamperti-moments}\ref{thm:strict-lamperti-moments-a}.
Set $Y_n := H_{\nu, \varphi_{\theta}}(X_n,\alpha_n) \in \RP$. 
Then, by choice of $\nu$ and $\varphi_{\theta}$ for which~\eqref{eq:choice-of-psi-moments} holds,
the coefficient in~$W_{\nu,\varphi_\theta}$ given by~\eqref{eq:strips-lyapunov-increment}
satisfies
\begin{align}
\label{eq:W-delta}
 2e_{u}  - (1-\nu) \sigma^2_{u} + \int_S (\varphi_{\theta} (v) - \varphi_{\theta} (u) ) \cK ( u, \ud v )
& =  2e_{u}  - (1-\nu) \sigma^2_{u} - g_\theta (u) \nonumber\\
& =    \delta_\theta  +    (\nu - 2\theta  ) \sigma^2_{u} \leq  \delta_\theta ,
\end{align}
since~$\nu \leq 2 \theta$
and $\sigma^2_u \geq 0$. 
By hypothesis, $\delta_\theta <0$, and, by~\eqref{eq:strips-lyapunov-increment}, there is an $x_1 \in \RP$ for which 
\[ \sup_{u \in S} W_{\nu,\varphi_\theta} (x, u) \leq   - \frac{\nu}{4} | \delta_\theta | x^{\nu-2}, \text{ for all } x \geq x_1.  \]
Since, by~\eqref{eq:strips-lyapunov},
$\sup_{u \in S} | H_{\nu, \varphi_{\theta}}(x,u) - x^\nu| = O (x^{\nu -2})$, it follows that
 there is an $x_2 \in \RP$ for which, setting  $a := \frac{\nu -2}{\nu} < 1$,
\[  W_{\nu,\varphi_\theta} (x, u) \leq   - \frac{\nu}{4} | \delta_\theta | \left( H_{\nu, \varphi_{\theta}}(x,u) \right)^{a}, \text{ for all } x \geq x_2 \text{ and all } u \in S.  \]
This, together with Proposition~\ref{prop:strips-lyapunov}, shows that~\eqref{eq:existence-of-moments-drift} holds with $Y_n = H_{\nu, \varphi_{\theta}}(X_n,\alpha_n)$, $a = \frac{\nu -2}{\nu}$, and $\delta = \frac{\nu}{4} | \delta_\theta | > 0$.
 Note that by the choice
  of $a$
  we have $(1 - a)^{-1}= \frac{\nu}{2} = \theta \wedge \frac{p}{2} \wedge \frac{q}{2}$.
  By Lemma~\ref{lem:borrow-aim-finite} we conclude that
  $\Exp[\lambda_y^s]< \infty$ for all
  $s \in [0, \theta \wedge \frac{p}{2}  \wedge \frac{q}{2})$ and all $y$ sufficiently large. 
  Moreover, by~\eqref{eq:strips-lyapunov} and the choices of $\varphi_\theta \in \Cb^+(S)$ and $\nu >0$, we have that
  $X_n^\nu \leq Y_n$, a.s. It follows that, with $\tau$ as defined at~\eqref{eq:def-passage-time}
  $\tau_{y^{1/\nu}} \leq \lambda_y$ and therefore
  $\Exp[\tau_r^s]< \infty$ for all $r$ large enough, which completes
  the proof of  Theorem~\ref{thm:strict-lamperti-moments}\ref{thm:strict-lamperti-moments-a}.

\paragraph*{Proof of Theorem~\ref{thm:strict-lamperti-moments}\ref{thm:strict-lamperti-moments-b}}
Take  $\nu:=  2\theta$, where $0< \theta < \frac{p}{2} \wedge \frac{q}{2}$ as in the hypotheses
of Theorem~\ref{thm:strict-lamperti-moments}\ref{thm:strict-lamperti-moments-b}.
Now define  $Y_n : =  (H_{\nu,\varphi_\theta}(\xi_n ))^{1/\nu}$,
where $\varphi_{\theta} \in \Cb^0(S)$ again satisfies~\eqref{eq:choice-of-psi-moments}.

We verify the hypotheses of Lemma~\ref{lem:borrow-aim-infinite} for this choice of $Y_n$; we examine the increment $Y_{n+1} - Y_n$.
First, observe that Taylor's theorem and the $x>1$ case of~\eqref{eq:strips-lyapunov} gives
  \begin{align}
    \label{eq:H-nu-H-1}
     (H_{\nu,\varphi}(x,u) )^{1/\nu} = x\Big( 1 +
    \frac{\nu}{2}\varphi(u) x^{-2} \Big)^{1/\nu}
   & = x + \frac{\varphi(u)}{2x} + O(x^{-3}), \nonumber\\
    & = H_{1,\varphi} (x,u) + O(x^{-3}),
  \end{align}
  as $x \to \infty$, uniformly in~$u \in S$. We claim that for every $\varphi \in \Cb^+(S)$, there exists $C = C(\varphi) \in (1, \infty)$ such that
\begin{equation}
\label{eq:H-scaled-bound} 
x \leq (H_{\nu,\varphi}(x,u) )^{1/\nu} \leq x + C, \text{ for all } x \in \RP , \, u \in S .\end{equation}
For $x \leq 1$, the bounds in~\eqref{eq:H-scaled-bound} are immediate from~\eqref{eq:strips-lyapunov}, while for $x>1$, they follow from~\eqref{eq:H-nu-H-1} and the fact that $\nu\varphi(u) \geq 0$ for all $u \in S$.
It follows from~\eqref{eq:H-scaled-bound}  that
\begin{equation}
\label{eq:Y-bound} | Y_{n+1}-Y_n | \leq | \Delta_n | + C , \as \end{equation}
Thus we verify~\eqref{eq:non-existence-variance} as a consequence of~\eqref{eq:Y-bound}
and~\eqref{ass:moments}. 

We next claim that
\begin{equation}
\label{eq:Y-drift}
\Exp [ Y_{n+1}-Y_n \mid \cF_n ] = \tW (\xi_n) ,
\text{ where }
 \tW (x,u)   =  \frac{1}{2x} \left[ \delta_\theta - (2\theta -1) \sigma_u^2 +  \eps_{x,u} \right]      ,
\end{equation}
and, as usual, $\lim_{x \to \infty} \sup_{u \in S} | \eps_{x,u} | =0$. 
We verify~\eqref{eq:Y-drift}. 
Choose $r \in (0,1)$ with
$r (p-1) > 1$ (recall that $p>2$). Then, 
\begin{equation}
    \label{eq:Y-split}
    \Exp [ Y_{n+1}-Y_n \mid \cF_n ]  
 =  \Exp [ ( Y_{n+1}-Y_n ) \2{E_{n,r}}  \mid \cF_n ]
  + \Exp [ ( Y_{n+1}-Y_n ) \2{E^\rc_{n,r}}  \mid \cF_n ],
  \end{equation}
  where $E_{n,r}$ is defined at~\eqref{eq:truncation-event}. 
A consequence
of~\eqref{eq:H-nu-H-1} is that
\[  ( Y_{n+1}-Y_n ) \2{E_{n,r}} =  \left( H_{1,\varphi}(\xi_{n+1}) - H_{1,\varphi}(\xi_{n}) \right)  \2{E_{n,r}} + O (x^{-3} ), \text{ on } \{ X_n \geq x \} ,
\]
uniformly over~$\eta_n \in S$.
Together with the $\nu =1$ cases of Proposition~\ref{prop:strips-lyapunov} and Lemma~\ref{lem:big-jump}, this implies that $\Exp [ ( Y_{n+1}-Y_n ) \2{E_{n,r}} \mid \cF_n ] = \tW (\xi_n)$,
where $\tW$ differs from~$W_{1,\varphi_\theta}$ 
given by~\eqref{eq:strips-lyapunov-increment} 
only in the $\eps_{x,u}$ term.
On the other hand, we have from~\eqref{eq:Y-bound} 
and an application of the $s =1$ and $s=0$ cases of Lemma~\ref{lem:truncation} that,
\[ \left| \Exp [ ( Y_{n+1}-Y_n ) \2{E^\rc_{n,r}}  \mid \cF_n ] \right| \leq  \Exp [ | \Delta_n | \2{E^\rc_{n,r}}  \mid \cF_n ] + C \Pr ( E^\rc_{n,r} \mid \cF_n ) \leq X_n^{-1} \eps (\xi_n),  \]
where $\lim_{x \to \infty} \sup_{u \in S} | \eps_{x,u} | =0$,
where we have used the fact that $r (p-1) > 1$. Applying~\eqref{eq:Y-split}
verifies~\eqref{eq:Y-drift},
appropriately redefining $\tW$ (the $\eps_{x,u}$ absorbs the additional error from the contribution on $E^\rc_{n,r}$). Thus~\eqref{eq:Y-drift} is proved. It follows that~\eqref{eq:non-existence-drift} holds
for this choice of $Y_n$, recalling from~\eqref{eq:H-scaled-bound} that $X_n \leq Y_n \leq X_n + C$, a.s.
Moreover, 
since $Y_n^{\nu} = H_{\nu,\varphi_\theta} (\xi_n)$ and $\nu = 2\theta < p \wedge q$, we have from Proposition~\ref{prop:strips-lyapunov}
that
$ \Exp [ Y_{n+1}^{\nu} - Y_n \mid \cF_n ] = W_{\nu,\varphi_\theta} (\xi_n) $, a.s.,
where, since $\nu = 2\theta$, by~\eqref{eq:strips-lyapunov-increment} and an analogous calculation to~\eqref{eq:W-delta},
\[ W_{\nu,\varphi_\theta} (x,u) =  \theta x^{2\theta -2} \left[ \delta_\theta + \eps_{x,u} \right] ,\]
and, under the hypotheses of Theorem~\ref{thm:strict-lamperti-moments}\ref{thm:strict-lamperti-moments-b},
we have $\delta_\theta >0$. This verifies that $Y_{n \wedge \lambda_y}^{2\theta}$ is a submartingale
for a sufficiently large~$y \in \RP$.
We have thus shown that $Y_n$ satisfies all the hypotheses of Lemma~\ref{lem:borrow-aim-infinite}
with $s_0 = \theta$, and this establishes the conclusion of Theorem~\ref{thm:strict-lamperti-moments}\ref{thm:strict-lamperti-moments-b},
recalling once more that $X_n \leq Y_n \leq X_n + C$.
\end{proof}

\subsection{The Lamperti regime} 
\label{sec:general-lamperti-proofs}

The aim of this section is to prove Theorems~\ref{thm:strips-general-lamperti} and~\ref{thm:moments-general}.
We do so by mapping the process back to the \emph{strict} Lamperti regime,
identifying the appropriate parameters, and verifying the conditions of, respectively, Theorems~\ref{thm:strips-strict-lamperti}
and~\ref{thm:strict-lamperti-moments}
for the transformed process.

Our transformation will be achieved by a collection of horizontal
shifts that eliminate the constant-order terms of the drifts,
following a similar idea to that in Section~5 of~\cite{LoWade17} for
the simpler case when~$S$ is finite.  For a given $\phi \in \Cb^+(S)$,
the transformation is
\begin{equation}\label{Hshift}
T_\phi : \Sigma \to \Sigma, \text{ given by } 
T_\phi(x,u ) =   (x + \phi(u),u) \text{ for all } (x,u) \in \Sigma.
\end{equation}

\begin{theorem}
\label{thm:general-to-strict}
Suppose that~\eqref{ass:non-confinement}, \eqref{ass:kernel},
and~\eqref{ass:asymptotically-markov+} hold, and
that~\eqref{ass:moments} holds with~$p>2$ and $q>0$. Suppose also
that~\eqref{ass:lamperti} and~\eqref{ass:detail-drifts} hold. There
exists a unique $\psi \in \Cb^+ (S)$ with
$\inf_{u \in S} \psi (u) = 0$ and
$\int_S ( \psi (u) - \psi (v) ) \cK ( u, \ud v ) = d_u$ for all
$u \in S$.  Then the time-homogeneous Markov process
$\txi = ( \txi_n, n \in \ZP)$ defined by
$\txi_n := T_\psi ( \xi_n ) \in \Sigma$ satisfies the hypotheses of
Theorem~\ref{thm:strips-strict-lamperti}. In particular, the moments
conditions in~\eqref{eq:generalized-lamperti-regime} are satisfied for
$\txi$ with coefficients $\te, \tsigma \in \Cb(S)$ given by
\begin{align}
\label{eq:new-e}
\te_u &= e_u + \int_S \left( \psi(v)-\psi(u) \right) \meas_u ( \ud v ); \\
\label{eq:new-sigma}
\tsigma_u &= \sigma^2_u + 2  \int_S  \lambda_u (v) \psi (v)  \cK ( u , \ud v ) + \int_S \left( \psi^2 (v) - \psi^2 (u) \right) \cK ( u , \ud v ).
\end{align}
\end{theorem}
\begin{proof} 
By Proposition~\ref{prop:fredholm}, there exists a unique $\psi \in \Cb^+ (S)$ with $\inf_{u \in S} \psi (u) = 0$ and
\begin{equation}\label{eq:psid}
\int_S ( \psi (u) - \psi (v) ) \cK ( u, \ud v )   =  d_u , \text{ for all } u \in S.
\end{equation}
Note then that $\te, \tsigma$ defined by~\eqref{eq:new-e} and~\eqref{eq:new-sigma}
are continuous, as claimed.
This follows from Lemma~\ref{lem:continuity-against-K} using the 
continuity of $\psi$, $e$, $\sigma^2$,  of $u \mapsto \meas_u$
from~\eqref{ass:asymptotically-markov+},
of $u \mapsto \lambda_u$
from~\eqref{ass:detail-drifts}, and the continuity of $\cK$ from~\eqref{ass:kernel}. 

Since $T_\psi$ is one-to-one and measurable (continuous, even), $\txi$ is time-homogeneous and Markov, and the non-confinement for $\txi$
is inherited from non-confinement~\eqref{ass:non-confinement} for $\xi$, since $\psi$ is bounded. Similarly the moments bound~\eqref{ass:moments} 
carries over easily. 
It remains to verify that~\eqref{eq:generalized-lamperti-regime} holds with the claimed coefficients.

Write $\txi_n = (\tX_n , \teta_n)$ in components.
Then, on $\{ X_n \geq x_B \}$,
\begin{align*}
\Exp [ \tX_{n+1} - \tX_n \mid \cF_n ]
= \Exp [ X_{n+1} - X_n \mid \cF_n ] + \Exp  [ \psi (\eta_{n+1} ) - \psi (\eta_n ) \mid \cF_n ] = \tmu_1 ( \txi_n ) , \as, \end{align*}
for a measurable $\tmu_1 : \Sigma \to \R$, 
where, with $\Kss$ as defined at~\eqref{eq:pseudo-kernel}, 
\[
  \tmu_1 (x, u) = \mu_1 ( x - \psi (u) ,u ) + \int_S \left( \psi (v) - \psi(u) \right) \Kss ( x - \psi (u) , u , \ud v ) .
\]
Since $\psi$ is bounded,~\eqref{ass:asymptotically-markov+} and~\eqref{eq:generalized-lamperti-regime} show that, for some $\eps_{x,u}$ with $\lim_{x \to \infty} \sup_{u \in S} | \eps_{x,u} | = 0$,
\begin{align*}
 \tmu_1 (x, u) & = d_u + \frac{e_u}{x} + \int_S \left( \psi (v) - \psi(u) \right) \left[  \cK (u, \ud v ) + \frac{\meas_u ( \ud v)}{x} \right]
+ \frac{\eps_{x,u}}{x}  \\
& = \frac{1}{x} \left[ e_u +  \int_S \left( \psi (v) - \psi(u) \right) \meas_u ( \ud v) + \eps_{x,u} \right] ,\end{align*} 
where the second equality follows from~\eqref{eq:psid}.
Similarly, $\Exp [ ( \tX_{n+1} - \tX_n )^2 \mid \cF_n ] = \tmu_2 (\txi_n)$, on $\{ X_n \geq x_B\}$.
Observe that 
\begin{align*}
\Exp [ ( \tX_{n+1} - \tX_n )^2 \mid \cF_n ]
&= \mu_2 ( \xi_n ) + 
2 \Exp [ ( X_{n+1} - X_n ) ( \psi (\eta_{n+1} ) - \psi (\eta_n ) ) \mid \cF_n ] \\
& {} \qquad {} + \Exp  [ ( \psi (\eta_{n+1} ) - \psi (\eta_n ) )^2 \mid \cF_n ]   .\end{align*}
By disintegration~\cite[Thm.~6.4, p.~108]{kall}, for measurable $f : \Sigma \to \R$,
\[
  \Exp [ f ( X_{n+1} , \eta_{n+1} ) \mid \cF_n  ] = \int_S \int_{\RP} f ( y , v ) \Kss ( \xi_n , \ud v ) \cL
  (\xi_n , v ,  \ud y ) , \text{ a.s.},
\]
where $\cL ( x, u , v , \, \cdot \, )$ is a regular conditional distribution for $X_{n+1}$ given $(\xi_n, \eta_{n+1} ) = (x,u,v)$.
In this notation, the conditional drifts appearing in~\eqref{eq:mu-disintegration} are given by 
\[
 \cmu_1 (x, u , v ) = \int_{\RP}  (y - x) \cL  ( x, u , v,  \ud y ) . \]
Hence, a.s.,
\begin{align*} \Exp [ ( X_{n+1} - X_n ) ( \psi (\eta_{n+1} ) - \psi (\eta_n ) ) \mid \cF_n ] & = 
\int_S \psi (v) \cmu_1 ( \xi_n , v ) \Kss ( \xi_n , \ud v ) - \psi( \eta_n )  \mu_1 ( \xi_n ) .\end{align*}
It follows that 
\begin{align}
\label{eq:tmu2}
\tmu_2 ( x,u)  & = \mu_2 ( x -\psi (u), u) 
+ 2 \int_S \psi (v) \cmu_1 ( x - \psi (u) , u , v ) \Kss ( x - \psi (u), u , \ud v ) 
\nonumber\\
& {} \qquad {} - 2\psi( u )  \mu_1 ( x - \psi (u), u )
+ \int_S   ( \psi (v) - \psi (u) )^2  \Kss ( x - \psi (u), u , \ud v )   . \end{align}
Note that
\begin{align*}
& {} \quad {} \left|  \int_S \psi (v) \cmu_1 ( x   , u , v ) \Kss ( x , u , \ud v ) 
- \int_S \psi (v)  \lambda_u (v) \cK ( u , \ud v )  \right| \\
& = \left| \int_S \psi (v)  \left( \cmu_1 ( x   , u , v ) - \lambda_u (v)  \right)  \Kss ( x , u , \ud v ) 
+ \int_S \psi (v)   \lambda_u (v) \left[ \Kss ( x , u , \ud v )  -  \cK ( u , \ud v ) \right] \right| \\
& \leq \| \psi \| \sup_{u, v \in S} \left| \cmu_1 ( x   , u , v )  - \lambda_u (v) \right|   
+  \| \psi \| \sup_{u,v \in S} | \lambda_u  (v) | \sup_{u \in S} \left\| \Kss ( x , u , \, \cdot \,  )  -  \cK ( u , \, \cdot \, ) \right\|_\tv 
  .\end{align*} 
Since $u \mapsto \sup_{v \in S} \lambda_u (v)$
is continuous, by~\eqref{ass:detail-drifts},   compactness shows that $\sup_{u,v \in S} | \lambda_u  (v) |< \infty$.
Hence from~\eqref{ass:detail-drifts} and~\eqref{eq:asymptotically-markov} we conclude that 
\begin{equation}
  \label{eq:detail-drifts-estimate}
  \lim_{x \to \infty} \sup_{u \in S} \left|  \int_S \psi (v) \cmu_1 ( x   , u , v ) \Kss ( x , u , \ud v ) 
  - \int_S \psi (v)  \lambda_u (v) \cK ( u , \ud v )  \right| = 0.
\end{equation}
It follows from~\eqref{eq:tmu2} with~\eqref{eq:detail-drifts-estimate} and~\eqref{eq:generalized-lamperti-regime}
that, for $\eps_{x,u}$ with $\lim_{x \to \infty} \sup_{u \in S} | \eps_{x,u} | = 0$, 
\begin{align*}
\tmu_2 ( x,u)  = \sigma^2_u  - 2 d_u \psi (u) + 2 \int_S     \psi (v)  \lambda_u (v) \cK ( u , \ud v )  
  + \int_S   ( \psi (v) - \psi (u) )^2  \cK ( u , \ud v )  + \eps_{x,u} .\end{align*}
The final observation is that
\begin{align*}
&  \int_S   ( \psi (v) - \psi (u) )^2  \cK ( u , \ud v ) -2 d_u    \psi (u)  \\
& {} \quad {}	= \int_S   ( \psi^2 (v) - \psi^2 (u) )  \cK ( u , \ud v )
 - 2 \psi (u) \left[ \int_S \left( \psi (v) -\psi (u) \right)  \cK ( u , \ud v ) + d_u \right] ,
 \end{align*}
and the term in square brackets vanishes, again by~\eqref{eq:psid}.
\end{proof}
  
\begin{proof}[Proof of Theorem~\ref{thm:strips-general-lamperti}.]
Under the conditions of  Theorem~\ref{thm:strips-general-lamperti}, 
 Theorem~\ref{thm:general-to-strict} shows that
 the transformed process $\txi_n$ defined therein satisfies the hypotheses of 
 Theorem~\ref{thm:strips-strict-lamperti}, with coefficients given by~\eqref{eq:new-e} and~\eqref{eq:new-sigma};
 note that $\xi_n$ is (positive, null) recurrent if and only if $\txi_n$ is (positive, null) recurrent.
 To obtain the expression for $\tdelta_\theta$ in~\eqref{eq:deltadef-general},
we note that, by stationarity of~$\pii$,
\begin{equation}\label{eq:general-simplifications}
  \int_S \int_S   ( \psi^2 (v) - \psi^2 (u) )  \cK ( u , \ud v ) \pii (\ud u ) = \int_S \psi^2 (v) \pii (\ud v) - \int_S \psi^2 (u) \pii (\ud u ) = 0,
\end{equation}
and, since $\meas_u (S) = 0$,
\begin{equation}
\nonumber
  \int_S \int_S \psi (u) \meas_u (\ud v) \pii (\ud u ) =  0 .
\end{equation}
In particular, we have from~\eqref{eq:general-simplifications} that  $\tsigma_u$~given by~\eqref{eq:new-sigma} satisfies
\[ 0 \leq \int_S  \tsigma_u \pii (\ud u ) = \int_S  \sigma^2_u \pii (\ud u )  + 2 \int_S \int_S \lambda_u (v) \psi (v) \cK (u, \ud v) \pii (\ud u )  , \]
which also implies that $\tdelta_\theta$ is non-decreasing in $\theta$.
Also note that, by~\eqref{ass:detail-drifts},
$\int_S \lambda_u (v) \cK (u,\ud v) = \lim_{x \to \infty} \int_S \cmu_1 (x, u, v) \Kss (x,u ,\ud v) = d_u$, and so 
\begin{equation}
\nonumber
  \int_S \int_S \lambda_u (v) \cK(u,\ud v) \pii (\ud u ) =   \int_S  d_u  \pii (\ud u ) = 0,
\end{equation}
which implies that the terms defined in~\eqref{eq:deltadef-general} are invariant under translation of $\psi$. 
This completes the proof of the recurrence classification.
The series representation for~$\psi$ given at~\eqref{eq:psi-series}
follows from Proposition~\ref{prop:potential}.
\end{proof}
  
\begin{proof}[Proof of Theorem~\ref{thm:moments-general}.]
Similarly to the preceding proof,  Theorem~\ref{thm:general-to-strict}
shows that we may apply Theorem~\ref{thm:strict-lamperti-moments} to the transformed process $\txi_n$
to obtain the result, noting that, since $| \xi_n - \txi_n| \leq C$, a.s., for some constant $C< \infty$ and all $n \in \ZP$, existence of a given passage-time moment is equivalent for the two processes.
\end{proof}

\section{Proofs for the stochastic billiards model}
\label{sec:billiards-proofs}

\subsection{Displacement estimates}
\label{sec:billiards-displacement}
 
Recall from Section~\ref{sec:billiards-setup} the construction of the
stochastic billiards process, 
that the reflection kernel~$\Kb$ is a Markov kernel on
  $S = [ -\frac{\pi}{2}, \frac{\pi}{2} ]$, 
and that
$S_0 = [ - \theta_0, \theta_0]$, for
  $\theta_0 \in (0, \frac{\pi}{2})$. 
  Each of the assumptions~\eqref{ass:ellipticity}--\eqref{ass:density-smoothness} plays a role in one or more of the subsidiary results in this section. 

Although~\eqref{ass:ellipticity} ensures that outgoing angles are confined to~$S_0$, the perturbation introduced by the curvature of the domain means that incoming angles can only be confined (for large enough horizontal coordinate) to a bigger interval; thus in this section we need estimates for our functions on angles over an interval $S_1$ containing $S_0$ in its interior. For this reason, we take
$\theta_1 \in ( \theta_0, \frac{\pi}{2})$, and
$S_1 = [ -\theta_1, \theta_1]$, so that $S_0 \subset S_1 \subset S$.
With~$\Lambda$ defined at~\eqref{eq:Lambda-def}, define
$D : \RP \times S_1 \to \R$ by
\begin{equation}\label{eq:D-def}
  D ( z, \oa ) := \Lambda_1 (z, 1 , \oa ) - z .
\end{equation}
Observe
 that $\ell_t (z, -j, \beta)$ is the reflection of $\ell_t (z, j, \beta)$
in the $x$-axis, which, with the reflection symmetry of $\partial \gR$, means that
$\lambda (z, -j, \beta) = \lambda (z, j, \beta)$ and $\Lambda_1 (z, j, \beta) = \Lambda_1 (z, -j, \beta)$.
Hence~\eqref{eq:billiards-construction-location} shows that $Z_{n+1} - Z_n = D (Z_n , \beta_n )$ for any  $Z_n \geq 0$,
so we can interpret~$D$ as the horizontal displacement of the billiards process.

Part~\ref{lem:displacement-a} of the next result states that, outside of a bounded set, successive collisions
occur on opposite sides of the boundary, part~\ref{lem:displacement-b} is a displacement bound, 
while~\eqref{eq:displacement} and~\eqref{eq:displacement-derivative}
give sharp expansions for $D ( z, \oa )$ and its $\beta$-derivative.

\begin{lemma}
\label{lem:displacement}
Suppose that~\eqref{ass:ellipticity} holds. 
There exist constants $C, z_0 \in \RP$ such that
\begin{enumerate}[label=(\alph*)]
\item
\label{lem:displacement-a}
for all $z \geq z_0$, all $\beta \in S_1$, and any $j \in \{-1,+1\}$,
$j \Lambda_2 ( z, j, \beta ) < 0$;
\item 
\label{lem:displacement-b}
for all $z \geq z_0$, $\sup_{\oa \in S_1} | D(z, \oa) | \leq C z^\gamma$.
\end{enumerate}
  Moreover, as $z \to \infty$,
\begin{align}
\label{eq:displacement}
 \sup_{\oa \in S_1} \left| D(z, \oa)  - \left[  2 z^{\gamma} \tan \oa + 2\gamma z^{2\gamma - 1} + 4 \gamma z^{2\gamma-1} \tan^2 \oa  \right] \right| =  O(z^{3\gamma -2}); \\
\label{eq:displacement-derivative}
 \sup_{\beta \in S_1} \left| \frac{\partial}{\partial \beta} D (z, \beta) - 2 z^{\gamma} \sec^2 \beta \right| =  O(z^{2\gamma - 1}).
\end{align}
\end{lemma}
\begin{proof}
The main part of the argument is essentially that given on~\cite[pp.~284--7]{MenPopWad16}, but, since there are a couple of minor errors there,
and~\eqref{eq:displacement-derivative} is new, we outline the main steps. Write $D:= D(z,\oa)$ for convenience.

Recall the definition of $\theta(z)$ from~\eqref{eq:normalangle}
and that, by~\eqref{eq:region-def}, 
$\partial \gR = \{(z,h(z,j)) \colon z \in \RP, \, j \in \{-1,1\} \}$ where $h(z,j) = j z^\gamma$.
Choose $\theta_2 \in (\theta_1, \frac{\pi}{2})$ and $\eps \in (0, \frac{\pi}{2} - \theta_2)$.
Since $( \partial / \partial z) h(z,1)= \gamma z^{\gamma -1} \to 0$ as $z \to \infty$, we may choose $z_1$ large enough
so that, for all $z \geq z_1$,
(i) $\abs{\gamma z^{\gamma -1} }< \eps$, and (ii) $| \theta (z) + \beta | < \theta_2$
for all $\beta \in S_1$. Since $| \theta (z) + \beta |<\pi/2$ it follows that, for
all~$z$ sufficiently large,
 $L(z,j,\beta)$ meets $\partial D_\gamma$ at the opposite boundary, giving~\ref{lem:displacement-a}. 
Furthermore, by assumption~\eqref{ass:ellipticity}, there is a constant $C \in \RP$ such that, for all $z$ sufficiently large,
the ray from $(z, h(z,j) )$ meets the opposite boundary at a point $(z', h(z',-j))$
with $D := z' - z$ satisfying $|D | \leq C z^\gamma$; this gives the bound in part~\ref{lem:displacement-a}.

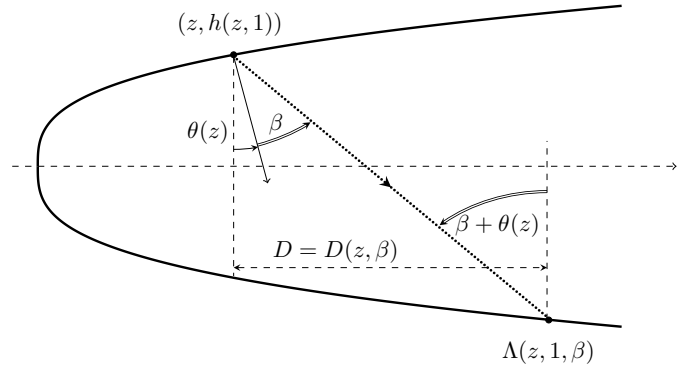
\begin{figure}[!ht]
  \centering
\scalebox{0.8}{\begin{tikzpicture}[scale = 1.4]
	\draw[black, line width = 0.40mm]   plot[smooth,domain=0:1.9,samples=500] (\x^3,    \x);
	\draw[black, line width = 0.40mm]   plot[smooth,domain=0:1.9,samples=500] (\x^3,  -  \x);
    
\draw[->,>=stealth] (2.3,.2) arc (-90:-79:1.5);
\draw[double,->,>=stealth](2.58,.25) arc (-75:-58:2.3);
\draw[double,->,>=stealth] (5.98,-.3) arc (90: 125:2.2);
		
		  \draw[decoration={
    markings,
    mark=at position 0.5 with {\arrow{stealth}}}, postaction={nomorepostaction,decorate}, 	
	black, >=stealth, line width=1.3pt, line cap=round, dash pattern=on 0pt off 1.5\pgflinewidth] (2.3, 1.3)-- (6, -1.8);
		
		\draw[black,dashed,->]  (-0.3,0)--(7.5,0);
\draw[black,dashed]  (2.3,1.3)--(2.3,- 1.3);
    \draw[black, ->] (2.3,1.3)--(2.7,- .2);
    		\draw[black,dashed,<->,>=stealth]  (2.3,-1.2)--(5.98,-1.2);

    \draw[black,dashed] (5.98,-1.82)--(5.98, .3);

\filldraw (2.3,1.32) circle (1.0pt);
\filldraw (6,-1.82) circle (1.0pt);

    \node at (2.26, 1.68) {{\small $(z,h(z,1))$}};
    \node at (2, 0.4) {\small $\theta(z)$};
    \node at (2.8, 0.5) {\small $\oa$};

    \node at (6, -2.2) {\small $\Lambda ( z,1,\oa)$}; 
    \node at (5.4, -0.7) {\small $\oa +\theta(z)$};
        \node at (3.5, -1) {\small $D = D(z,\beta)$}; 
  \end{tikzpicture}}
\caption{An illustration of the horizontal increment between successive collisions.}
\label{F:increment}
\end{figure}

Suppose $z \geq z_0$ as above. Some geometry (see Figure~\ref{F:increment}) 
shows that 
\begin{align}
\label{eq:D-implicit}
D    = (z^\gamma  + (z+D)^{\gamma}) \tan (\theta(z) + \oa)   
               = z^{\gamma} \left[ 1 + (1 + D/z)^{\gamma} \right]\tan(\theta(z) + \oa ),
  \end{align}
	since $| \oa + \theta (z) | < \pi/2$ for $\oa \in S_1$.
	Because $|D| \leq C z^\gamma$, we can use a Taylor's theorem expansion in~\eqref{eq:D-implicit} 
to obtain, uniformly in $\beta \in S_1$, 
\begin{equation}
\nonumber
D = z^\gamma \left[ 2 + \gamma z^{-1} D + O ( z^{2\gamma -2} ) \right] \frac{ \tan \oa + \tan \theta (z) }{1 - \tan \oa \tan \theta (z) } .
\end{equation}
Since $\tan \theta (z) = \gamma z^{\gamma-1}$, and using the fact that $D/ z^\gamma$ and $\tan \beta$ are both $O(1)$,
rewriting the fraction above as $(\tan \oa  + \gamma z^{\gamma-1})(1 + \gamma z^{\gamma-1} \tan \oa +O(z^{2\gamma -2}))$ we obtain 
\begin{align*}
  D &= z^\gamma(2 + \gamma D z^{-1}) (\tan \oa + \gamma z^{\gamma -1})( 1 + \gamma z^{\gamma -1} \tan \oa) + O(z^{3\gamma -2})  \\
  &= 2 z^\gamma \tan \oa + \gamma D z^{\gamma -1 } \tan \oa + 2 \gamma z^{2\gamma -1} (1+ \tan^2 \oa ) + O ( z^{3\gamma -2} ) .
\end{align*}
(The above display corrects the corresponding display at the bottom of p.~286 in~\cite{MenPopWad16}, which has an erroneous extra term.)
Re-arranging the above display we get
\begin{align*} 
D & = \frac{ 2 z^\gamma \tan \oa  + 2 \gamma z^{2\gamma -1} (1+ \tan^2 \oa ) }{ 1 - \gamma  z^{\gamma -1 } \tan \oa }  + O ( z^{3\gamma -2} ) \\
& = \left(  2 z^\gamma \tan \oa  + 2 \gamma z^{2\gamma -1} (1+ \tan^2 \oa ) + O ( z^{3\gamma -2} ) \right) \left( 1 + \gamma  z^{\gamma -1 } \tan \oa
+ O ( z^{2\gamma -2} ) \right) ,\end{align*}
which yields~\eqref{eq:displacement}.
Finally, note that~\eqref{eq:D-implicit} and the implicit function theorem  show that
$D(z,\beta)$ is differentiable in $\beta$. Writing
$D' = (\partial/\partial \beta) D$, we obtain from~\eqref{eq:D-implicit} that
\begin{equation}
\nonumber
 D' 
 =  (z^\gamma  + (z+D)^{\gamma}) \sec^2 (\theta(z) + \beta) + \gamma (z+D)^{\gamma-1} D' \tan (\theta(z) + \beta)   ,
\end{equation}
and hence
\begin{equation}
\nonumber
  D' = \frac{  (z^\gamma  + (z+D)^{\gamma}) \sec^2 (\theta(z) + \beta)}{1 - \gamma (z+D)^{\gamma-1}  \tan (\theta(z) + \beta) }.
\end{equation}
Since $|D| = O (z^\gamma)$, $\sec^2 (\theta(z) + \beta) = \sec^2 \beta + O (z^{\gamma -1})$,
and $\tan  (\theta(z) + \beta) = O(1)$ we obtain~\eqref{eq:displacement-derivative}.
\end{proof}

Lemma~\ref{lem:displacement} gives control over the increments
of the billiards process outside a bounded set,
which, after a suitable transformation of the process
(see Section~\ref{sbsl})
 will more than suffice to check the condition~\eqref{eq:lamperti-p-moments} in~\eqref{ass:moments}. Near the origin,
we must verify the weaker condition~\eqref{eq:q-bound}. This is the purpose of the next result.

\begin{lemma}
\label{lem:bound-near-apex}
Suppose that~\eqref{ass:ellipticity} and~\eqref{ass:billiards-density} hold.
For any $z_0 \in \RP$,
there is a constant $C \in \RP$ such that
for every $r > 0$, 
$\Pr ( Z_{n+1} > r \mid \cF_n ) \leq C/r^{1-\gamma}$ on $\{ Z_n \leq z_0\}$.
\end{lemma}
\begin{proof}
For any $z \leq z_0$, from point $(z, h(z,j))$
the set of angles~$\beta$ that give $\Lambda_1 (z, j, \beta ) \geq r$
is contained in an interval $I(z,j) \subset S$ with $| I (z,j) | \leq C r^{\gamma -1}$.
Then, since $\Pr ( \beta_{n} \in B \mid \cF_n ) = \Kb ( \alpha_n , B)$, a.s.,
\begin{equation}
\nonumber
  \Pr ( Z_{n+1} > r \mid \cF_n ) \leq \cK ( \alpha_n , I (Z_n, \chi_n ) ) \leq C  r^{\gamma -1} \sup_{\alpha, \beta \in S} \kappa (\alpha, \beta) ,
~ \text{on} ~  \{ Z_n \leq z_0\},
\end{equation}
which gives the result, since $\kappa$ is uniformly bounded under~\eqref{ass:billiards-density}.
\end{proof}

The following fact
will be used to show that the billiards process is non-confined, 
and also concerns the implications of our results for the continuous-time
version of the stochastic billiards process (see Remark~\ref{rems:billiards-real}\ref{rems:billiards-real-b}).

\begin{lemma}
\label{lem:lower-bound}
Suppose that~\eqref{ass:ellipticity} holds. 
There exists $\eps_0 >0$ such that $\lambda (z,j,\beta) \geq \eps_0$
for all $z \in \RP$, $j \in \{-1,+1\}$, and all $\beta \in S_0$.
\end{lemma}
\begin{proof} 
Due to the smoothness of $\partial\gR$, $\inf_{\beta \in S_0} \lambda (z,j,\beta) > 0$ everywhere. 
Moreover, $(z,\beta) \mapsto \lambda (z,j,\beta)$ is continuous over $(z, \beta ) \in \RP \times S_0 \setminus \{ (0,0) \}$. Also, for any $z, j$, $\inf_{\beta \in S_0} \lambda (z,j,\beta)$
is attained at $\beta \in \{ -\theta_0, \theta_0 \}$. Thus $z \mapsto \inf_{\beta \in S_0} \lambda (z,j,\beta)$ is continuous over $z \in \RP$, and tends to $\infty$ as $z \to \infty$. Hence $\inf_{z \in \RP} \inf_{\beta \in S_0} \lambda (z,j,\beta) >0$.
\end{proof}

\subsection{Translation to the half-strip model}
\label{sbsl}

Define $X_n := Z_n^{1-\gamma}$, a rescaling of the horizontal displacement. 

\begin{lemma}
\label{lem:billiards-to-strip-markov}
Suppose that~\eqref{ass:ellipticity}--\eqref{ass:loopable} hold.
The process $(X_n, \alpha_n)$ is a time-homogeneous Markov process on $\Sigma := \RP \times S$, satisfying the following.
\begin{enumerate}[label=(\alph*)]
\item
\label{lem:billiards-to-strip-markov-a}
There exist $x_B, B \in \RP$ such that $\Pr ( | X_{n+1} - X_n | \leq B \mid \cF_n ) = 1$
on $\{ X_n \geq x_B\}$. 
\item
\label{lem:billiards-to-strip-markov-b}
There exists $C \in \RP$ such that $\Pr ( X_{n+1} > r \mid \cF_n ) \leq C/r$ on $\{ X_n \leq x_B \}$.
\item
\label{lem:billiards-to-strip-markov-c}
There is non-confinement: $\limsup_{n \to \infty} X_n = +\infty$, a.s.
\end{enumerate}
\end{lemma}
\begin{proof}
We already observed below~\eqref{eq:D-def} that $Z_{n+1} = \Lambda_1 (Z_n, 1 , \beta_n)$ is 
a function of $Z_n, \beta_n$ only. 
On the other hand, $\Lambda_2 (z, -j, \beta ) = - \Lambda_2 (z, j, \beta)$, which means
that the sign of $j \Lambda_2 (z, j, \beta)$ is the same for~$j \in \{-1, +1\}$.
Hence $\Theta (z, j, \beta)$ defined by~\eqref{eq:inangle} does not depend on~$j$,
and $\alpha_{n+1} = \Theta (Z_n, 1, \beta_n)$ given at~\eqref{eq:billiards-construction-angle}
is a function of $Z_n, \beta_n$ only. Hence $(Z_n, \alpha_n)$ is a time-homogeneous Markov process on $\Sigma := \RP \times S$.
The same is true for $(X_n, \alpha_n)$, since $(z,\alpha) \mapsto (z^{1-\gamma},\alpha)$
is a bijection for $\gamma \in (0,1)$.

For statement~\ref{lem:billiards-to-strip-markov-a},
Taylor's theorem applied to the function $z \mapsto z^{1-\gamma}$ shows that
\begin{equation}
\label{eq:simple-taylor}
 (z+D)^{1-\gamma} - z^{1-\gamma} = z^{1-\gamma} \left[ \left(1 + \frac{D}{z} \right)^{1-\gamma} - 1 \right]
= (1-\gamma) D z^{-\gamma} (1+o(1) ) ,
\end{equation}
if $|D| = o(z)$ as $z \to \infty$.
Lemma~\ref{lem:displacement} shows that
$|Z_{n+1} - Z_n | \leq C Z_n^\gamma$ on $\{ Z_n \geq z_0 \}$,
and so~\eqref{eq:simple-taylor}
with $z = Z_n$ and $D = Z_{n+1} - Z_n$ implies that
 $| X_{n+1} - X_n | \leq 2 C (1-\gamma)$, on $\{ X_n \geq x_B \}$ for $x_B$ sufficiently large.
For statement~\ref{lem:billiards-to-strip-markov-b},
it follows directly from Lemma~\ref{lem:bound-near-apex} 
that $\Pr ( X_{n+1} > r \mid \cF_n ) = O (1/r)$ on $\{ X_n \leq x_B \}$.
For statement~\ref{lem:billiards-to-strip-markov-c},
we have from~\eqref{ass:ellipticity} and Lemma~\ref{lem:lower-bound} that there is a $z_0 \in (0,\infty)$ such that
$\Pr ( X_{n+1} \geq 2 z_0 \mid \cF_n ) = 1$ on $\{ X_n \leq z_0 \}$,
while assumption~\eqref{ass:loopable}
ensures that there is $\eps>0$ for which $\Pr (X_{n+1} - X_n \geq \eps \mid \cF_n) \geq \eps$ on $\{ X_n > z_0 \}$.
The combination of these two facts implies  $\limsup_{n \to \infty} X_n = +\infty$, by,
for instance, Proposition~3.3.4 of~\cite{MenPopWad16}.
\end{proof}

For $z \in \RP$ and $\beta \in S$, define
\begin{equation}
\label{eq:alpha-beta} 
b_z (\beta ) := \theta ( z ) + \theta ( \Lambda_1 (z,1,\beta) ) .
\end{equation}
From Lemma~\ref{lem:displacement}\ref{lem:displacement-a}, \eqref{eq:inangle} and~\eqref{eq:alpha-beta}, we have that for
all $\beta \in S_1$ and all~$z \in \RP$ sufficiently
large~$\Theta (z, j , \beta ) = \beta + b_z (\beta)$ does not depend
on~$j$; for ease of notation, we write
$\Theta_z ( \beta ) := \beta + b_z (\beta)$ and
$\Theta'_z (\beta) := (\partial/\partial \beta) \Theta_z (\beta)$.  We
will need the following basic properties of~$\Theta_z$.

\begin{lemma}
\label{lem:theta-properties}
 Let $\theta_1 \in ( \theta_0, \frac{\pi}{2})$ be arbitrary and
  recall  $S_1 = [ -\theta_1, \theta_1]$.
Then, as $z \to \infty$,
\begin{equation}
  \label{eq:theta_unif}
   \sup_{\beta \in S_1} \left| z^{1-\gamma} ( \Theta_z (\beta) - \beta
    ) - 2 \gamma \right| \to 0, ~\text{and}~ \sup_{\beta \in S_1}
  \left| \Theta'_z (\beta) - 1 \right| = O ( z^{2\gamma -2} ).
\end{equation}
Moreover there exists a differentiable $T_z : S_1 \to S $ such that,
for all $z$ sufficiently large, $\Theta_z (T_z (\beta)) = \beta$
for every $\beta \in S_1$.  The
function~$T_z$ satisfies, as $z \to \infty$,
\begin{equation}
  \label{eq:Tzsup}
  \sup_{\beta \in S_1}
  \left| z^{1-\gamma}(T_z (\beta)-\beta) + 2\gamma \right|\to 0,
  ~\text{and}~ \sup_{\beta \in S_1}
  \left| T'_z (\beta) - 1 \right| = O ( z^{2\gamma -2} ).
\end{equation}
\end{lemma}
\begin{proof}
  Note that~\eqref{eq:D-def} and Lemma~\ref{lem:displacement}\ref{lem:displacement-b}  imply that
    $\sup_{\beta \in S_1} \abs{\Lambda_1(z,1,\beta) - z} = O(z^\gamma)$, as $z \to \infty$.
  Then, by~\eqref{eq:normalangle}, it follows that
  $\theta (z) \sim \gamma z^{\gamma -1}$ and
  $\theta(\Lambda_1(z,1,\beta)) \sim \gamma z^{\gamma -1}$, uniformly for $\beta \in S_1$. We also
  note that by Lemma~\ref{lem:displacement}\ref{lem:displacement-a}, \eqref{eq:inangle}, and \eqref{eq:alpha-beta},
  for all $z$ large enough and all $\beta \in S_1$,
  $\Theta_z(\beta) -\beta = b_z(\beta) = \theta ( z ) + \theta (
  \Lambda_1 (z,1,\beta) )$. Therefore, 
\begin{equation}
\nonumber
  \lim_{z \to \infty}  \sup_{\beta \in S_1} \left| z^{1-\gamma} ( \Theta_z (\beta) - \beta
    ) - 2 \gamma \right| =
    \lim_{z \to \infty} \sup_{\beta \in S_1}  |z^{1-\gamma} b_z (\beta)
  - 2 \gamma | = 0, 
\end{equation} 
establishing the first statement in~\eqref{eq:theta_unif}.
Let $b_z' (\beta) := (\partial/\partial \beta) b_z (\beta)$, where $b_z$~is defined at~\eqref{eq:alpha-beta}.
Since $\Theta_z (\beta) = \beta + b_z (\beta)$, we have that
$\Theta'_z (\beta) = 1 + b_z'(\beta)$. From~\eqref{eq:D-def} we have
$\frac{\partial}{\partial\beta}\Lambda_1(z,1,\beta) = D'(z,\beta)$ and, by the chain rule, 
  \begin{equation}
    \label{eq:derivative-b-bound}
    \sup_{\beta \in S_1} | \Theta'_z(\beta) -1 | =
    \sup_{\beta \in S_1} | b'_z(\beta)|
    = \sup_{\beta \in S_1} | \theta'(\Lambda_1(z,1,\beta))
    D'(z,\beta)|,
  \end{equation}
where differentiation of~\eqref{eq:normalangle} shows that
\begin{equation}
\nonumber
  \theta' (z) := \frac{\ud}{\ud z} \theta (z) = - \frac{\gamma(1-\gamma) z^{\gamma-2}}{1 + \gamma^2 z^{2\gamma -2}} = - \gamma(1-\gamma) (1+o(1)) z^{\gamma -2} .
\end{equation}
By~\eqref{eq:D-def} and Lemma~\ref{lem:displacement}\ref{lem:displacement-b}, it follows that $\theta'
(\Lambda_1(z,1,\beta)) = O(z^{\gamma -2})$ and also, by
\eqref{eq:displacement-derivative}, it holds
that $\sup_{\beta \in
  S_1} \abs{D'(z,\beta)} = O(z^\gamma)$. Thus, by~\eqref{eq:derivative-b-bound}, we
  obtain the second statement 
in~\eqref{eq:theta_unif}.  We now turn to the
  proof of \eqref{eq:Tzsup}. Take $\theta_1' \in (\theta_0,\theta_1)$, and let $S_1' := [ -\theta'_1, \theta'_1]$, so that $S_0 \subset S_1' \subset S_1$. 
By~\eqref{eq:theta_unif}, we have
 $\inf_{\beta \in S_1} \Theta'_z (\beta) > 0$ for all $z$
sufficiently large,
and the image $\Theta_z (S_1)$ contains $S_1'$.
Thus for all $z$ sufficiently large, there is an inverse
function $T_z : S'_1 \to S_1$
such that
$\Theta_z (T_z (\beta)) = \beta$ for all $\beta \in S_1'$.
Moreover, 
by~\eqref{eq:theta_unif},
$T'_z (\beta) = 1/ \Theta_z' (T_z(\beta))$ satisfies
$\sup_{\beta \in S'_1} | T'_z (\beta) - 1 | = O (
z^{2\gamma-2})$.  Since
$\beta = \Theta_z (T_z (\beta ) ) = T_z (\beta) + b_z ( T_z (\beta))$ for every $\beta \in S'_1$,
by \eqref{eq:theta_unif} we have that
\begin{align*}
   \lim_{z \to \infty} \sup_{\beta \in S'_1} \left| z^{1-\gamma} ( T_z
      (\beta)- \beta) + 2\gamma \right|
    &= \lim_{z \to \infty} \sup_{\beta \in S'_1}
    \left|z^{1-\gamma} (\beta  - T_z(\beta)) - 2\gamma \right|\\
    &= \lim_{z \to \infty} \sup_{\beta \in S'_1}
    \left| z^{1-\gamma} ( \Theta_z (T_z (\beta)) - T_z (\beta)) - 2\gamma \right|= 0.
\end{align*}
Thus we have established~\eqref{eq:Tzsup}, but over $\beta \in S'_1$
rather than the (larger) $S_1$;  since both $S'_1$ and $S_1$ were chosen arbitrarily, a suitable relabelling shows that~\eqref{eq:Tzsup} holds as written.
\end{proof}

Write
$T_z (B) := \{ \beta \in S_0 : \beta + b_z (\beta ) \in B \}$ for $B \in \cB (S)$, and
define $\Ksp : \RP \times S \times \cB (S) \to [0,1]$ via
\begin{equation}
\label{eq:billiards-pseudo-kernel-def}
 \Ksp ( x , \alpha , B ) := \cK ( \alpha, T_{x^{1/(1-\gamma)}} (B) ) .
\end{equation}
By~\eqref{eq:inangle}  and~\eqref{eq:billiards-construction-angle},  if
we denote the next incoming angle by $\alpha_{n+1}$ and the next
outgoing angle by $\beta_n$,
 with the notation of \eqref{eq:alpha-beta} we see that there is $x_1 \in
 \RP$ for which
\begin{equation}
\label{eq:alpha-beta-2} 
\alpha_{n+1} = \beta_n + b_{Z_n} ( \beta_n), ~\text{on}~ \{ X_n \geq
x_1 \}.
\end{equation}
We now note that, on $\{ X_n \geq x_1\}$,
\[
  \Pr ( \alpha_{n+1} \in B \mid \cF_n)   = \Pr ( \beta_n + b_{Z_n}
  (\beta_n) \in B \mid \cF_n ) = \Pr ( \beta_n \in T_{Z_n} ( B ) \mid
  \cF_n ) ,
\]
and that, since $\Pr ( \beta_{n} \in B \mid \cF_n ) = \Kb ( \alpha_n , B)$, a.s., 
\begin{equation}
\label{eq:billiards-pseudo-kernel-as}
\Pr  ( \alpha_{n+1} \in B \mid \cF_n) = \cK ( \alpha_n, T_{Z_n} (B) ) = \Ksp ( X_n, \alpha_n ,B ), ~\text{on}~ \{ X_n \geq x_1 \}, 
\end{equation}
using~\eqref{eq:billiards-pseudo-kernel-def}. The next result shows that $\Ksp$ satisfies
the asymptotic Markovianity condition~\eqref{ass:asymptotically-markov} or~\eqref{ass:asymptotically-markov+}, as appropriate.

\begin{lemma}
\label{lem:billiards-asymptotically-markov}
Suppose that~\eqref{ass:ellipticity}--\eqref{ass:loopable} hold. 
Then for  $\Ksp$ as defined at~\eqref{eq:billiards-pseudo-kernel-def} and $\cK$ the
billiards 
reflection kernel, it holds that
\begin{equation}
\label{eq:billiards-asymptotically-markov}
\lim_{x \to \infty} \sup_{\alpha \in S} \norm{ \Ksp (x, \alpha, \, \cdot \, ) - \cK (\alpha, \, \cdot \,)}_\tv = 0 .
\end{equation}
Moreover, if~\eqref{ass:density-smoothness} holds, then, as $x \to \infty$, 
\begin{equation}
\label{eq:billiards-asymptotically-markov+} 
\sup_{\alpha \in S} \norm{ \Ksp (x, \alpha, \, \cdot \, ) - \cK (\alpha, \, \cdot \,) - x^{-1} \meas_\alpha }_\tv = o (x^{-1} ).
\end{equation}
Here  $\alpha \mapsto \meas_\alpha$ is continuous from $(S, d_S)$ to $( \Meass(S), d_\tv)$, given by
\begin{equation}
  \label{eq:gammadef}
  \meas_\alpha (B) := - 2 \gamma \int_B   \kappa' (\alpha,\beta) \ud \beta, \text{ for all } B \in \cB(S),
\end{equation} 
where $\kappa' (\alpha,\beta)  = (\partial/\partial \beta) \kappa (\alpha,\beta)$ as in~\eqref{ass:density-smoothness}.
\end{lemma}
\begin{proof}
Let $\kappa$ be the density from~\eqref{ass:billiards-density}. 
First note that 
\begin{equation}
\label{ksbtv}
  \norm{\Ksp (x,\alpha, \, \cdot \, ) - \Kb(\alpha, \, \cdot \, )}_{\tv} = 
	{\sup_{f \in \Cb(S): \norm{f}\leq 1}} \int_S  f(\beta) \left( \Ksp(x,\alpha, \ud \beta) -  \Kb(\alpha , \ud \beta) \right),
\end{equation}
where $\| f \|:= \sup_{\alpha \in S} | f(\alpha) |$, and we emphasize that in~\eqref{ksbtv}, $\beta$ in $\Kb$ represents the next \emph{outgoing angle}, but in $\Ksp$ it is the subsequent \emph{incoming angle}.
From~\eqref{eq:billiards-pseudo-kernel-def}, for $x \geq x_1$, 
\begin{align}\label{eq:ksb}
\int_S  f(\beta)  \Ksp(x,\alpha, \ud \beta) & = \int_S  f(\beta) \cK ( \alpha, T_{x^{1/(1-\gamma)}} ( \ud \beta ) ) \nonumber\\
& = \int_{S} f (\beta ) \kappa (\alpha, T_{x^{1/(1-\gamma)}} (  \beta ) )  T'_{x^{1/(1-\gamma)}} (  \beta )  \ud \beta .
\end{align}
Since~\eqref{ass:ellipticity} states that $\kappa (\alpha, \beta) =0$ for $\beta \notin S_0$,
and $T_z (\beta ) \to \beta$  uniformly for $\beta \in S_1$ (see~\eqref{eq:Tzsup} in~Lemma~\ref{lem:theta-properties}),
for all $x$ large enough we can replace $S$ by $S_1$ in the final integral in~\eqref{eq:ksb}. 
Thus from~\eqref{eq:ksb}, the boundedness of~$\kappa$ from~\eqref{ass:billiards-density}, and the bound on~$T_z'$ from Lemma~\ref{lem:theta-properties},
\begin{equation}
\label{eq:density-shift}
\sup_{f : \| f \| \leq 1} \sup_{\alpha \in S} \left|  \int_S  f(\beta)  \Ksp(x,\alpha, \ud \beta)  - \int_{S_1} f ( \beta  )  \kappa (\alpha,  T_{x^{1/(1-\gamma)}} (\beta) ) \ud  \beta \right| = O (x^{-2}) ,
\end{equation}
as $x \to \infty$.
By the uniform equicontinuity  of $\kappa (\alpha, \, \cdot \,)$ from~\eqref{ass:billiards-density}, 
and the fact that $T_z (\beta) \to \beta$ uniformly in $\beta \in S_1$, it then follows from~\eqref{eq:density-shift} that
\[
  \lim_{x \to \infty} \sup_{f : \| f \| \leq 1} \sup_{\alpha \in S} \left|  \int_S  f(\beta)  \Ksp(x,\alpha, \ud \beta)  - \int_{S} f (  \beta  )  \kappa (\alpha,  \beta ) \ud   \beta \right| = 0 .
\]
Together with~\eqref{ksbtv}, this yields~\eqref{eq:billiards-asymptotically-markov}.
 
For~\eqref{eq:billiards-asymptotically-markov+}, we look in more detail at~\eqref{eq:density-shift}. 
Under assumption~\eqref{ass:density-smoothness},
we have that $(\partial / \partial \beta) \kappa (\alpha ,\beta ) = \kappa' (\alpha, \beta)$
and $(\partial^2 / \partial^2 \beta) \kappa (\alpha ,\beta ) = \kappa'' (\alpha, \beta)$
are both uniformly bounded over $\alpha, \beta \in S$.
Then, by the uniform boundedness of $\kappa'$ and $\kappa''$, 
and the asymptotics for~$T_z(\beta) - \beta$ from Lemma~\ref{lem:theta-properties}, 
Taylor's theorem with Lagrange remainder shows that
\[
 \sup_{\alpha \in S} \sup_{\beta \in S_1} \left| \kappa (\alpha, T_{x^{1/(1-\gamma)}} (\beta)) - \kappa (\alpha, \beta)
 + 2\gamma x^{-1}  \kappa' (\alpha, \beta) \right| = o (x^{-1}) .
\]
Thus from~\eqref{eq:density-shift} we obtain
\[
 \sup_{f : \| f \| \leq 1} \sup_{\alpha \in S} \left|  \int_S  f(\beta)  \Ksp(x,\alpha, \ud \beta)  - 
\int_{S} f (  \beta  )  \kappa (\alpha,  \beta ) \ud \beta 
+ \frac{2\gamma}{x} \int_{S} f (  \beta  ) \kappa' (\alpha,   \beta) \ud \beta \right| = o (x^{-1}) .
\]
But the left-hand side here is  $ \norm{ \Ksp (x, \alpha, \, \cdot \, ) - \cK (\alpha, \, \cdot \,) - x^{-1} \meas_\alpha }_\tv$,
where $\meas_\alpha$ is given by~\eqref{eq:gammadef}.
Finally, continuity of $\alpha \mapsto \meas_\alpha$
follows from the fact that 
\[ \sup_{B \in \cB (S)} \left| \meas_\alpha (B) - \meas_{\alpha'} (B) \right| \leq 2 \gamma \int_S \left| \kappa' (\alpha, \beta ) - \kappa' (\alpha' , \beta ) \right| \ud \beta \to 0 ,\]
as $\alpha' - \alpha \to 0$, by dominated convergence, the uniform boundedness of~$\kappa'$, and the continuity of $\alpha \mapsto \kappa' (\alpha, \beta)$ from~\eqref{ass:density-smoothness}.
\end{proof}
 
By Lemma~\ref{lem:billiards-to-strip-markov} there exist measurable $\mu_k : \RP \times S \to \R$ such that, on $\{ X_n \geq x_B\}$, 
\[
  \Exp [ X_{n+1} - X_n \mid \cF_n ] = \mu_1 (X_n, \alpha_n ) , ~~~ \Exp [ ( X_{n+1} - X_n)^2 \mid \cF_n ] = \mu_2 (X_n, \alpha_n ) .
\]
Also, similarly to~\eqref{eq:mu-disintegration}, by disintegration there exists
 $\cmu_1 : \RP \times S \times S \to \R$ such that 
\begin{equation}
\label{eq:tmu1-def}
\mu_1 (x, \alpha ) = \int_S \Ksp ( x, \alpha , \ud \beta ) \cmu_1 ( x, \alpha, \beta ) , ~\text{for all}~x \geq x_B,
\end{equation}
where~$\Ksp$ is given by~\eqref{eq:billiards-pseudo-kernel-def};
we emphasize that while we often use~$\beta$ for the next outgoing angle, in~\eqref{eq:tmu1-def}
and other equations involving~$\cmu_1$, $\beta$~represents the  subsequent \emph{incoming} angle. 
Recall the definition of $\rho_k$ from~\eqref{eq:rho-def}.

\begin{lemma}
\label{lem:transformed-billiards-lamperti}
Suppose that~\eqref{ass:ellipticity}--\eqref{ass:billiards-invariant} hold.
For each $k \in \N$, $\rho_k \in \Cb (S)$, and, as $x \to \infty$,
\begin{align}
\label{eq:billiards-mu1}
\sup_{\alpha \in S} \bigl| \mu_1 (x, \alpha ) -   2 (1-\gamma) \rho_1 ( \alpha)  -  2 x^{-1} \gamma (1-\gamma) ( 1 + \rho_2 (\alpha ) ) \bigr| & = O ( 1/x^2) ; \\
\label{eq:billiards-tmu1}
\sup_{\alpha \in S} \sup_{\beta \in S_1} \bigl| \cmu_1 (x, \alpha , \beta ) -  2 (1-\gamma) \tan \beta \bigr| & = O (1/x) ; \\
\label{eq:billiards-mu2}
\sup_{\alpha \in S} \bigl| \mu_2 (x, \alpha ) -  4 (1-\gamma)^2 \rho_2 (\alpha ) \bigr| & = O (1/x) .
\end{align}
\end{lemma}
\begin{proof}
Since $\cK (\alpha, S_0) = 1$, we may extend $\tan^k \alpha$ continuously to a uniformly bounded function over $S$,
and then an application of Lemma~\ref{lem:continuity-against-K} shows that $\rho_k$ is continuous and bounded, as claimed.
Denote the increment of process $X_n$ at $(x,\alpha)$ when the outgoing angle is $\beta$ by
$\tcmu_1 ( x, \alpha, \beta )$.
Then, on $\{ X_n > x_1 \}$ for $x_1$ sufficiently large,
\[
  X_{n+1} - X_n = Z_{n+1}^{1-\gamma} - Z_n^{1-\gamma} = \Lambda_1 (Z_n , 1, \beta_n )^{1-\gamma} - Z_n^{1-\gamma} = \tcmu_1 ( X_n, \alpha_n, \beta_n) .
\]
Moreover, since $\Pr ( \beta_n \in B \mid \cF_n ) = \cK ( \alpha_n , B)$, a.s., 
we have 
\[
  \mu_1 (x, \alpha ) = \int_S \cK ( \alpha , \ud \beta )   \tcmu_1 ( x, \alpha, \beta ) ;
\]
compare to~$\cmu_1$ as defined through~\eqref{eq:tmu1-def}. 
For $x = z^{1-\gamma}$ and $D(z,\beta)$ as defined at~\eqref{eq:D-def},
\begin{align*}
\tcmu_1 ( x, \alpha, \beta )= \Lambda_1 (z , 1, \beta )^{1-\gamma} - z^{1-\gamma} & = z^{1-\gamma} \left[ \left( 1 + \frac{D(z, \beta)}{z} \right)^{1-\gamma} - 1 \right] .
\end{align*}
Lemma~\ref{lem:displacement} shows that $\sup_{\beta \in S_0} | D(z, \beta ) | = O (z^\gamma)$,
so that, by Taylor's theorem,
\[
  \tcmu_1 ( z^{1-\gamma} , \alpha , \beta ) = (1-\gamma) z^{-\gamma} D(z, \beta) - \frac{\gamma (1-\gamma)}{2} z^{-1-\gamma}  D(z, \beta)^2 + O ( z^{2\gamma-2} ) ,
\]
uniformly over $\alpha \in S, \beta \in S_0$. 
Thus we obtain from~\eqref{eq:displacement} 
that 
\begin{equation}
  \label{eq:tmu1-expansion}
  \sup_{\alpha \in S} \sup_{\beta \in S_0} 
  \left|  \tcmu_1 (x , \alpha , \beta ) - 2 (1-\gamma) \tan \beta  - 2 \gamma (1-\gamma) x^{-1} \left[ 1 +\tan^2 \beta  \right] \right| = O ( 1/x^2 ) .
\end{equation}
Recall from~\eqref{eq:billiards-construction-angle} and Lemma~\ref{lem:theta-properties} that
$\alpha_{n+1} = \Theta_{Z_n} (\beta_n)$, where $\Theta_z$ has an inverse $T_z$ 
such that  $\sup_{\beta \in S_1} | T_z ( \beta ) - \beta | = O (z^{\gamma-1})$.
 Then we see that
\[  \cmu_1 (x,\alpha, \beta ) =  \tcmu_1 ( x, \alpha, T_z ( \beta ) ) ,\]
and thus~\eqref{eq:billiards-tmu1} follows from~\eqref{eq:tmu1-expansion}. 
Moreover,  since $ X_{n+1} - X_n  = \tcmu_1 ( X_n, \alpha_n, \beta_n)$ on $\{ X_n > x_1 \}$,  we obtain from~\eqref{eq:tmu1-expansion} that 	
\begin{equation}
\label{eq:X-increment}
\left| X_{n+1} - X_n - 2 (1-\gamma) \tan \beta_n - 2 \gamma (1-\gamma) X_n^{-1} \left[ 1 +\tan^2 \beta_n \right] \right| \leq C X_n^{-2} ,
\end{equation}
for some $C \in \RP$.
Since $\Pr ( \beta_n \in B \mid \cF_n ) = \cK (\alpha_n, B)$, we have
$\Exp [ \tan^k \beta_n \mid \cF_n ] = \rho_k ( \alpha_n)$, 
with $\rho_k$ as defined at~\eqref{eq:rho-def}.
It then follows from~\eqref{eq:X-increment} that
\[
\Exp [ X_{n+1} - X_n \mid \cF_n ] = 2 (1-\gamma) \rho_1 (\alpha_n) +  2 \gamma (1-\gamma) X_n^{-1} ( 1 + \rho_2 (\alpha_2) ) + O (X_n^{-2} ),\]
where the implicit constants in the $O( \, \cdot \,)$ are non-random. This gives~\eqref{eq:billiards-mu1}.
A similar argument, starting from~\eqref{eq:X-increment}, yields~\eqref{eq:billiards-mu2}.
\end{proof}

\subsection{Recurrence classification}

To prove our results from Section~\ref{sec:billiards-results},
we will combine 
Lemmas~\ref{lem:billiards-to-strip-markov}, \ref{lem:billiards-asymptotically-markov} and~\ref{lem:transformed-billiards-lamperti}
  to show that the rescaled billiards process $(X_n , \alpha_n)$ satisfies the conditions of
	the appropriate half-strip results from Section~\ref{sec:strips-results}. 
First we present the proof of Proposition~\ref{prop:billiards-non-critical}.

\begin{proof}[Proof of Proposition~\ref{prop:billiards-non-critical}.]
Under the conditions of Proposition~\ref{prop:billiards-non-critical},
 the process $(X_n,\alpha_n)$
is a half-strip Markov chain for which~\eqref{ass:non-confinement} holds (by Lemma~\ref{lem:billiards-to-strip-markov}\ref{lem:billiards-to-strip-markov-c})
and~\eqref{ass:moments} holds for all $p >1$ and all $q \in (0,1)$ (Lemma~\ref{lem:billiards-to-strip-markov}\ref{lem:billiards-to-strip-markov-a} and~\ref{lem:billiards-to-strip-markov-b}).
Condition~\eqref{ass:kernel} follows from~\eqref{ass:billiards-density} and~\eqref{ass:billiards-invariant}, with the identification $\pii(\ud \alpha) = \pib(\alpha)\ud
  \alpha$.
  Also by Lemma~\ref{lem:billiards-asymptotically-markov} it follows that~\eqref{ass:asymptotically-markov} holds, 
  and by Lemma~\ref{lem:transformed-billiards-lamperti} it follows
  that~$\lim_{x \to \infty} \sup_{\alpha \in S} | \mu_1 (x,\alpha) - d_\alpha | = 0$
  where $d_\alpha := 2 (1-\gamma) \rho_1 (\alpha)$.
	Write $\delta = \delta (\gamma) = \int_S d_\alpha \pib (\alpha)  \ud \alpha = 2 (1-\gamma) \bar \rho_1$,
	by the $k=1$ case of~\eqref{eq:billiards-average-drift}. 
Then Proposition~\ref{prop:strips-crude}
says that 
 the process is transient if $\delta >0$ and  recurrent if $\delta <0$, and the sign of $\delta$ is the same as the sign of $\bar \rho_1$.
\end{proof}

Observe that equations~\eqref{eq:billiards-mu1} and~\eqref{eq:billiards-mu2} show that~\eqref{eq:generalized-lamperti-regime} holds with
\begin{equation}
\label{eq:billiards-generalized-parameters}
  d_\alpha  = 2 (1-\gamma) \rho_1 (\alpha); ~~~ e_\alpha = 2 \gamma (1-\gamma) (1+\rho_2 (\alpha)); ~~~ \sigma_\alpha^2 = 4 (1-\gamma)^2 \rho_2 (\alpha).
\end{equation}
Also~\eqref{eq:billiards-tmu1} shows that~\eqref{ass:detail-drifts} holds with 
\begin{equation}
\label{eq:billiards-lambda}
  \lambda_\alpha (\beta ) = 2 (1-\gamma ) \tan \beta .
\end{equation}

\begin{proof}[Proof of Theorem~\ref{thm:billiards-strict-lamperti}.]
  When $\rho_1 (\alpha ) =0$ for all $\alpha \in S$ we have $d_\alpha \equiv 0$
  and so~\eqref{ass:strict-lamperti} holds.
Then Theorem~\ref{thm:strips-strict-lamperti}
applies;
write $\delta = \delta (\gamma) = \delta_0$ as in~\eqref{eq:deltadef}, so, by~\eqref{eq:billiards-generalized-parameters}, \[
\delta   = \int_S ( 2 e_\alpha - \sigma^2_\alpha ) \pib (\alpha ) \ud \alpha = 
4 (1-\gamma) \left(\gamma (1 + 2 \bar{\rho}_2)- \bar{\rho}_2  \right) .
\]
Theorem~\ref{thm:strips-strict-lamperti} gives
 recurrence if $\delta  <0$ and transience if $\delta >0$, where
 $\delta   <0 $ if $0 < \gamma < \gammacs $
and $\delta  >0$ if $\gammacs < \gamma < 1$,
with $\gammacs $ given in~\eqref{eq:critical-parameter-strict}.
\end{proof}

Moving on to Theorem~\ref{thm:billiards-general-lamperti}, we will
denote by $\psi_\gamma \in \Cb (S)$ a function (whose existence is
guaranteed by Proposition~\ref{prop:fredholm}) such that
\[
  \int_S ( \psi_\gamma (\beta) - \psi_\gamma (\alpha) ) \cK (\alpha , \ud \beta ) = - d_\alpha = - 2 (1-\gamma) \rho_1 (\alpha) .
\]
The function~$\psi_\gamma$ is unique up to translation (see
Proposition~\ref{prop:fredholm}). We may also suppose that $\psi_\gamma$ and
$\psi_0$ are related by $\psi_\gamma = (1-\gamma) \psi_0$.
 
\begin{proof}[Proof of Theorem~\ref{thm:billiards-general-lamperti}.]
  Theorem~\ref{thm:strips-general-lamperti} shows that we have
  recurrence or transience according to the sign of
  $\delta = \delta (\gamma) = \tdelta_0$, as defined at~\eqref{eq:deltadef-general}; 
  by~\eqref{eq:gammadef}, \eqref{eq:billiards-generalized-parameters}, and~\eqref{eq:billiards-lambda},
\begin{align*} \delta   & = 4 \gamma (1-\gamma) ( 1 + \bar \rho_2 ) - 4 (1-\gamma)^2 \bar \rho_2   - 4 \gamma \int_S \int_S \psi_\gamma (\beta) \kappa' (\alpha , \beta ) \pib ( \alpha) \ud \alpha \ud \beta \\
& {} \qquad {} - 4 (1-\gamma ) \int_S \int_S \psi_\gamma (\beta ) \kappa (\alpha, \beta ) \pib (\alpha ) \tan \beta \ud \alpha \ud \beta . \end{align*}
Moreover, $\delta$ is invariant under translation of $\psi_\gamma$.
Using the fact that $\pib$ is invariant to simplify the last term, and $\psi_\gamma = (1-\gamma) \psi_0$, we get
\begin{align*} \delta  & = 4 (1-\gamma) ( \gamma + (2\gamma -1) \bar \rho_2 )  - 4 \gamma (1-\gamma)  \int_S \int_S \psi_0 (\beta) \kappa' (\alpha , \beta ) \pib ( \alpha) \ud \alpha \ud \beta \\
& {} \qquad {} - 4 (1-\gamma )^2   \int_S \psi_0 (\beta ) \pib (\beta ) \tan \beta \ud \beta . \end{align*}
Thus, with~\eqref{eq:pi-derivative} and  the definitions of $A_1, A_2$ at~\eqref{eq:A-def}, we get
\begin{equation}
\nonumber
\label{eq:billiards-delta-general}
  \delta  = 4 (1-\gamma)  \left[ ( \gamma + (2\gamma -1) \bar \rho_2 ) - (1-\gamma) A_1 - \gamma A_2 \right] .
\end{equation}
It follows from~\eqref{eq:billiards-delta-general} that, for $\gamma \in (0,1)$, the sign of $\delta$ is the same as that of
$c(\gamma)$, where
\begin{equation}
\label{eq:c-gamma-def}
c ( \gamma )   := \gamma \left( 1  + A_1 - A_2 + 2\bar \rho_2 \right)  - A_1 - \bar \rho_2 , \text{ for } 0 < \gamma < 1.
\end{equation}
Theorem~\ref{thm:strips-general-lamperti} then shows that $\zeta$ is transient if $c(\gamma) >0$ and recurrent if $c(\gamma) < 0$.

A consequence of the fact that~$\tdelta_\theta$ as defined at~\eqref{eq:deltadef-general} is non-decreasing in~$\theta$ (see Theorem~\ref{thm:strips-general-lamperti}) is that $A_1 + \bar \rho_2 \geq 0$.
Under the hypothesis~\eqref{eq:A-conditions}, 
the function $\gamma \mapsto c(\gamma)$
given by~\eqref{eq:c-gamma-def} is non-decreasing with $c(0) \leq 0$, and  
$\gammag \in [0,1]$ given by~\eqref{eq:critical-parameter-general} is well defined (see Remark~\ref{rem:billiards-general}).
If $1+A_1-A_2+2\bar\rho_2 =0$, then, by~\eqref{eq:A-conditions},
$A_1 +\bar\rho_2 >0$ and $\gammag = 1$, while $c(\gamma) = - A_1 -\bar\rho_2 <0$ (recurrence) for all $0 < \gamma < 1=\gammag$.  
If $1+A_1-A_2+2\bar\rho_2  > 0$,
then $c(\gamma)$ is strictly increasing, and has the property that $c(\gamma) <0$ if $\gamma < \gammag$ and $c(\gamma)>0$ if $\gamma > \gammag$. This completes the proof of the recurrence classification.

The expression for~$\psi_0$ in~\eqref{eq:psi-0-billiards}, under the hypothesis on the total-variation convergence of $\cK^n$,
is a
consequence
 of~\eqref{eq:psi-series}, \eqref{eq:rho-def},
and~\eqref{eq:billiards-generalized-parameters}.
\end{proof}
	 
 \begin{proof}[Proof of Proposition~\ref{prop:example-family}.]
 If~\eqref{eq:rho1-condition} holds then  
we claim that a solution $\psi_0 \in \Cb(S)$ to $\int_S ( \psi_0 (\beta) - \psi_0 (\alpha) ) \kappa (\alpha, \beta ) \ud \beta   = - 2 \rho_1 (\alpha)$ is given by
\begin{equation}
\label{eq:explicit-psi0}
\psi_0 (\beta) = \frac{2\lambda}{1-\lambda} \tan \beta , ~ \text{for}~ \beta \in S.
\end{equation}
Indeed, 
with the choice for $\psi_0$ given by~\eqref{eq:explicit-psi0}, 
by~\eqref{eq:rho1-condition} it follows that
\begin{align*}
\int_S ( \psi_0 (\beta) - \psi_0 (\alpha) ) \kappa (\alpha, \beta ) \ud \beta & = \frac{2\lambda}{1 - \lambda} \int_S \kappa (\alpha , \beta )\tan \beta \ud \beta
- \frac{2\lambda}{1-\lambda} \tan \alpha  \\
& = \frac{2\lambda}{1 - \lambda} \rho_1 (\alpha ) -  \frac{2\lambda}{1 - \lambda}\tan \alpha
=   - 2\rho_1 (\alpha), 
\end{align*}
as required. By uniqueness of~$\psi_0$ up to translation, we may suppose that~$\psi_0$ is given by~\eqref{eq:explicit-psi0}. 
Recall the definitions of $A_1, A_2$ from~\eqref{eq:A-def}. 
With $\psi_0$ given by~\eqref{eq:explicit-psi0}, we have 
\[  A_1 = \frac{2\lambda}{1 - \lambda} \int_S \varpi (\beta) \tan^2 \beta \ud \beta  = \frac{2\lambda}{1 - \lambda} \bar \rho_2 .
\]
      Now to compute $A_2$ observe first that the 
 function $\psi_0$ given by~\eqref{eq:explicit-psi0} is differentiable on~$S_0$,
with derivative $\psi'_0 (\beta ) = \frac{2\lambda}{1 - \lambda} ( 1 + \tan^2 \beta)$. Under~\eqref{ass:density-smoothness}, the density $\varpi$ is also differentiable
with derivative given by~\eqref{eq:pi-derivative}.
Also, since $\psi_0 (-\beta) \varpi (-\beta ) = -\psi_0 (\beta ) \varpi (\beta )$ for all $\beta \in S_0$,
we have that $h (\beta ) := \psi_0 (\beta) \varpi (\beta)$ has $h (-\beta) =- h (\beta)$.
Thus 
\[
\int_{S_0} \left[ \psi_0 (\beta) \varpi' (\beta) + \psi_0' (\beta) \varpi (\beta) \right] \ud \beta = \int_{S_0} h'(\beta)\ud \beta = 2h(\theta_0) =0,
\]
since $\pib ( \theta_0 ) = 0$, by the comment after~\eqref{eq:pi-derivative}. The above computation implies that
\begin{align*}
A_2 & = \int_{S_0} \psi_0 (\beta) \varpi' (\beta) \ud \beta =  - \int_{S_0} \psi_0' (\beta) \varpi (\beta) \ud \beta  \\
& =  - \frac{2\lambda}{1 - \lambda} \int_{S_0} \left[  1 + \tan^2 \beta \right] \varpi (\beta) \ud \beta   = -\frac{2\lambda}{1-\lambda} \left[ 1 + \bar \rho_2 \right].
\end{align*}
Since $\lambda \in(-1,1)$,
      $A_1 + \bar{\rho}_2 = \frac{1 + \lambda}{1 - \lambda} \bar{\rho}_2>0$,
      and $1 + A_1 -A_2 +2\bar \rho_2 = \frac{1 + \lambda}{1 - \lambda} ( 1 +2\bar\rho_2) >0$,
      which means that condition~\eqref{eq:A-conditions} holds, and,
      moreover, that $\gammag$ defined by~\eqref{eq:critical-parameter-general} is given by
$\gammag = \frac{\bar\rho_2}{1+2\bar\rho_2} = \gammacs$,
as given by~\eqref{eq:critical-parameter-strict}. The result now follows from Theorem~\ref{thm:billiards-general-lamperti}. \end{proof}

\appendix
 
\section{Kernels, operators, and Fredholm theory}
\label{sec:fredholm}

As in Section~\ref{sec:strips}, let $(S,d_S)$ be a compact metric space,
$\cB(S)$ its Borel $\sigma$-algebra, and $\cK : S \times \cB(S) \to [0,1]$ a Markov kernel on $S$. 
Recall that $\Mb(S)$
is the set of bounded measurable functions on $S$, and 
 $\Cb(S)$ the continuous functions on $S$. We endow $\Cb(S)$ with the supremum norm $\| f \| := \sup_{u \in S} | f(u) |$,
so $\Cb(S)$ is a Banach space. 
The kernel $\cK$ is associated with a functional $T_\cK : \Mb (S) \to \Mb(S)$ whose operation is defined by
\begin{equation}
\nonumber
\label{eq:operator-def}
  ( T_\cK f ) (u) = \int \cK (u , \ud v ) f (v) , ~\text{for all}~u \in S .
\end{equation}
The \emph{Feller property} is that $f \in \Cb(S)$ implies $T_\cK f \in \Cb (S)$~\cite[\S 12.1]{dmps}.
The Feller property does not hold in general, but it does under assumption~\eqref{ass:kernel}\ref{ass:kernel-ii},
which implies the stronger fact that $T_\cK f \in \Cb (S)$ for all $f \in \Mb (S)$: see Lemma~\ref{lem:continuity-against-K}.  
We also note that $\| T_\cK f \| \leq \| f \|$ for all~$f$.
Thus $T_\cK$ defines a continuous linear operator $T_\cK: \Cb(S) \to \Cb (S)$~\cite[p.~127]{ka}.

Consider for $f, g \in \Cb (S)$ the \emph{Poisson equation} 
\begin{align}
\label{eq:fredholm1}
f - T_\cK f & = g .\end{align}
Recall that if $\cK$ satisfies~\eqref{ass:kernel}\ref{ass:kernel-i}, then there is a unique invariant probability measure $\pii \in \cP(S)$;
recall the definition of $\Cb^0 (S)$ from~\eqref{eq:pi-kernel}.
The main result of this section is Proposition~\ref{prop:fredholm} below.
We will employ Proposition~\ref{prop:fredholm} in two  
ways in the proofs of our results on the half-strip model:
first, to establish existence of Lyapunov functions with appropriate properties
to conduct the proofs for the strict Lamperti regime,
as described in Section~\ref{sec:strict-lamperti-proofs},
and second,
to construct a transformation mapping the general Lamperti case
into the strict Lamperti case,
as described in Section~\ref{sec:general-lamperti-proofs}.

\begin{proposition}
\label{prop:fredholm}
Let $(S,d_S)$ be a compact metric space, $\cB(S)$ its Borel $\sigma$-algebra, and $\cK : S \times \cB(S) \to [0,1]$ a Markov kernel
satisfying~\eqref{ass:kernel}. 
Then there exists a continuous linear operator $F : \Cb^0 (S) \to \Cb^0 (S)$
such that for every $g \in \Cb^0 (S)$ there is a unique $f = F(g) \in \Cb^0 (S)$ that solves~\eqref{eq:fredholm1}.
\end{proposition}

We establish Proposition~\ref{prop:fredholm} by  the Fredholm alternative theorem 
for linear operators. First we collect some necessary concepts and notation.
The linear dual space to $\Cb(S)$ is the Banach space $L(S)$ of continuous linear functionals from $\Cb (S) \to \R$,
  endowed with the induced (operator) norm $\| \phi \| := \sup_{ \| f \| \leq 1 } | \phi ( f) |$. 
By the Riesz representation theorem~\cite[p.~265]{ds},
$L(S)$ can be identified isometrically with $\Meass (S)$, the space of finite signed Borel measures on~$S$,
 with total variation norm 
\begin{equation}
  \label{eq:tv} \| \nu \|_\tv = \sup \left\{ \int_S f ( u) \ud \nu (u) : f \in \Cb(S), \, \| f \| \leq 1 \right\} ,
\end{equation}
since a 
continuous linear functional $\phi \in L(S)$ corresponds to a unique finite signed Borel measure $\nu$,
via $\phi (f) = \int_S f( u) \nu (\ud u)$ over all $f\in \Cb (S)$.

The adjoint operator $\Ts_\cK$ to $T_\cK$, acts as $\Ts_\cK : L (S) \to L (S)$ via
$\Ts_\cK \phi  = \phi T_\cK$, or, equivalently, as $\Ts_\cK : \Meass (S) \to \Meass (S)$ via 
\begin{equation}
  \label{eq:T-adjoint} ( \Ts_\cK \nu ) (B) := \int_S \nu (\ud u ) \cK (u, B) , ~\text{for all}~B \in \cB(S) .
\end{equation}
In particular, 
$\Ts_\cK$ restricts to a functional given by~\eqref{eq:T-adjoint} on the metric space $(\cP(S), d_\tv)$,
where $d_\tv ( \mu,\nu ) = \tfrac{1}{2} \| \mu - \nu \|_\tv$
is the total variation distance.

A linear operator between two Banach spaces is \emph{compact} if it maps
bounded sets into relatively compact sets. 
The following lemma is essentially given in~\cite[pp.~36--37]{revuz}; we include a short proof here for completeness.

\begin{lemma}
If~\eqref{ass:kernel}\ref{ass:kernel-ii} holds, then the operator $T_\cK : \Cb(S) \to \Cb (S)$ is compact.
\end{lemma}
\begin{proof}
Let $B_r = \{ f \in \Cb(S) : \| f \| \leq r \} \subset \Cb(S)$.
It suffices to prove that $T_\cK B_r = \{T_\cK f : f \in B_r\} \subseteq B_r$ is relatively compact.
For $f \in \Cb(S)$ and $u, v \in S$ we can write
\begin{equation}
\label{eq:T-difference}
  T_\cK f   (u) -   T_\cK f  (v) = \int_S f (z) L_{u,v} (\ud z ) ,
\end{equation}
where $L_{u,v} \in \Meass (S)$ is the signed measure defined by
$L_{u,v} (B) = \cK (u, B) - \cK (v,B)$ for $B \in \cB(S)$. 
It then follows from~\eqref{eq:tv} and~\eqref{eq:T-difference} that
\begin{equation}
\label{eq:equicontinuity}
\left| T_\cK f  (u) -   T_\cK f  (v) \right| \leq r \cdot \| L_{u,v} \|_\tv , \text{ for all } f \in B_r .
\end{equation}
Let $u \in S$. By~\eqref{ass:kernel}\ref{ass:kernel-ii}, for any $\eps >0$ there exists $\delta >0$
such that $\| L_{u,v} \|_\tv < \eps$ for all $v \in S$ with $|u-v| < \delta$. Thus~\eqref{eq:equicontinuity}
shows that the collection of functions~$T_\cK B_r$ is equicontinuous.
Furthermore, $\norm{T_\cK f}\leq \norm{f}$. 
Hence the Arzel\`a--Ascoli theorem~\cite[p.~266]{ds}
shows that $T B_r$ is relatively compact.
\end{proof}

Now we can complete the proof of Proposition~\ref{prop:fredholm}. 
Let $T$ be a compact operator on a Banach space $\cX$ and $T^*$ its adjoint on the dual space $\cX^*$;
in both spaces we denote by~$I$ the identity operator.
For a set $C \subseteq \cX^*$ let $C^{{a}} := \{ x \in \cX :  \phi (x) = 0 \text{ for all } \phi \in C \}$,
the annihilator of~$C$.
We write `$\ker$' and `$\Ima$' for kernel and range, respectively.
We will use the following result,
which can be found e.g.~in~\cite[pp.~609--610]{ds} or~\cite[p.~369]{ka}.

\begin{lemma}[Fredholm alternative]
\label{lem:fredholm}
Let $T$ be a compact operator on a Banach space $\cX$ and $T^*$ its adjoint on the dual space $\cX^*$.
Fix a scalar $\lambda$.
Then 
\[
\dim \ker ( \lambda I - T ) = \dim \ker ( \lambda I - T^* ),
\]
and
\[ \Ima  ( \lambda I - T ) = ( \ker ( \lambda I - T^* ) )^{{a}}. \]
Moreover, for any $y \in ( \ker ( \lambda I - T^* ) )^{{a}}$, 
  the set of all solutions $x \in \cX$ with $( \lambda I - T ) x = y$
is equal to $\{ x_0 + z : 
z \in \ker ( \lambda I - T ) \}$ for any particular solution~$x_0$.
\end{lemma}

\begin{proof}[Proof of Proposition~\ref{prop:fredholm}.]
In~\eqref{ass:kernel}\ref{ass:kernel-i}
we have assumed uniqueness of solutions to $\Ts_\cK \nu = \nu$ over $\cP(S)$; we claim that 
this implies that
\begin{equation}
\label{eq:claim}
\text{every solution to $\Ts_\cK \nu = \nu$ over
 $\Meass (S)$ has $\nu = \rho \pii$ for some $\rho \in \R$.}
\end{equation}
To prove~\eqref{eq:claim},  we use a decomposition argument.
If $\nu  = \Ts_\cK \nu$ for $\nu  \in \Meass (S)$, the Hahn--Jordan decomposition of~$\nu$
is $\nu = \nu^+ - \nu^-$ for two finite measures $\nu^+, \nu^-$,
and $\nu^+ = \Ts_\cK \nu^+$ and
$\nu^- = \Ts_\cK \nu^-$
too~\cite[p.~17]{dmps}.
By assumption, 
 $\nu = \Ts_\cK \nu$ has a unique  solution $\nu = \pii \in \cP(S)$,
which means that every  $\nu  \in \Meass (S)$ for which $\nu  = \Ts_\cK \nu$
has $\nu^+ = \rho_+ \pii$ and $\nu^- = \rho_- \pii$ for $\rho_+, \rho_- \in \RP$.
Thus $\nu = (\rho_+ - \rho_-) \pii = \rho \pii$, $\rho \in \R$, verifying~\eqref{eq:claim}.

Now Lemma~\ref{lem:fredholm} with $\lambda =1$ together with~\eqref{eq:claim} shows that $\ker ( I - \Ts_\cK ) = \{ \rho \pii : \rho \in \R \}$
so both $\ker ( I - \Ts_\cK)$ and $\ker (I - T_\cK)$ are one-dimensional. Hence
$\ker (I - T_\cK)$ consists of only the constant functions.
In addition, by the definition of $\Cb^0(S)$ at~\eqref{eq:pi-kernel},
\begin{equation}
\label{eq:im-ker}
\Ima  (   I - T_\cK ) = ( \ker (   I - T_\cK^* ) )^{{a}} = \Cb^0 (S) .
\end{equation}
Thus~\eqref{eq:fredholm1} has a solution $f \in \Cb(S)$ 
for a given $g \in \Cb (S)$ if and only if $g \in \Cb^0 (S)$.

Moreover,
given $g \in \Cb^0 (S)$, 
 the set of all solutions to~\eqref{eq:fredholm1} is $\{ f + c, c \in \R \}$, where
 $f  \in \Cb(S)$  is any solution to~\eqref{eq:fredholm1}. It follows that for $g \in \Cb^0 (S)$
there is a unique $f \in \Cb^0 (S)$ that solves~\eqref{eq:fredholm1}.
Thus we may define $F : \Cb^0 (S) \to \Cb^0 (S)$ by $F (g) = f$ satisfying~\eqref{eq:fredholm1}.
It is easy to see that $F$ is linear. It remains to prove that $F$ is continuous.

Consider $U = I - T_\cK$. Then~\eqref{eq:im-ker}
says that the range of $U$
is~$\Cb^0(S)$. The set~$\Cb^0(S)$ is closed in $\Cb (S)$. To see
this, take $g_n \in \Cb^0 (S)$ with $\lim_{n \to \infty} g_n = g\in \Cb (S)$; then
$\int_S g (u) \pii ( \ud u)  = \lim_{n \to \infty} \int_S g_n (u) \pii (\ud u) = 0$, by the bounded
convergence theorem. Since~$U$ has a closed range,
there exists a constant $K < \infty$ such that
for every $g \in \Cb^0(S)$, we can find $h \in \Cb (S)$ with $ U h = g$ and $\| h \| \leq K \| g \|$~\cite[p.~487]{ds}. 
But if $F (g) = f \in \Cb^0(S)$, then $U f = g$ and since solutions to~\eqref{eq:fredholm1} are related
by additive constants, we must have $f = h - c$ where $c = \int_S h(u) \pii ( \ud u)$.
Hence
\[  \| F (g) \| = \| h - c \| \leq 2 \| h \| \leq 2 K \| g \|, ~ \text{for all}~ g \in \Cb^0 (S) .\]
Thus $F$~is bounded, and hence continuous~\cite[p.~127]{ka}.
\end{proof}

We will also use the following simple continuity result.

\begin{lemma}
\label{lem:continuity-against-K}
Let $(S,d_S)$ be a compact metric space. 
Suppose that $\cL : (S,d_S) \to (\Meass (S), d_{\tv} )$ is continuous,
and that $g_u \in \Mb (S)$, $u \in S$,
is a family of functions with $u \mapsto g_u$ continuous. For $u \in S$, define 
$G ( u ) = \int_S \cL (u, \ud v) g_u (v)$. Then $G \in \Cb (S)$.
\end{lemma}
\begin{proof}
Since $u \mapsto \| g_u \|$ is continuous and $S$ is compact, $\sup_{u \in S} \| g_u \| < \infty$.
 Similarly, since $u \mapsto \| \cL ( u, \, \cdot \, ) \|_\tv$
is continuous, $\sup_{u \in S}  \| \cL ( u, \, \cdot \, ) \|_\tv < \infty$ also. 
Hence $G$ is bounded. 
Define $L_{u,u'} \in \Meass (S)$ by
$L_{u,u'} (B) := \cL (u, B) - \cL (u', B)$, $B \in \cB(S)$. Then 
\begin{align*}
  | G (u) - G (u') | & =   \left| \int_S L_{u,u'} ( \ud v) g_u (v) 
- \int_S \cL (u',  \ud v) ( g_{u'} ( v) - g_u (v) ) \right| \\
& \leq \| g_u \| \cdot \| L_{u,u'} \|_\tv + \| \cL (u', \, \cdot \, ) \|_\tv \cdot \| g_{u'} - g_u \| ,
\end{align*}
which tends to $0$ as $d_S (u,u') \to 0$, since both $\|  L_{u,u'} \|_\tv \to 0$ and $\| g_{u'} - g_u \| \to 0$.
\end{proof}

We conclude this section with a more explicit description of the function $F$
from Proposition~\ref{prop:fredholm}, under an additional uniform convergence assumption
on $\cK^n$,
the $n$-fold convolution of $\cK$.
Related results can be found in~\cite[pp.~57--63]{orey}. 

\begin{proposition}
\label{prop:potential}
Suppose that~\eqref{ass:kernel} and \eqref{eq:strips-convergence} hold. Let $g \in \Cb^0 (S)$. Then
$f = F (g) \in \Cb^0 (S)$ defined in Proposition~\ref{prop:fredholm} has the representation
	\begin{equation}
	\label{eq:psi-series-g}
	f ( v ) = \sum_{n=0}^\infty \int_S \cK^n ( u , \ud v ) g(v) .
        \end{equation}
\end{proposition}
\begin{remarks}
\phantomsection
\label{rmks:potential} 
\begin{remenumi}
\item
\label{rmks:potential-a} 
As in Remark~\ref{rmks:psi}\ref{rmks:psi-a},  note that by~\eqref{eq:operator-def}, $\int_S \cK^n ( u , \ud v ) g(v) = T^n_\cK g(v)$ and~\eqref{eq:psi-series-g} is equivalent to~$f = \sum_{n=0}^\infty T_\cK^n g$.
\item
\label{rmks:potential-b}
We emphasize that while the series on the right-hand
side of~\eqref{eq:psi-series-g} converges for $g \in \Cb^0(S)$, 
it does not, in general, make sense to interchange the integral and the sum, since
 the measures $H ( u, \, \cdot \, ) := \sum_{n=0}^\infty \cK^n (u, \, \cdot \,)$ will typically be trivial
in our setting; here $H$ is the \emph{potential kernel} of $\cK$~\cite[p.~41]{revuz}.
\end{remenumi}
\end{remarks}

\begin{proof}[Proof of Proposition~\ref{prop:potential}.]
Suppose that $g \in \Cb^0(S)$. 
For $n \in \ZP$, define
\[
  f_n (u) := \sum_{k=0}^n \int_S \cK^k ( u, \ud v ) g(v) =  \sum_{k=0}^n \int_S \left[\cK^k ( u, \ud v ) -  \pii (\ud v) \right]  g(v) .
\]
Recall that
$\cK^0 (u, B) = \1 { u \in B}$, so that $f_0 = g$. 
Note that $\| \cK^k ( u, \, \cdot\,) - \cK^k ( u', \, \cdot\,) \|_\tv$
is non-increasing in~$k$ (see e.g.~Lemma D.2.10 of~\cite[p.~634]{dmps}),
so~\eqref{ass:kernel}\ref{ass:kernel-ii} implies that, for every $k \in \N$, 
$u \mapsto \cK^k ( u , \, \cdot \, )$ is continuous from
$(S, d_S)$ to $(\Meass(S),
d_{\tv})$, and hence so is 
$u \mapsto \sum_{k=1}^n \cK^k ( u , \, \cdot \, )$. 
Lemma~\ref{lem:continuity-against-K} then shows that
$f_n - f_0 = f_n - g \in \Cb(S)$, and hence $f_n \in \Cb(S)$. Moreover, $\int_S f_n (u) \pii (\ud u) = 0$
by~\eqref{ass:kernel}\ref{ass:kernel-i}, so $f_n \in \Cb^0(S)$ for all $n \in \ZP$. Furthermore, 
\begin{align*}
 \int_S  f_n (v) \cK (u , \ud v) & = \sum_{k=0}^n \int_S \int_S \cK (u , \ud v) \cK^k (v, \ud w )  g(w) \\
 & = \sum_{k=0}^n  \int_S \cK^{k+1} (u, \ud w )  g(w) = f_{n+1} (u) - g(u).
\end{align*}
Thus
\begin{equation}
\label{eq:f-n-g-n}
\int_S \left( f_n (u) - f_n (v) \right) \cK (u , \ud v) = g (u) - \int_S \cK^{n+1} ( u , \ud v) g(v) =: g_n (u) .
\end{equation}
Note that~\eqref{eq:f-n-g-n} is equivalent to $f_n - T_\cK f_n = g_n$. Also note 
that $g_n \in \Cb(S)$ (by Lemma~\ref{lem:continuity-against-K}) and
$\int_S g_n (u) \pii (\ud u) = 0$, so $g_n \in \Cb^0(S)$ for all $n \in \ZP$. 
By assumption~\eqref{eq:strips-convergence}, $g_n \to g$ in $\Cb(S)$ as $n \to \infty$.
In particular, $\sup_n \| g_n \| < \infty$.

By uniqueness of solutions to~\eqref{eq:f-n-g-n} over $\Cb^0(S)$, we have that $f_n = F ( g_n)$ where $F$
is the continuous functional from 
Proposition~\ref{prop:fredholm}. 
Since $F$ is continuous, it is bounded, so $\sup_n \| f_n \| \leq C \sup_n \| g_n \| < \infty$.
Next, we have that
\begin{align*}
f_{n+1} (u) - f_{n+1} (u') & =  g(u) - g(u') + \sum_{k=1}^{n+1} \int_S \left[ \cK^k (u, \ud v) - \cK^k (u', \ud v)  \right] g(v)  \\
& =   g(u) - g(u') + \sum_{k=0}^{n} \int_S \int_S \left[ \cK (u, \ud w) - \cK  (u', \ud w)  \right] \cK^k ( w , \ud v)  g(v)  \\
& =   g(u) - g(u') +   \int_S \left[ \cK (u, \ud w) - \cK  (u', \ud w)  \right] f_n (w) .\end{align*}
It follows that
\begin{align*}
\sup_n \left| f_{n+1} (u) - f_{n+1} (u') \right| \leq \left|  g(u) - g(u')  \right|  + \left\|  \cK (u, \, \cdot \, ) - \cK  (u', \, \cdot \,  ) \right\|_\tv \cdot \sup_n \| f_n \| .
\end{align*}
Thus $f_n$, $n \in \ZP$ are bounded and equicontinuous, and hence relatively compact by the Arzel\`a--Ascoli theorem~\cite[p.~266]{ds}.
This means that any subsequential limit $f$ of $f_n$ is continuous, and so $f = F ( g )$ by continuity of~$F$. Hence all subsequential limits coincide,
and we have $f = \lim_{n \to \infty} f_n = F(g) \in \Cb^0 (S)$, as claimed.
\end{proof}

\section{Semimartingale criteria}
\label{sec:semimartingale}

We obtain our recurrence classification using some semimartingale criteria, related to those presented in~\cite[Ch.~3]{MenPopWad16}, which apply to 
discrete-time adapted processes on $\RP$ without any irreducibility assumptions. 
We present appropriate generalizations that apply to processes on $\RP \times S$.
The following recurrence result is based on Theorem~3.5.8 of~\cite{MenPopWad16}.

\begin{lemma}
\label{lem:recurrence} 
Let $\Sigma = \RP \times S$ for a compact metric space~$S$, and
suppose that $( \xi_n , n\in \ZP)$ is a stochastic process with $\xi_n = (X_n, \eta_n) \in \Sigma$,
 adapted to a filtration $(\cF_n, n \in \ZP)$.
Let $f : \Sigma \to \RP$ be such that $\inf_{u\in S} f(x , u) \to \infty$ as $x \to \infty$.
Suppose that 
$\Exp f ( \xi_n ) < \infty$ for all $n \in \ZP$, and 
there
exists $r_0 \in \RP$ for which, for all $n \in \ZP$,
\begin{align*}
\Exp [ f (\xi_{n+1} ) - f(\xi_n) \mid \cF_n] & \leq 0, \text{ on } \{  X_n  \geq r_0 \}.
\end{align*}
Then if $\Pr ( \limsup_{n \to \infty}  X_n  = \infty ) =1$,  $\Pr ( \liminf_{n \to \infty} X_n \leq r_0 ) = 1$.
\end{lemma}
\begin{proof}
By hypothesis, $\Exp f(\xi_n) < \infty$ for all $n$. Fix $n \in \ZP$ and let $\lambda_{n} := \min \{ m \geq n :  X_m  \leq r_0 \}$
and, for some $r > r_0$, set $\sigma_n := \min \{ m \geq n :  X_m  \geq  r \}$.
Since $\limsup_{n \to \infty}  X_n  = \infty$ a.s., we have that $\sigma_n < \infty$, a.s.
Then $f(\xi_{m \wedge \lambda_n \wedge \sigma_n } )$, $m \geq n$, is a non-negative supermartingale
with $\lim_{m \to \infty} f(\xi_{m \wedge \lambda_n \wedge \sigma_n } ) = f(\xi_{\lambda_n \wedge \sigma_n } )$, a.s. 
By Fatou's lemma
and the fact that $f$ is non-negative,
\[
  \Exp f (\xi_n) \geq \Exp f(\xi_{\lambda_n \wedge \sigma_n } ) \geq \Pr ( \sigma_n < \lambda_n ) \inf_{(y,u) :  y  \geq r} f (y,u) .
\]
So
\begin{align*}
 \Pr \left( \inf_{m \geq n}  X_m  \leq r_0 \right)
& \geq \Pr (\lambda_n < \infty) \geq \Pr (\lambda_n < \sigma_n )  \geq 1 - \frac{\Exp f(\xi_n)}{\inf_{(y,u) :  y  \geq r} f (y,u) } .\end{align*}
Since $r > r_0$ was arbitrary, and $\inf_{(y,u)  :  y  \geq r} f (y,u) \to \infty$ as $r \to \infty$,
it follows that, for fixed $n \in \ZP$,
$\Pr ( \inf_{m \geq n}   X_m  \leq r_0  ) = 1$. 
Since this holds for all $n \in \ZP$, the result follows.
\end{proof}

The corresponding transience result is based on Theorem~3.5.6 of~\cite{MenPopWad16}.

\begin{lemma}
\label{lem:transience} 
Let $\Sigma = \RP \times S$ for a compact metric space~$S$, and
suppose that $( \xi_n , n\in \ZP)$ is a stochastic process with $\xi_n = (X_n, \eta_n) \in \Sigma$,
 adapted to a filtration $(\cF_n, n \in \ZP)$.
Let $f : \Sigma \to \RP$ be bounded, with $\sup_{u \in S} f(x,u) \to 0$ as $ x  \to \infty$, and $\inf_{(x,u) :  x  \leq r} f(x,u) >0$
for all $r \in \RP$. Suppose that there
exists $r_0 \in \RP$ for which, for all $n \in \ZP$,
\begin{align*}
\Exp [ f (\xi_{n+1} ) - f(\xi_n) \mid \cF_n] & \leq 0, \text{ on } \{  X_n \geq r_0 \}.\end{align*}
Then if $\Pr ( \limsup_{n \to \infty}  X_n  = \infty ) =1$,  $\Pr ( \lim_{n \to \infty}  X_n = \infty ) = 1$.
\end{lemma}
\begin{proof}
Since $f$ is bounded, $\Exp f(\xi_n) < \infty$ for all $n$. Fix $n \in \ZP$ and $r_1 \geq r_0$.
For $r \in \ZP$ let $\sigma_r := \min \{ n \in \ZP :  X_n  \geq r \}$.
Since $\Pr ( \limsup_{n \to \infty}  X_n  = \infty ) =1$, we have $\sigma_r < \infty$, a.s., for every $r \in \ZP$. 
Let $\lambda_{r} := \min \{ n \geq \sigma_r :  X_n  \leq r_1 \}$.
Then $f(\xi_{n \wedge \lambda_r} )$, $n \geq \sigma_r$, is a non-negative supermartingale,
which converges, on $\{ \lambda_r < \infty \}$, to $f(\xi_{\lambda_r} )$. By optional stopping
(e.g.~Theorem~2.3.11 of~\cite{MenPopWad16}), a.s.,
\begin{align*}
  \sup_{(x,u) :  x  \geq r} f(x,u) \geq f ( \xi_{\sigma_r}) & \geq \Exp [ f( \xi_{\lambda_r} ) \1 { \lambda_r < \infty }  \mid \cF_{\sigma_r} ] \\
  & \geq \Pr ( \lambda_r < \infty \mid \cF_{\sigma_r} ) \inf_{(x,u) :  x  \leq r_1} f(x,u)  .
\end{align*}
So 
\begin{align*}
\Pr ( \lambda_r < \infty ) \leq \frac{\sup_{(x,u) :  x  \geq r} f(x,u)}{\inf_{(x,u) :  x  \leq r_1} f(x,u) } ,\end{align*}
which tends to $0$ as $r \to \infty$, by our hypotheses on $f$.
Thus,
\[
 \Pr \left( \liminf_{n \to \infty}   X_n  \leq r_1 \right) = \Pr \left( \cap_{r \in \ZP} \left\{ \lambda_r < \infty \right\} \right)
 = \lim_{r \to \infty} \Pr ( \lambda_r < \infty ) = 0 .
\]
Since $r_1 \geq r_0$ was arbitrary, we get the result. 
\end{proof}

\section*{Acknowledgements}
\addcontentsline{toc}{section}{Acknowledgements}

 The authors gratefully acknowledge two anonymous referees,
 whose constructive comments and suggestions have led to
 significant improvements in this paper. This work was supported by the Engineering and Physical Sciences Research Council [EP/W00657X/1].

\end{document}